\documentclass{amsart}
\usepackage{amsmath,amsfonts,amssymb,amscd,amsthm,epsf}
\usepackage{pinlabel}
\newtheorem{prop}{Proposition}[section]
\newtheorem{thm}[prop]{Theorem}
\newtheorem{cor}[prop]{Corollary}
\newtheorem{sch}[prop]{Scholium}
\newtheorem{ques}[prop]{Question}
\newtheorem{quess}[prop]{Questions}
\newtheorem{prob}[prop]{Problem}
\newtheorem{probs}[prop]{Problems}

\theoremstyle{definition}
\newtheorem{de}[prop]{Definition}

\theoremstyle{remark}
\newtheorem{Remark}[prop]{Remark}             
\newtheorem{Remarks}[prop]{Remarks}             
\def\C{{\mathbb C}}

\def\Z{{\mathbb Z}}
\def\R{{\mathbb R}}
\def\Q{{\mathbb Q}}

\def\D{{\mathcal D}}
\def\S4{{\mathcal S}}
\def\P{{\mathcal P}}
\def\emb{{\hookrightarrow}}
\def\inter{\mathop{\rm int}}
\def\cl{\mathop{\rm cl}}

\def\im{\mathop{\rm Im}}
\def\id{\mathop{\rm id}}
\def\max{\mathop{\rm max}\nolimits}

\def\Diff{\mathop{\rm Diff}\nolimits}
\def\co{\colon\thinspace}

\begin{document}
\title{Topologically trivial proper 2-knots}
\author{Robert E. Gompf}
\address{The University of Texas at Austin}
\email{gompf@math.utexas.edu}
\begin{abstract} 
We study smooth, proper embeddings of noncompact surfaces in 4-manifolds, focusing on {\em exotic planes} and {\em annuli}, i.e., embeddings pairwise homeomorphic to the standard embeddings of $\R^2$ and $\R^2-\inter D^2$ in $\R^4$. We encounter two uncountable classes of exotic planes, with radically different properties. One class is simple enough that we exhibit explicit level diagrams of them without 2-handles. Diagrams from the other class seem intractable to draw, and require infinitely many 2-handles. We show that every compact surface embedded rel nonempty boundary in the 4-ball has interior pairwise homeomorphic to infinitely many smooth, proper embeddings in $\R^4$. We also see that the almost-smooth, compact, embedded surfaces produced in 4-manifolds by Freedman theory must have singularities requiring infinitely many local minima in their radial functions. We construct exotic planes with uncountable group actions injecting into the pairwise mapping class group. This work raises many questions, some of which we list.
\end{abstract}
\maketitle


\section{Introduction}\label{Intro}

Classical knot theory has spawned various other lines of research with a common theme of studying ambient isotopy classes of embeddings of manifolds. In the traditional setting the domain is compact, but the problem naturally extends to the noncompact setting if we require the embeddings to be proper. The classical case $S^1\emb \R^3$ then extends to the case of knotted embeddings into $\R^3$ of the line $\R$ and ray $[0,\infty)$. For example, it is known (perhaps counterintuitively) that knotted rays exist (Fox--Artin \cite{FA}) and realize uncountably many ambient isotopy classes (McPherson \cite{McP}). (These are each obtained from the cited references by deleting the wild endpoint of an arc in $S^3$.) In a different direction, higher-dimensional spheres in $\R^n$ have been extensively studied. For example, {\em 2-knots} of a 2-sphere into $\R^4$, as well as higher genus knotted surfaces in $\R^4$, have been receiving recent attention. However, higher-dimensional {\em proper} knots are largely terra incognita. The present paper addresses proper 2-knots of surfaces in $\R^4$, with domain usually taken to be the plane $\R^2$ or the (half-open) annulus $[0,\infty)\times S^1$. In dimension 4, the smooth and topological categories are quite different. For example, there are families of compact, nonorientable surfaces in $\R^4$ that are smoothly distinct but topologically isotopic (Finashin, Kreck and Viro \cite{Fi}, \cite{Fi2}), and in fact are topologically standard by Kreck \cite{Kr} (although no orientable examples are presently known). In this paper, we work in the smooth category, but focus on those examples that appear simplest in the topological category. That is, we study smooth ambient isotopy classes of smooth, proper embeddings that are topologically ambiently isotopic to the standard plane $\R^2\subset \R^4$ or annulus $[0,\infty)\times S^1\subset\R^4$ (the standard plane minus an open disk). We call nontrivial examples {\em exotic planes} and {\em exotic annuli}, respectively. These have been known (but not widely) since the 1980s, with exotic planes implicitly given by the author in \cite[Remark~4.2]{kink} (see Remark~\ref{original} below) and a different family observed by Freedman (previously unpublished but described in Section~\ref{Branch}) shortly thereafter. They are a uniquely 4-dimensional phenomenon (Proposition~\ref{highdim}): A self-homeomorphism of $\R^n$ that is a local diffeomorphism near the standard $\R^k$ can be assumed the identity there after a smooth ambient isotopy, except in the case $(n,k)=(4,2)$; the analogous statement for annuli is only slightly weaker. The topological simplicity of exotic planes and annuli makes them particularly subtle: All of the classical invariants, such as from the homotopy type of the complement or a branched cover, fail to distinguish them. Nevertheless, we uncover a rich structure using more subtle invariants of smooth 4-manifolds. This structure often transfers to more general proper 2-knots. For example, every oriented surface in $\R^4$ obtained as the interior of a compact surface embedded rel its nonempty boundary in $B^4$ has infinitely many exotic cousins topologically isotopic to it (Corollary~\ref{slice}). This suggests a future direction of studying smooth proper 2-knots ``modulo" exotic planes, and whether these differ from topological proper 2-knots (Questions~\ref{modplanes} and \ref{univ}). However, the present paper focuses on the exotic planes, methods of distinguishing them, their range of symmetries, and some explicit diagrams of such exotica (the simplest being Figure~\ref{plane} below).

For our first approach to constructing invariants, note that any annulus $A$ in $\R^4$ has a simply connected complement. (We henceforth assume all embeddings of positive codimension are proper and all annuli are half-open, while working up to isotopies of the ambient space.) It follows that $A$ can be extended to an immersion of $\R^2$ by adding an immersed disk $D$ with $A\cap D=\partial A=\partial D$. This can be transformed to an embedded surface of finite genus. (For example, tube away double points in pairs after adding double points of one sign as necessary.) Conversely, any immersed surface with one end, finite genus and finitely many double points determines an embedded annulus. (Remove the interior of a suitably large compact surface with a single boundary component, and notice that the resulting isotopy class is independent of the choice of such surface.) For such a surface, we will say the end is {\em annular}. 

\begin{de} The {\em minimal genus} $g(A)$ of an embedded annulus $A$ is the smallest genus of an embedded, oriented surface determining $A$. The {\em kinkiness} $\kappa(A)$ is the pair $(\kappa_+,\kappa_-)$ for which $\kappa_+$ (resp.\ $\kappa_-$) is the minimal number of positive (resp.\ negative) double points in a generically immersed $\R^2$ determining $A$.
\end{de}

\noindent Note that $\kappa_+$ and $\kappa_-$ may not be realized by the same immersion of $\R^2$. (An example with $\kappa_\pm=0$ but $g=1$ can be constructed from the figure-eight knot in $\partial B^4$.) As we will see, there are exotic annuli realizing all possible values of $\kappa$, and all possible minimal genera (Theorem~\ref{0g}(b)).

These invariants are not directly useful for an embedded $\R^2$ since they obviously vanish on the annulus it determines. However, a more useful version describes the behavior at infinity of any surface $F$ determining an annulus $A$ in $\R^4$: If we smoothly one-point compactify $\R^4$ to $S^4$ in the obvious way, then $A$ becomes an {\em almost-smooth} embedded disk $D\subset S^4$, smooth except at a unique isolated singularity occurring at the added point  $\infty$. Working in a preassigned neighborhood $V$ of $\infty$ in $S^4$, we may remove a singular disk from $D$, and replace it with either a smoothly embedded surface or an immersed disk as before. Minimizing as before gives a version of $g$ or $\kappa$. However, these numbers depend in general on the choice of $V$, nondecreasing as we reduce the size of $V$.

\begin{de} The {\em minimal genus at infinity} $g^\infty(A)=g^\infty(F)$, and {\em kinkiness at infinity} $\kappa_\pm^\infty(A)=\kappa_\pm^\infty(F)$, are given by the limit in $\Z^{\ge0}\cup\{\infty\}$ of the corresponding numbers for the pair $(D,V)$ as the neighborhood $V$ of $\infty$ becomes arbitrarily small.
\end{de}

\noindent We also call these invariants the minimal genus $g(D)$ and kinkiness $\kappa_\pm(D)$ of the singular disk (or the singularity), which is equivalent to the author's original usage for disks in \cite{kink}. In Section~\ref{CH}, we reinterpret that paper and its follow-up in \cite{MinGen} to obtain the following, in the cases with nonvanishing invariants:

\begin{thm}\label{ginfty} 
There are exotic planes in $\R^4$ realizing all values of $\kappa^\infty=(\kappa_+^\infty,\kappa_-^\infty)$, with $g^\infty=\max\{\kappa_+^\infty,\kappa_-^\infty\}$.
\end{thm}

\noindent The exceptional case with vanishing $\kappa^\infty$ and $g^\infty$ (Theorem~\ref{0g}(a) below) is proved in Section~\ref{Branch}, using a different construction that also realizes each of infinitely many values of $g^\infty$ by uncountably many exotic planes (Corollary~\ref{ginf}).

While realizing large values of these invariants gives a sense in which exotic planes can be arbitrarily complicated at infinity, we also investigate how simple they can be. An annulus $A$ in $\R^4$ has $g^\infty=0$ if and only if there is a homotopy from the corresponding singular disk in $S^4$ to a smoothly embedded disk, supported in an arbitrarily small neighborhood of the singularity in $S^4$. This implies the kinkiness at infinity also vanishes. Specializing to a proper 2-knot $\R^2\emb\R^4$ and inverting our viewpoint, we have:

\begin{de}\label{gen} We will say a proper 2-knot $F\co\R^2\emb\R^4$ is {\em generated by 2-knots} if for every compact subset $K\subset\R^4$ there is a disk containing $F^{-1}(K)$ in $\R^2$ whose image can be extended to an embedded sphere by adding a disk in the complement of $K$. If the sphere can always be chosen to be unknotted, we will say $F$ is {\em generated by unknots}.
\end{de}

\noindent Thus, $F$ is generated by 2-knots if and only if $g^\infty(F)=0$, but generation by unknots is stronger (strictly, as we see below). For comparison, note that the definitions generalize to any proper embedding between Euclidean spaces. It is clear that every proper knot $\R\emb\R^3$ is generated by knots. (Remove the ends and connect the resulting endpoints by an arc near infinity.) This contrasts with the above exotic planes that have $g^\infty\ne0$ so are not generated by 2-knots. Some proper knots $\R\emb\R^3$ are generated by unknots. (Thicken any knotted ray $\gamma$ to an embedded $I\times[0,\infty)$. The resulting boundary line is generated by unknots of the form $\partial(I\times[0,t])$, but it is still knotted since each end is isotopic to $\gamma$.) In contrast, infinite connected sums, for example, are not generated by unknots, since any truncation near infinity will give a nontrivial connected sum. Similarly, knotted planes that are generated by 2-knots but not unknots can be constructed by summing the standard plane with an infinite sequence of 2-knots. However, it is not clear whether exotic planes can have this behavior (Questions~\ref{gen2knots}).

\begin{thm}\label{0g}
(a) There is an uncountable collection of exotic planes that are generated by unknots, so $g^\infty=0$, determining pairwise nonisotopic exotic annuli.

(b) For each value of $\kappa\in\Z^{\ge0}\times\Z^{\ge0}$, there are uncountably many exotic annuli realizing this value with $g=\max\{\kappa_\pm\}$ and $g^\infty=0$.
\end{thm}

\noindent We construct and distinguish these examples in, respectively, Section~\ref{planeFamilies} (proof of Theorem~\ref{order}) and Section~\ref{CH} (Corollary~\ref{annuli}). The proof is completed in Section~\ref{DrawPlanes}, with the behavior at infinity established by explicitly drawing the surfaces. We will see that the exotic planes in (a) are simpler than our other exotic planes in many ways (summarized in Section~\ref{Questions}). For clarity of exposition, we will refer to such examples as {\em simple}, although it is not presently clear which properties should be singled out for a formal definition.

Since the simple exotic planes in (a) have $g^\infty=\kappa^\infty_{\pm}=0$, we need a new invariant to distinguish them. In classical and other versions of knot theory, the homotopy type of the double branched cover of the knot provides important information. Since the double branched cover of an exotic plane is homeomorphic to $\R^4$, it provides no homotopy-theoretic invariants. However, an unpublished example of Freedman (later expanded by the author and exhibited in \cite{menag}) showed that such a branched cover need not be diffeomorphic to $\R^4$. Its diffeomorphism type can then be used as an invariant. The well-developed theory of exotic {\em $\R^4$-homeomorphs} (oriented diffeomorphism types homeomorphic to $\R^4$) can now be used to establish a theory of such exotic planes, obtaining the family in (a) and further results in this paper. For example, Section~\ref{Group} exhibits exotic planes (both simple and otherwise) with various discrete group actions, some uncountable, that inject into the pairwise mapping class group. Other exotic planes $P$ have large group actions near the end, whose nontrivial elements cannot extend over the entire pair $(\R^4,P)$. For some of the global actions, each compact subset of $\R^4$ has infinitely many pairwise disjoint images. In contrast (Theorem~\ref{sum}) there is an exotic plane $P'$ and a compact subset $K$ of $\R^4$ such that no pairwise diffeomorphism of $(\R^4,P')$ sends $K$ into $\R^4-K$.

We attempt to organize the set of proper 2-knots with a relation: For two such knots $F_1$ and $F_2$, we write $F_1\le F_2$ if there is a (nonproper) embedding of $\R^4$ into itself sending $F_1$ onto $F_2$. We call $F_1$ and $F_2$ {\em equivalent} if $F_1\le F_2\le F_1$, obtaining equivalence classes comprising a partially ordered set. Perhaps surprisingly, this yields rich structure. Let $\Sigma\subset I$ be obtained from the standard Cantor set by removing the upper endpoint of each of the deleted middle thirds. Then $\Sigma$ has the cardinality of the continuum. Partially order $\Sigma\times\Sigma$ so that $(s_1,s_2)\le(t_1,t_2)$ means $s_i\le t_i$ for each $i$. In Section~\ref{planeFamilies} we prove:

\begin{thm}\label{order}
(a) The exotic planes of Theorem~\ref{0g}(a) are all equivalent to the standard plane.

(b) There is an uncountable set of equivalence classes of exotic planes with the order type of $\Sigma\times\Sigma$.

(c) There is an uncountable set of equivalence classes of exotic planes with the order type of $\Sigma$ such that each class has uncountably many distinct elements.
\end{thm}

The simple exotic planes presented in (a) of the two previous theorems seem quite different from the other planes of Theorems~\ref{ginfty} and \ref{order}, simpler in ways besides their vanishing $g^\infty$ and equivalence to the standard plane. For example, they are simple enough that we can draw them explicitly. In Section~\ref{Diagrams}, which can mostly be read after Sections~\ref{Satellites}--\ref{Casson} and \ref{Branch1} (the latter needed for planes but not annuli), we draw them as level diagrams (movies) using a proper Morse function given by distance to a generic point. For Theorem~\ref{0g}(a) and each choice of $\kappa$ in (b), we explicitly draw such an example, and describe the other members of the uncountable family up to unspecified ramification.  In each case, we obtain a ribbon surface, with local minima successively appearing as the radius function increases, and each eventually being connected to the rest by a ribbon (saddle point). Notably, we do not need any local maxima. The only difference between these diagrams and the more familiar diagrams of compact ribbon surfaces is that in our case, the process never terminates. (If it did, the topologically standard annulus would be smoothly standard, cf.~text preceding Definition~\ref{0framing}, hence it would be unique and with $g=0$.) Our simplest example is Figure~\ref{plane}. To interpret the figure, consider the recursively defined tangles $\alpha_n$, $n=0,1,2,\dots$, for which the thick curves represent bunches of $2^{n+1}$ strands as indicated, parallel in the plane of the paper (except where subject to the two indicated full left twists). Thus, $\alpha_n$ has $2^{n+2}$ strands exiting the top of the box, and the same number exit the bottom. The lower diagrams show how the plane intersects 3-spheres of radius $r\in\Z^+$. The first diagram with $n=0$ shows a ribbon link in the sphere of radius 1. One component is obtained by connecting the top and bottom of $\alpha_0$ by four arcs  in the manner of a braid completion, while threading through a 4-component unlink. There is an obvious ribbon move between the two outermost arcs of the completion, allowing us to interpret the diagram as six 0-handles (local minima) at some radius less than 1, with one pair connected by a 1-handle appearing at radius 1. Then at radius 2 (middle diagram), four more 1-handles connect the knot to the four linking circles while threading through the boundary circles of eight more 0-handles. The resulting diagram has the same form as the first, with $n$ incremented by 1 (right diagram). The diagrams continue recursively. The fact that this is an exotic plane arises from a general construction in Sections~\ref{Branch} and~\ref{Diagrams}. However, as a check, we also show directly that it is topologically standard (Remark~\ref{pi1}(b)). The fact that the plane is exotic is more difficult to prove, but perhaps more believable.

\begin{figure}
\labellist
\small\hair 2pt
\pinlabel {$\alpha_{n+1}$} at 255 49
\pinlabel {$\alpha_n$} at 105 49
\pinlabel {$\alpha_n$} at 10 49
\pinlabel {$\alpha_0=$} at 37 117
\pinlabel {$\alpha_n=$} at 165 117
\pinlabel {$\alpha_{n-1}$} at 208 132
\pinlabel {$-2$} at 84 111
\pinlabel {$-2$} at 218 103
\pinlabel {$-2$} at 152 8
\pinlabel {$2^{n+2}$ strands} at 68 22
\pinlabel {$2^{n+1}$ strands} at 280 117
\pinlabel {$2^{n+2}$ strands} at 217 22
\pinlabel {$2^{n+3}$ strands} at 312 22
\pinlabel {$\approx$} at 230 40
\pinlabel {$r=n+1$} at 20 -5
\pinlabel {$r=n+2$} at 230 -5
\endlabellist
\centering
\includegraphics{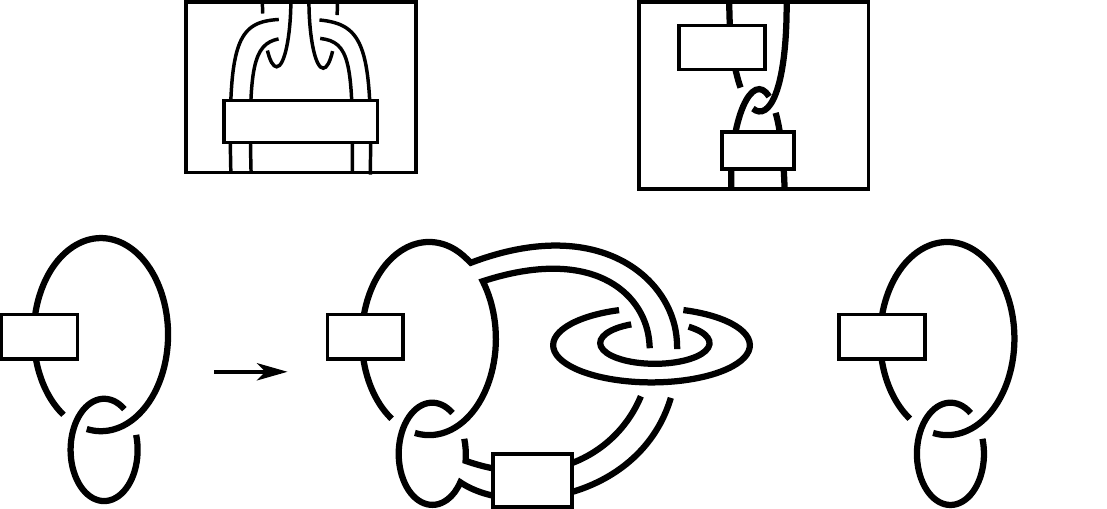}
\caption{An exotic plane from Theorem~\ref{0g}(a), as an infinite, recursive level diagram. The thick curves represent bunches of the indicated numbers of parallel strands.}
\label{plane}
\end{figure}
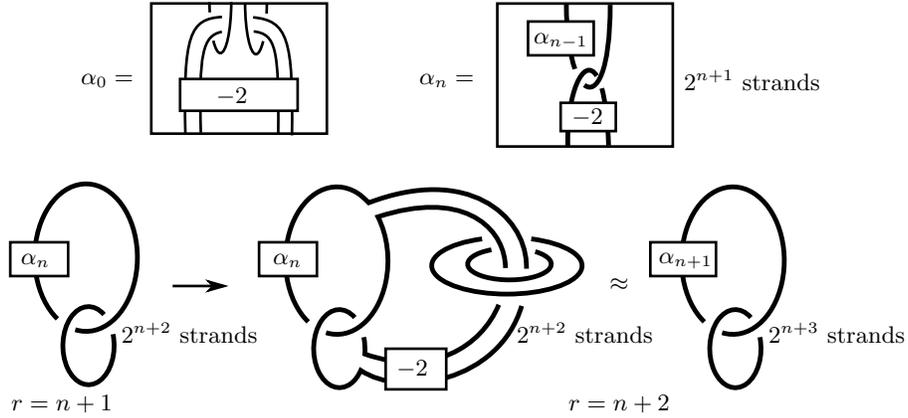

Unlike the examples of Theorem~\ref{0g}, the other examples of Theorems~\ref{ginfty} and \ref{order} are created by an infinite process with poorly controlled, superexponentially growing complexity, so extracting an explicit diagram seems intractable. (More specifically, the infinite nature of the constructions of Theorem~\ref{0g} comes from Casson handles, which can be described concretely. The other constructions involve topologically embedded surfaces, which ultimately arise from intersections of complicated infinite nestings of Casson handles; see Section~\ref{AlmostSmooth}.) In contrast to the previous paragraph, any level diagrams of these more complicated exotic planes would require infinitely many local maxima (Scholium~\ref{ginftySch} and Corollary~\ref{inf2h}, respectively). In some cases, the number of components of the superlevel sets $r^{-1}[a,\infty)$ must become arbitrarily large as the radius $a$ increases (Scholium~\ref{infEnd}), whereas the superlevel sets in the previous paragraph are connected.

We can invert our viewpoint to get a discussion of isolated singularities of embedded surfaces. A smooth annulus in $\R^4$ is topologically standard if and only if the disk made by compactifying at infinity is {\em locally flat}, so locally it is pairwise homeomorphic to a smoothly embedded surface. The annulus is smoothly standard if and only if the singularity is trivial under an equivalence relation that we call {\em almost-smooth isotopy}, topological (ambient) isotopy that is smooth except at the singular point. This relation preserves $g$ and $\kappa$ of the singularity. Any isolated singularity in an otherwise smooth surface in a 4-manifold locally admits a level diagram from the radius function at the singularity, which we can assume is Morse elsewhere on the surface. This diagram is obtained from one of the corresponding annulus by inverting the radial coordinate, so it typically fails to terminate with decreasing radius. Such diagrams can be varied in the usual way by almost-smooth isotopy. We use this viewpoint to study the structure of singularities of compact surfaces: A known corollary of Freedman \cite{F} and Quinn \cite{Q} (see Corollary~\ref{cpctSurf} below) is that locally flat (topologically embedded) surfaces can always be topologically isotoped to be smooth except at a point. The minimum genus and kinkiness of such singularities were addressed by the author in \cite[Theorems~6.2 and 8.4]{MinGen}. The present paper further elucidates the complexity of such singularities at the end of Section~\ref{Taylor}:

\begin{thm}\label{sing} Every compact, locally flat surface $F$ in a smooth 4-manifold is topologically (ambiently) isotopic to a surface $F'$ that is smooth except at a singular point $p$ at which every level diagram requires infinitely many local minima. This singularity can be chosen so that
\begin{itemize}
\item[a)] $g=\max\{\kappa_+,\kappa_-\}$, realizing any preassigned, sufficiently large $\kappa_\pm$ (finite or infinite), or so that
\item[b)] $g$ is infinite, and for each $m\in\Z^+$ there is a neighborhood $U$ of $p$ such that every integral homology 4-ball $B$ with $p\in\inter B\subset U$ and $\partial B$ transverse to $F'$ intersects $F'$ in at least $m$ components.
\end{itemize}
If $F$ is smooth, the singularity of $F'$ can be chosen to realize any nonzero $\kappa$ in (a), or alternatively, to have no local minima and $g=\kappa_\pm=0$ but still not be almost-smoothly isotopic to a smooth surface.
\end{thm}

The proof (Section~\ref{Taylor}) actually shows that the almost-smooth surfaces arising from Freedman's construction (with sufficient ramification) always require infinitely many local minima. It proceeds by immediately smoothing $F$ near its 1-skeleton and reducing to the case of the core disk of a Casson handle (whose definition we review in Section~\ref{Casson}). It follows that the theorem applies more generally to smoothing embedded 2-complexes, creating a singularity as above on each 2-cell. The cases with $g\ne0$ can alternatively be proved using generalized Casson handles with embedded surface stages, the technology needed by the author for \cite{steintop}. Corollary~6.1 of that paper showed that any 2-complex tamely topologically embedded in a complex surface is topologically isotopic to one with an uncountable system of Stein neighborhoods (so a finite complex becomes a ``Stein compact"). It follows that the resulting 2-cells typically must have singularities requiring infinitely many local minima (and can be chosen with $g$ and $\kappa_\pm$ arbitrarily large, although the condition on homology balls does not follow in the Stein setting).

This paper is organized as follows: After discussing our basic tools in Section~\ref{Tools}, we prove most of Theorem~\ref{ginfty} in Section~\ref{Context}, exhibiting exotic planes with all nonzero values of $\kappa^\infty$ (and thereby $g^\infty$). This leads into a summary of necessary background from exotic $\R^4$ theory. For context, we also show nonexistence of exotic linear spaces in other dimensions, as well as considering exotic annuli, and briefly discuss the dual problem of smoothing topological submanifolds of 4-manifolds. Section~\ref{Branch} studies exotic planes and other surfaces by the diffeomorphism types of their double branched covers. We obtain uncountably many exotic planes with $g^\infty=0$ as well as with arbitrarily large (finite or infinite) $g^\infty$. Understanding the resulting ends allows us to prove Theorem~\ref{sing} on singularities of almost-smooth surfaces. We also discuss exotic planes with many symmetries (Section~\ref{Group}). In Section~\ref{Diagrams}, we draw explicit exotic annuli and planes, and exhibit some symmetries. Finally, we summarize the behavior of our two types of exotic planes and discuss some open questions (Section~\ref{Questions}). Throughout the text, we work in the setting of oriented, connected, smooth manifolds, except where otherwise specified. Embeddings with positive codimension (only) are assumed to be proper, and in the topological category they are locally flat. Isotopies are implicitly ambient, i.e., we compose the embedding with an isotopy of the ambient space through diffeomorphisms (or homeomorphisms in the topological category). Since all orientation-preserving self-diffeomorphisms of $\R^2$ and $[0,\infty)\times S^1$ are isotopic to the identity, we often abuse notation by conflating embeddings of these spaces with their images. Similarly, pairwise diffeomorphism and isotopy are equivalent for embeddings in $\R^n$, and analogously in the topological category.  (To isotope a self-diffeomorphism of $\R^n$ to the identity, first arrange this to first order at 0, then conjugate by a dilation. A similar procedure works in the topological category by first applying the Stable Homeomorphism Theorem to make it the identity near 0 (Kirby--Siebenmann \cite{KS} for $n\ge5$ and Quinn \cite[2.2.2]{Q} for $n=4$).) The symbol ``$\approx$" denotes diffeomorphism (sometimes pairwise).

The author would like to acknowledge the 2019 BIRS  2-knots conference 19w5118, which planted the seed for this paper.

 
\section{Basic tools}\label{Tools}

We begin by assembling some basic tools for proper knots, beginning with a way to distinguish annuli by using them to enlarge the ambient manifold. We then discuss end sums, satellites, Casson handles and isotoping topologically embedded surfaces to become almost-smooth.

\subsection{Distinguishing annuli}\label{Annuli}
One way to distinguish annuli is the following:
\begin{prop}\label{rebuild}
a) For any smooth $n$-manifold $X$ and $k\le n$, there is a canonical bijection between isotopy classes of normally framed annuli $[0,\infty)\times S^{k-1}\emb X$ and manifolds $\hat{X}$ containing $X$ as the complement of a distinguished boundary component identified as $S^{k-1}\times\R^{n-k}$ (up to diffeomorphisms with restriction to $X$ isotopic to the identity).

\noindent b) There is a canonical map from isotopy classes of such framed annuli to manifolds (up to diffeomorphism) obtained by attaching an open $k$-handle to $X$ at infinity as defined below.
\end{prop}

\begin{proof}
In (a), the distinguished boundary component of $\hat{X}$ has a tubular neighborhood $(-1,\infty]\times S^{k-1}\times\R^{n-k}$. The required framed annulus in $X$ is given by $[0,\infty)\times S^{k-1}\times\{0\}$. For the reverse correspondence, glue such a tubular neighborhood of a boundary component onto $X$ using the obvious identification of its interior with a neighborhood of the annulus in $X$. These correspondences are easily seen to be well-defined inverses up to the given equivalences. We can now add a $k$-handle at infinity for (b): Use the new boundary of $\hat{X}$ with the given framing to attach an open $k$-handle $D^k\times\R^{n-k}$, or equivalently, identify a tubular neighborhood of the attaching region in the handle with the framed neighborhood of the annulus in $X$.
\end{proof}

Handles at infinity, which were used by the author in \cite{MfdQuot}, are more general than would be expected by considering interiors of compact handlebodies. This is because $\hat{X}$ typically cannot be compactifed by adding more boundary, as the fundamental group behavior of $X$ at infinity sometimes shows. For example, if $\R^4$ is exhibited as the interior of a compact manifold, the boundary must be simply connected and hence diffeomorphic to $S^3$. If this contains the boundary of $\hat{X}$ for some topologically standard annulus in $\R^4$, the resulting circle in $S^3$ must have knot group $\Z$, so it bounds a disk $D$ in $S^3$. Then the annulus lies in the boundary of $[0,\infty)\times D$ in $\R^4$, so it is smoothly standard. In particular, exotic annuli as in this paper can never arise as interiors of compact pairs. However, we can still canonically keep track of framings as we do for knots in $S^3$:

\begin{de}\label{0framing}
The {\em 0-framing} of an annulus in $\R^4$ is the unique normal framing for which attaching a 2-handle at infinity gives a manifold with vanishing intersection pairing. Equivalently, it is the unique framing that extends over any embedded surface generating the annulus.
\end{de}

\subsection{End sums}\label{Sums}
The {\em end sum} operation consists of connecting two manifolds by a 1-handle at infinity, so that their orientations agree. This was analyzed in detail by Calcut and the author \cite{CG} (expanding on the author's earlier work in \cite{infR4}): The operation is well-defined on diffeomorphism types in dimensions $n\ge 4$ when (for example) the ends are simply connected, by uniqueness of the defining rays up to isotopy. (In contrast, summing a pair of one-ended manifolds with complicated fundamental group structure at infinity can even result in uncountably many diffeomorphism types, as shown by Calcut, Guilbault and \ Haggerty \cite{CGH}.) There is a natural identification of the end sum $X^n\natural\R^n$ with $X^n$ that is the identity outside a neigborhood of the ray in $X$. This extends the operation to sums of countably infinite collections, where we sum each onto $\R^n$ using an infinite collection of disjoint rays in the latter. It is then independent of the order of the summands and grouping -- commutativity and associativity in the infinite setting. We now turn this into an operation for studying proper 2-knots that is analogous to the connected sum of classical knots. For any two embedded noncompact surfaces $F_i\subset X^4_i$, we can choose a ray in each $F_i$ and perform the sum pairwise, respecting all orientations, to get a new pair $(X_1\natural X_2,F_1\natural F_2)$. This is well-defined on diffeomorphism types of pairs whenever the end of each $F_i$ annular (i.e.~$F_i$ has one end and finite genus), since the rays are then unique up to pairwise isotopy. (Without annularity, there are examples with all relevant manifolds one-ended, but the resulting 4-manifolds nonunique. For example, the rays used in \cite{CGH} can be assumed to lie on a surface of infinite genus. See Questions~\ref{uniqueSum} for related issues.) In general, we should not expect the isotopy class of $F_1\natural F_2\subset X_1\natural X_2$ to be uniquely determined by the isotopy classes of the summands. (Already in the simpler setting of pairwise connected sums of circles in tori, the result changes under $2\pi$-rotation of the disk.) But when $X_2$ is $\R^4$ and the end of each $F_i$ is annular, the end sum $F_1\natural F_2\subset X_1\natural\R^4=X_1$ is well-defined on isotopy classes since it just inserts $F_2$ into $F_1$ near the isotopically unique ray. We can form a countable sum $\natural_{i=1}^N (X_i,F_i)$ of pairs for any $N\in\{0,1,2,\dots,\infty\}$ by summing each into $(\R^4,\R^2)$, along a collection of $N$ disjoint rays in $\R^2$ indexed by positive integers. (Then $N=0$ returns $(\R^4,\R^2)$.) In Section~\ref{Group} it will be useful to allow collections of rays that densely fill regions, such as $[0,\infty)\times\Q\subset\R\times\R$.

\begin{prop}\label{infSum}
The end sum $\natural_{i=1}^N (X_i,F_i)$ of noncompact surfaces $F_i\subset X^4_i$ can be defined using any collection of $N$ disjoint rays $\gamma_i$ in $\R^2$. If the end of each $F_i$ is annular, the diffeomorphism type of the sum depends only on the diffeomorphism types of the pairs $(X_i,F_i)$. In particular, it is independent of the order. Iterated sums of such pairs $(X_i,F_i)$ can equivalently be performed simultaneously if the resulting genus is finite.
\end{prop}

\begin{proof}
Truncate each ray $\gamma_i$ so that its distance to the origin is at least $i$. The union of the rays is then a 1-manifold in $\R^2$. After pairwise isotopy of $(\R^4,\R^2)$, we can further assume the rays are radial, with $\gamma_i$ beginning at radius $i$. It is then easy to define the sum: Find disjoint, pairwise, tubular neighborhoods $\nu_i$ of the rays and identify each $\nu_i$ with $[0,\frac12)\times(\R^3,\R)$ in a copy of $[0,1]\times(\R^3,\R)$, then identify $(\frac12,1]\times(\R^3,\R)$ with a neighborhod of a ray  in $F_i\subset X_i$. If the end of each $F_i$ is annular, these latter rays are unique up to pairwise isotopy, and the resulting sum is easily seen to be independent of all choices except perhaps the order of the terms.

For independence of order, suppose we have another such collection of radial rays $\gamma'_i$. After rotating $(\R^4,\R^2)$, we can assume $\gamma'_1=\gamma_1$. Form the sum $(\R^4,\R^2)\natural(X_1,F_1)=\natural_{i=1}^1 (X_i,F_i)$. Since the end of $F_1$ is annular, there is an annulus $A_1$ given by a proper embedding $[1,\infty)\times S^1\to\R^2\natural F_1$ that agrees with polar coordinates on $\R^2-\nu_1$. After isotopy in a neighborhood of $A_1$ in $\R^4$, preserving the first coordinate of $A_1$, we can assume $\gamma'_2=\gamma_2$. Since the end of $F_2$ is annular, the sum $\natural_{i=1}^2 (X_i,F_i)$ now contains an annulus $A_2$ agreeing with $A_1$ on $[2,\infty)\times S^1$ along $\R^2\natural F_1-\nu_2$. Continuing by induction, we prove independence of order for all $N<\infty$. (The annulus $A_n$ allows $F_{n+1}$ to jump over previous summands as necessary to obtain the required order around $\R^2$.) Since every point has a neighborhood on which all but finitely many of these diffeomorphisms agree, there is a well-defined limiting local diffeomorphism for $N=\infty$ that is easily seen to be bijective.

We can iterate the operation of summing collections as above, possibly infinitely. If the end of each original surface is annular and each partial sum has finite genus, then the partial sums inherit annular ends, and we can assume each required ray for subsequent sums lies in a central $\R^2$.  The final sum then contains multiple central copies of $(\R^4,\R^2)$ end summed according to some tree. The sum of these copies is again diffeomorphic to $(\R^4,\R^2)$, since it can be written as a nested union of standard ball pairs $(B^4,B^2)$. Thus, the original iterated sum is diffeomorphic to a single sum. Since $(\R^4,\R^2)$ is the identity element, an end sum as above with finite $N$ is then diffeomorphic to the corresponding iterated 2-fold sum.
\end{proof}

The set of diffeomorphism types of pairs $(X,F)$ such that the end of $F$ is annular forms a commutative monoid under end sum, with various submonoids such as the (genus-0) proper 2-knots in $\R^4$ and the topologically standard proper 2-knots. Since these two submonoids are closed under infinite sums (although the original monoid is not), they are far from being a group. First, the Eilenberg Swindle (Mazur Trick) shows there are no inverses: If an embedded plane $F\subset\R^4$ has an inverse $F^{-1}$, then $F\approx F\natural\R^2\natural\R^2\natural\dots\approx F\natural(F^{-1}\natural F)\natural(F^{-1}\natural F)\natural\dots\approx (F\natural F^{-1})\natural (F\natural F^{-1})\natural\dots\approx \R^2$ (where $\R^2$ denotes the standard plane in $\R^4$ and the third diffeomorphism is by associativity). Secondly, every homomorphism $\varphi$ from either of these monoids to a group is trivial: For all such surfaces $F$, we have $\varphi(\natural_\infty F)=\varphi(F\natural(\natural_\infty F))=\varphi(F)\varphi(\natural_\infty F)$, so $\varphi(F)$ is the identity. Thus, there can be no useful analog of the knot concordance group for proper 2-knots or exotic planes. However, the double branched cover of a pairwise end sum is the end sum of the corresponding double branched covers. In particular, we obtain a homomorphism from the monoid of topologically standard planes in $\R^4$ to the monoid of diffeomorphism types homeomorphic to $\R^4$, respecting infinite end sums. Infinite end sums are compatible in the obvious way with the equivalence and partial ordering arising for Theorem~\ref{order}, and the ordering corresponds to inclusion of double branched covers. (For a well-defined ordered monoid structure on $\R^4$-homeomorphs, one should descend further to ``compact equivalence" classes defined by setting $R_1\le R_2$ if every compact subset of $R_1$ embeds in $R_2$.)

\subsection{Satellites}\label{Satellites}
Another useful tool is the {\em satellite} construction. Classically, we start with a {\em pattern} $\P$, which is a knot in a solid torus $T=S^1\times D^2$. The corresponding satellite operator replaces a {\em companion} knot $K \subset S^3$ by the satellite knot $\P(K) \subset S^3$ obtained from $\P \subset T$ by identifying $T$ with a tubular neighborhood of $K$ so that the product framing of the core of $T$ corresponds to the 0-framing of $K$ (and all orientations are preserved). For example, if $\P$ is given by the dotted circle in Figure~\ref{double}(a), where $T$ is the complement of the lower circle in $S^3$, identified so that the circles $S^1\times\{p\}$ in $T$ are unlinked from each other in the diagram, $\P(K)$ is called  the positive (untwisted) {\em Whitehead double} $DK$ of $K$. The negative Whitehead double is obtained from the mirror image of $\P$. The result of doubling three meridians of an unknot, one negatively, is shown in (b). The satellite construction has various generalizations to higher dimensions. The most well-known is to take the product of a classical pattern with $I$ and insert it into a tubular neighborhood of a compact annulus embedded rel boundary in a 4-manifold. This shows, for example, that if $K_0$ and $K_1$ are concordant (the boundary components of an annulus in $I\times S^3$ with $K_i\subset\{i\}\times S^3$) then so are $\P(K_0)$ and $\P(K_1)$. This notion immediately generalizes to half-open, proper annuli. For such an annulus in $\R^4$, we canonically identify a tubular neighborhood as a product using the 0-framing of Definition~\ref{0framing}. Aside from doubling annuli in this manner, we will take satellites with other pattern and companion surfaces. An important example is doubling disks. Note that the dotted circle in Figure~\ref{double}(a) is unknotted in $S^3$, so it bounds an unknotted disk $\P$ in the 4-ball $B$ whose boundary is shown. This disk $\P$ can be visualized in the figure as a pair of parallel disks connected by a twisted band (and with interior pushed into $\inter B$). We take this disk as the pattern for doubling a disk $\Delta$ embedded rel boundary in $D^4$, by identifying a tubular neighborhood $N$ of $\Delta$ with $B$ so that $N\cap\partial D^4$ corresponds to $T$ (necessarily inducing the 0-framing). Now (b) of the figure, interpreted 4-dimensionally, shows the result of doubling three normal disks to the unknotted disk in $D^4$ whose boundary is the lower circle. The disks are more easily seen after an isotopy producing (c) (using the fact that the Whitehead link in Figure~\ref{double}(a) is {\em symmetric}, i.e., there is an isotopy interchanging its components).

\begin{figure}
\labellist
\small\hair 2pt
\pinlabel {(a)} at -2 91
\pinlabel {(b)} at 126 91
\pinlabel {(c)} at 39 0
\endlabellist
\centering
\includegraphics{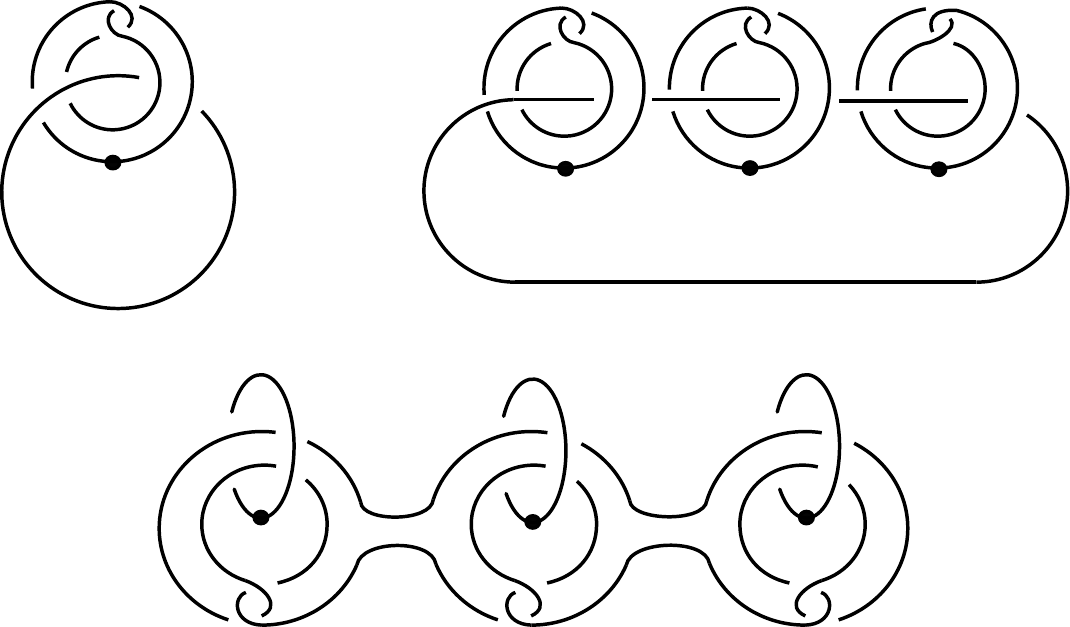}
\caption{Whitehead doubles and kinky handles}
\label{double}
\end{figure}

\subsection{Casson handles}\label{Casson}
We will make extensive use of {\em Casson handles}. These smooth 4-manifolds were first introduced by Casson~\cite{C}, then shown by Freedman~\cite{F} to be homeomorphic (rel boundary) to the open 2-handle $D^2\times\R^2$ as the cornerstone of his classification theorem for simply connected topological 4-manifolds. When Donaldson~\cite{D} showed that the smooth analogue of that theorem is false, it followed immediately that Casson handles are not all diffeomorphic to the open 2-handle -- in fact, the topological core disk of a Casson handle (the homeomorphic image of the core $D^2\times\{0\}\subset D^2\times\R^2$) is typically not topologically isotopic to a smooth disk. But as Freedman observed in his original paper, his proof showed that the topological core could always be assumed isotopic to an almost-smooth disk (as we discuss in Section~\ref{AlmostSmooth}). He also observed that the interior of a Casson handle {\em is} diffeomorphic to $\R^4$, by a simple engulfing argument that we reproduce below (Proposition~\ref{engulf}).

The basic building blocks of Casson handles are {\em kinky handles}. A kinky handle $T_1$ is a compact tubular neighborhood in a 4-manifold of a generically immersed 2-disk, its {\em core}. Equivalently, $T_1$ is made from a trivial disk bundle over the core disk (a 2-handle) by self-plumbing. The boundary of the core is the {\em attaching circle} of $T_1$, and the {\em attaching region} $\partial_-T_1\subset \partial T_1$ is the tubular neighborhood of the attaching circle obtained by restricting the disk bundle (i.e., the attaching region of the plumbed 2-handle). Either description of a kinky handle shows that it is homotopy equivalent to a wedge of $k$ circles, where $k$ is the number of double points of the core. A slightly closer analysis shows that it is diffeomorphic to a boundary sum of $k$ copies of $S^1\times D^3$. This is shown in Figure~\ref{kink} in the case with two positive double points and one negative: Take the product of the pictured genus-3 handlebody with $I$, thinking of the $I$ coordinate as time $t$. The attaching circle is the pictured curve at $t=1$. As $t$ decreases, the core is depicted as a circle in $T_1$ that unknots itself by the obvious homotopy with three self-crossings, and then bounds a disk (local minimum) and disappears. This figure also depicts the kinky handle as a 4-ball with three 1-handles attached, which is described in Kirby calculus by Figure~\ref{double}(c) (or equivalently (b)). In these latter diagrams, the 4-manifold is given as the complement of tubular neighborhoods of the disks of Section~\ref{Satellites} bounded by the dotted circles. The general case is similar. (For more details see, e.g.,~\cite[Chapter~6]{GS}.) Note that every kinky handle is made from $k$ copies of the simplest one (shown in Figure~\ref{double}(a)) by reversing some orientations and pairwise boundary summing. A similar description builds their interiors by pairwise end summing. We often think of a kinky handle as a generalized 2-handle. That is, we glue its attaching region to the boundary of another 4-manifold so that a preassigned framed circle in the latter is identified with the attaching circle $C$ of $T_1$ with the 0-framing as it appears in Figures~\ref{double}~and~\ref{kink}. Note that this framing extends over any embedded surface in $T_1$ with boundary $C$, but does not in general agree with the normal framing of the immersed core disk (which is the blackboard framing in Figures~\ref{double}(c)~and~\ref{kink}, so has coefficient $+2$ in that example). In general, the framing induced by the immersed core has coefficient $2(k_+-k_-)$ in these diagrams, where $k_+$ (resp.\ $k_-$) is the number of positive (resp.\ negative) double points.

\begin{figure}
\centering
\includegraphics{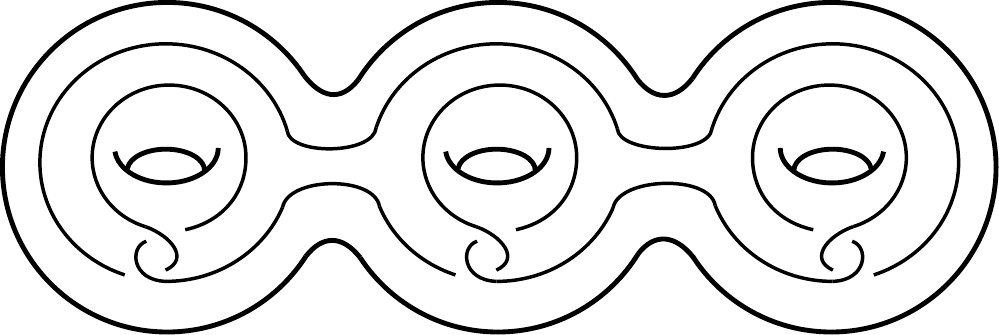}
\caption{A kinky handle: The diagram represents the top boundary of a product with $I$. The core is the obvious immersed disk with three double points bounded by the pictured curve.}
\label{kink}
\end{figure}

A Casson handle is made from an infinite stack of kinky handles. To begin, a {\em 1-stage tower} $T_1$ is a kinky handle. There is an obvious framed link in its boundary, consisting of the 0-framed meridians of the dotted circles in Figure~\ref{double}, for which attaching 2-handles would cancel the 1-handles to yield a 4-ball with an unknotted attaching circle. These 2-handles would fill in the holes of Figure~\ref{kink} in the obvious way. To obtain a {\em 2-stage tower} $T_2$, we instead attach kinky handles to this framed link. As Figure~\ref{double}(b) indicates, a kinky handle is obtained from a 2-handle $H$ (the $k=0$ case) by removing disks. Specifically, we take a {\em ramified double} of the cocore disk of $H$, the satellite operation corresponding to a pattern as in the figure (i.e., doubling parallel copies of the cocore disk), then delete a tubular neighborhood of the resulting disks from $H$. To obtain $T_2$, we apply this procedure to the cancelling 2-handles for $T_1$. The overall result is to replace the dotted disks for $T_1$ by their ramified doubles. Iterating this procedure using the framed links at the top stage kinky handles, we get a sequence of towers $T_1\subset T_2\subset T_3\subset\cdots$. A Casson handle $CH$ is obtained from the infinite union of such a sequence by removing all of its boundary except the open attaching region $\inter\partial_-T_1=\partial CH$. Equivalently, we can assume the neighborhoods of ramified doubles removed at each stage are nested; the Casson handle is then obtained by removing their infinite intersection (and some boundary) from $H$. (At generic points, this intersection is locally a product of $\R^2$ with a Cantor set, appearing in the boundary as a generalized Whitehead continuum.) The almost-smooth core of such a standardly embedded Casson handle $CH\subset H$ is topologically ambiently isotopic in $H$ to the core of $H$ (since the knot group is $\Z$ or by the more direct method of \cite[Theorem~6.2]{MinGen}), so the Casson handle itself is (nonambiently) topologically isotopic rel boundary to the open 2-handle made from $H$ by removing suitable boundary. A Casson handle is completely specified by a based, signed tree with no finite branches. (Each vertex represents a kinky handle, and its signed edges directed away from the base correspond to its double points.)

\begin{de}\label{refinement} A {\em refinement} of $CH$ is a Casson handle $CH'$ whose signed tree contains that of $CH$.
\end{de}

\noindent Any such refinement has a canonical embedding $CH'\subset CH$ with $\partial CH'=\partial CH$, made by ambiently adding new double points and kinky handles (or removing a certain nested intersection). Any two Casson handles have a common refinement, for example, by identifying the base points of the corresponding signed trees.

\begin{prop}\label{engulf} {\rm (Freedman~\cite[Theorem~2.1]{F}.)}
The interior of every Casson handle is diffeomorphic to $\R^4$.
\end{prop}

\begin{proof}
Slightly thin the towers $T_n$ of the given Casson handle $CH$ by deleting boundary collars, so that each $T_n$ ($n\ge2$) contains $T_{n-1}$ in its interior and the union of these compact towers is $\inter CH$. We have seen that each kinky handle is a closed tubular neighborhood of a wedge of circles. Each such circle can be identified with the attaching circle of the corresponding kinky handle at the next higher stage. It follows by induction that each tower is also a neighborhood of a wedge of circles. (Collapse from the first stage up.) Since each circle is nullhomotopic in its next-stage kinky handle, it follows that $T_{n-1}$ is nullhomotopic in $T_n$. Since homotopy implies isotopy for circles in a 4-manifold, $T_{n-1}$ can be smoothly isotoped into a 4-ball in $T_n$. Equivalently, we can find a ball $B_n$ with $T_{n-1}\subset B_n \subset\inter T_n$ in the original nest of (slightly thinned) towers. Thus, $\inter CH$ is a nested union of balls. It is now easy to construct a diffeomorphism $\inter CH\approx\R^4$ sending each $B_n$ onto the ball of radius $n$.
\end{proof}

\subsection{Almost-smooth surfaces}\label{AlmostSmooth}  
To arrange Casson handle cores to be almost-smooth, we need the following notion:

\begin{de}\label{cellularity} A subset $C$ of an $m$-manifold $X$ is {\em smoothly cellular} if it can be described as a nested intersection of smooth $m$-balls $B_i$ with $\inter B_i\supset B_{i+1}$ for each $i$. It is {\em smoothly boundary cellular} if it is an intersection of half-balls, each intersecting $\partial X$ in an ($m-1$)-ball, and nested as before (using ``int'' in the set-theoretic sense).
\end{de}

The notion of cellularity was well-known at the time of Freedman's work. The author is not aware of explicit previous usage of boundary cellularity, although it was surely implicitly known to Freedman. (We suppress further usage of the adjective ``smoothly'' since we are taking everything to be smooth unless otherwise specified.) It is routine to check that if $C$ is cellular in either sense then $X-C$ is diffeomorphic to $X-\{p\}$, where $p$ is an interior or boundary point, respectively. Intuitively, $C$ can be ``shrunk to a point'' in $X$ without changing the ambient smooth structure.

\begin{thm}\label{almostSmooth} {\rm (Freedman.)}
The topological core of any Casson handle $CH$ is topologically ambiently isotopic (rel boundary and with compact support) to a disk $D$ that is smooth except at one point $p$ that can be chosen to be in either the interior or boundary of $D$.
\end{thm}

The interior case is essentially \cite[Addendum~A to Theorem~1.1]{F}. The boundary case is unpublished but contemporaneous. We simplify the proof in places, using more recent methods. The proof also verifies that the (compactly supported) topological isotopy class of the core disk does not depend on the topological identification of $CH$ with a standard open 2-handle. We will usually assume such cores are smooth except at one interior point, but will also have use for the boundary case (Proposition~\ref{smoothing}).

\begin{proof}
Freedman's proof \cite{F} that $CH$ is a topological open 2-handle uses a difficult lemma that embeds Casson handles inside preassigned finite towers. Repeated use of this in towers at high stages exhibits $CH$ as a nested union of compacta parametrized, preserving order, by the standard Cantor set. Each compactum $C$ is the end compactification of an infinite Casson tower with the same attaching circle as $CH$. (These compacta are not manifolds, although they can be taken to be topological 2-handles if we replace Casson handles by Freedman's more general towers with many embedded surface stages, as in Freedman--Quinn \cite{FQ}.) Consider such a $C$ whose parameter is approached from above by a sequence in the Cantor set. Isotope $C$ slightly away from $\partial CH$. Then $C$ is cellular. This is because by construction, each neighborhood $V$ of $C$ contains some Casson tower $T_n$ (again isotoped away from $\partial CH$) whose subtower $T_{n-1}$ contains $C$. As in the proof of Proposition~\ref{engulf}, there is a smoothly embedded ball $B_n$ with $C\subset T_{n-1}\subset B_n\subset T_n\subset V$, where we include into interiors after slightly thinning $T_{n-1}$. Such balls can be constructed to nest as required. Cellularity of $C$ implies $CH-C$ is diffeomorphic to $CH-\{p\}$, for an arbitrary $p\in\inter CH$. The annulus $A$ connecting the attaching circle of $CH$ to the corresponding circle in $C$ is sent by this diffeomorphism to an annulus in $CH-\{p\}$ that compactifies to an almost-smooth disk $D$ in $CH$. Freedman showed that $D$ is a topological core by a deep dive into his 2-handle recognition proof, but this can be avoided with more modern technology: Analyzing $\pi_1$ in $CH-D$ shows that $D$ is ``locally homotopically unknotted'', hence locally flat  by Venema \cite{V}, and $\pi_1(CH-D)\cong\Z$. Then $D$ is topologically isotopic to the core of any given topological open 2-handle structure (essentially by the proof that a 2-sphere in $S^4$ with knot group $\Z$ is topologically unknotted \cite[Theorem~11.7A]{FQ}). To similarly arrange the singularity to lie on the boundary, do not isotope the subsets entirely away from $\partial CH$, but instead leave their intersections with $\partial CH$ a nested sequence of 3-balls. We can arrange the resulting boundary-cellular set to have the form $C'=C\cup\gamma$, where $\gamma$ is an arc in $A$ from $\partial CH$ to $C$. Shrinking $C'$ to a point sends $A$ to the required disk $D'$ that is smooth except at a boundary point. If we instead do the shrink in two stages, first shrinking $C$ gives the disk $D$ that we have already identified with a topological core, containing the embedded image of $\gamma$. Shrinking  that image does not change the homeomorphism type of $(CH,D)$, so the resulting $D'$ is also a topological core.
\end{proof}

Freedman actually used the innermost $C$ of the uncountable family as the cellular set, but we will later have use of the whole family.

We can alternatively derive the case $p\in\partial D$ from the case $p\in\inter D$ proved in \cite{F}. Given the latter, choose an arc $\gamma$ in $D$ from $\partial D$ to $p$. We can assume this is smooth except at $p$. In $CH-\{p\}$, there is a {\em unique} ray toward $p$ (up to smooth ambient isotopy) since the end at $p$ is simply connected (cf.~Section~\ref{Sums}). Thus, after an almost-smooth isotopy of $D$ in $CH$, we can assume $\gamma$ is smooth. Then $\gamma$ is smoothly boundary cellular, so we can shrink it to a boundary point, away from which $D$ is smooth.

Combining this theorem with work of Quinn \cite{Q} immediately gives:

\begin{cor}\label{cpctSurf}
Every compact, locally flat surface $F$ in a 4-manifold is almost-smoothable.
\end{cor}

\begin{proof}
Decompose $F$ as a CW-complex with a unique 2-cell, then thicken in the obvious way to a topological 2-handlebody. By \cite[2.2.2 and 2.2.4]{Q} (see also \cite[Theorem~5.2]{JSG}), we can assume the underlying 1-handlebody and its intersection with $F$ are smooth, and replace the 2-handle by a Casson handle. The latter is only given in \cite{Q} to be ``weakly unknotted'' in the original topological 2-handle, but we can now infer their two cores are topologically isotopic (as in \cite[Theorem~11.7A]{FQ} again). That is, we can topologically isotope the remaining disk of $F$ to agree with an almost-smooth core.
\end{proof}

The minimal genus and kinkiness of the resulting singularity were addressed by the author in \cite[Theorems~6.2 and 8.4]{MinGen}; cf.~Theorem~\ref{sing}. Proposition~\ref{smoothing} below shows that no singularity is necessary when $F$ is open.


\section{Initial results}\label{Context}

We now prove the results that follow most easily from the literature. This naturally leads into a brief discussion of the exotic $\R^4$ theory that we will need in Section~\ref{Branch}. For context, we then briefly digress to discuss (non)existence of exotic linear subspaces and annuli in general  dimensions, and the dual problem of smoothability of topological submanifolds of 4-manifolds.

\subsection{Exotic annuli and planes from Casson handles}\label{CH}

In this section, we construct exotic annuli and planes realizing all nonzero values of $\kappa$ and $\kappa^\infty$, respectively. We prove Theorem~\ref{ginfty}, that exotic planes realize all values of $\kappa^\infty$ and $g^\infty=\max\{\kappa_\pm^\infty\}$. (The case with $\kappa^\infty=0$ consists of quoting Theorem~\ref{0g}(a), whose proof uses different methods and is deferred to later sections.) In addition, part of Theorem~\ref{0g}(b), uncountably many annuli with each nonzero $\kappa$, is immediate from Corollary~\ref{annuli} below, with the rest completed in Section~\ref{DrawPlanes} by drawing the annuli. To begin, we recall what is known about the diffeomorphism classification of Casson handles, based on the author's previous work \cite{kink}, \cite{D5T}, \cite{MinGen}.

\begin{de}\label{CHkink}
The {\em minimal genus} $g(CH)$ of a Casson handle is the minimal genus of an embedded surface in $CH$ bounded by the attaching circle \cite{MinGen}. Similarly, the {\em kinkiness} $\kappa_\pm(CH)$ is the minimal number of double points of the given sign in a generically immersed disk bounded by the attaching circle \cite{kink}.
\end{de}

Note that if $CH'$ is a refinement of $CH$ (Definition~\ref{refinement}) then it canonically embeds in $CH$, so $g(CH')\ge g(CH)$ and similarly for $\kappa_\pm$.

\begin{thm}\label{CHclass} \cite[Theorem~8.6]{MinGen}
For each $(m,n)\in\Z^{\ge0}\times\Z^{\ge0}-\{(0,0)\}$, there are uncountably many diffeomorphism types of Casson handles with $\kappa=(\kappa_+,\kappa_-)=(m,n)$ and $g=\max\{m,n\}$.
\end{thm}

\begin{proof}
The reference deals with $\kappa$, but $g$ follows also. Let $CH_+$ be the Casson handle with one double point at each stage and all signs positive, so each kinky handle is given by Figure~\ref{double}(a). For each $m,n\in\Z^{\ge0}$ except for $m=n=0$, let $CH_{m,n}$ be the Casson handle made from $m$ copies of $CH_+$ and $n$ copies of its mirror image by gluing their attaching regions in the obvious way so that the resulting first stage core has $m$ positive and $n$ negative double points. Clearly, $\kappa_+(CH_{m,n})\le m$ and $\kappa_-(CH_{m,n})\le n$. The embedded surface made from the first stage core by tubing together pairs of double points of opposite sign and smoothing the remaining double points (replacing each local pair of intersecting disks by an annulus) shows that $g(CH_{m,n})\le \max\{m,n\}$. For lower bounds on these invariants, attach $CH_{m,0}$ to a 4-ball along an unknot with framing $2m-2$. The resulting interior admits a Stein structure. (See, for example, \cite[Chapter~11]{GS}. Note that the first stage is made from the cotangent bundle of $S^2$ by self-plumbing.) The adjunction inequality for Stein surfaces now shows that $g(CH_{m,0})\ge m$, and similarly $\kappa_+(CH_{m,0})\ge m$. (The adjunction inequality is insensitive to negative double points, ultimately since they can be blown up without changing the homology class.) Since $CH_{m,n}$ is a refinement of $CH_{m,0}$, we conclude that $\kappa(CH_{m,n})=(m,n)$ (with $\kappa_-$ evaluated by reversing orientation) and $g(CH_{m,n})=\max\{m,n\}$.

To produce uncountable families, let $X=\C P^2\# k\overline{\C P^2}$ and note that for large $k$, the class $\alpha=3e_0+\sum_{i=1}^8e_i\in H_2(X)$ (using the obvious basis) has $\alpha^2=1$ but orthogonal complement that is negative definite and not diagonalizable. (When $k=8$, the complement is even and hence $E_8$. This persists as a summand when $k$ increases.) Thus, $\alpha$ cannot be represented by an embedded sphere, which could be blown down to contradict Donaldon's Diagonalizability Theorem \cite{D}. However, Casson's algorithm \cite{C} represents $\alpha$ by a (highly ramified) Casson handle $CH$ attached to a 1-framed unknot in a 4-ball for $k\ge9$ (showing that $g(CH)\ne0$). With more work \cite{D5T}, one can explicitly embed the first stage (or first several) with only one (positive) double point so that Casson's algorithm still generates the rest of the Casson handle. Thus, we can assume $CH$ is obtained from $CH_+$ by refining only the higher stages. Refining further, we can alternatively assume $CH$ is made from any given $CH_{m,n}$ by refining only higher stages (and reversing orientation on $X$ if $m=0$), so that it has the same $\kappa$ and $g$ as $CH_{m,n}$. Now recall from the proof of Theorem~\ref{almostSmooth} that $CH$ contains an uncountable nest of Casson handles parametrized by a Cantor set. These can be constructed by reimbedding only above the first stage, so that they all have the same $\kappa$ and $g$ as $CH_{m,n}$. If any two of these Casson handles were diffeomorphic, then attaching them to concentric 4-balls in $X$ would create a diffeomorphic pair $W$, $W'$ of open subsets carrying $\alpha$ and with $\cl W'\subset W$. We could then cut $W'$ out of $X$ and replace it by an infinite stack of copies of $W-W'$, glued along diffeomorphic ends (Figure~\ref{periodic}). This would yield a 4-manifold with a periodic end and a definite, nondiagonalizable intersection form, contradicting Taubes \cite{Tb} (cf.~\cite[Theorem~9.4.10]{GS}).
\end{proof}

\begin{figure}
\labellist
\small\hair 2pt
\pinlabel {$W'$} at 78 45
\pinlabel {$W$} at 89 77
\pinlabel {diffeomorphic ends of $W$ and $W'$} at 152 -1
\pinlabel {$X$} at 45 5
\pinlabel {$W-W'$} at 239 101
\endlabellist
\centering
\includegraphics{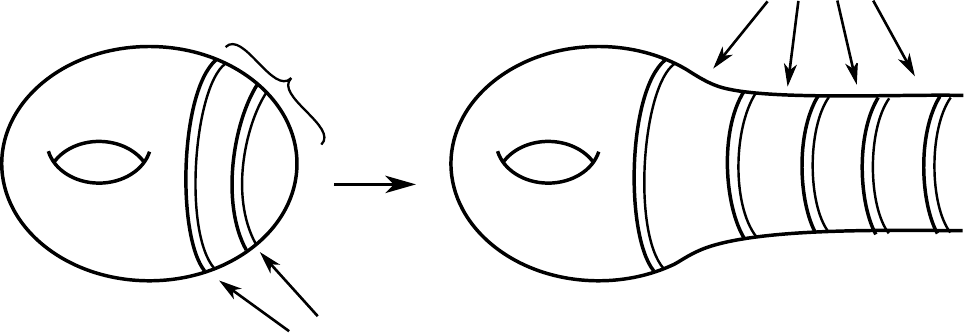}
\caption{Constructing a periodic end.}
\label{periodic}
\end{figure}
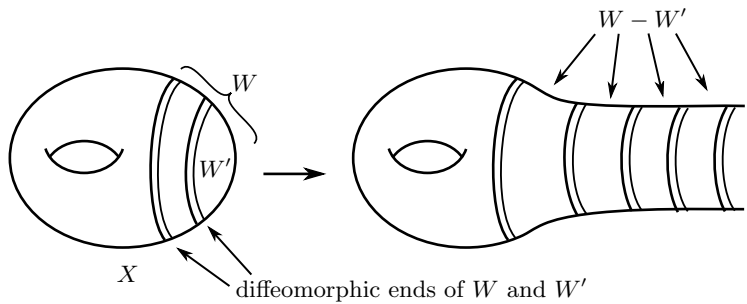

\begin{cor}\label{annuli}
For each $(m,n)\ne(0,0)$, there is an uncountable family of pairwise nonisotopic exotic annuli in $\R^4$ with $\kappa=(m,n)$ and $g=\max\{m,n\}$. Each annulus in the family extends to a topologically standard, almost-smooth plane that is relatively unsmoothable in that it cannot be smoothed by any topological isotopy that is smooth outside a compact set.
\end{cor}

\begin{proof}
An almost-smooth core disk of a Casson handle $CH$ (Theorem~\ref{almostSmooth}, interior case) is topologically standard in $\inter CH$, which is diffeomorphic to $\R^4$. Thus, it determines a smooth, topologically standard annulus. By Proposition~\ref{rebuild}(a) we can reconstruct the Casson handle from the annulus, so nondiffeomorphic Casson handles yield nonisotopic annuli in $\R^4$. The invariants $\kappa$ and $g$ of the annulus equal those of the Casson handle, and the core disk interior is relatively unsmoothable whenever $\kappa$ or $g$ is nonzero.
\end{proof}

More generally, every exotic annulus with $g>0$ extends as above, by Corollary~\ref{cpctSurf} applied to the topological spanning disk.

\begin{proof}[Proof of Theorem~\ref{ginfty}]
We wish to construct an exotic plane realizing any given value of $\kappa^\infty$. First, we topologically isotope the standard $S^2\subset S^4$ to create a suitable unique singularity: Decompose its tubular neighborhood as $B^4$ with a 2-handle attached, then canonically embed a given $CH_{m,n}$ in the 2-handle. As in the proof of Theorem~\ref{almostSmooth}, $CH_{m,n}$ contains a nested, decreasing sequence of Casson handles with intersection $C$ that is cellular in $S^4$. Isotope $C$ and the Casson handles away from $\partial CH_{m,n}$ so that the latter form a neighborhood system of $C$ in $CH_{m,n}$ (Figure~\ref{shrink}, left). Now shrinking $C$ to a point as in the figure preserves the nest of Casson handles, but their intersection becomes the singular point $p$ of the resulting almost-smooth disk $D$, with each Casson handle intersecting $D$ in a disk. After topological isotopy of the standard $S^2$, we can assume (as in the proof of Corollary~\ref{cpctSurf}) that it contains $D$. Deleting $p$ from $(S^4,S^2)$ now gives $P$, a topologically standard $\R^2$ in $\R^4$. The sequence of Casson handles can be constructed to all have the same first stage, and hence the same $\kappa_\pm$ and $g$, as $CH_{m,n}$. It is now easily verified that $\kappa^\infty(P)=(m,n)$. Alternatively, the first stages can be refined so that $\kappa_+$ or $\kappa_-$ (or both) increase without bound, so we can realize any preassigned $\kappa^\infty(P)\ne(0,0)$ and $g^\infty(P)=\max\{\kappa_\pm(P)\}$. The remaining case of vanishing invariants follows from Theorem~\ref{0g}(a). The exotic planes of that theorem are constructed in Section~\ref{planeFamilies} (proof of Theorem~\ref{order}) and shown to have vanishing invariants in Section~\ref{DrawPlanes}.
\end{proof}

\begin{figure}
\labellist
\small\hair 2pt
\pinlabel {$C$} at 76 31
\pinlabel {$p$} at 264 48
\pinlabel {$S^2$} at 192 39
\pinlabel {$CH_{m,n}$} at 312 3
\pinlabel {$D$} at 220 38
\pinlabel {$S^2$} at 14 39
\pinlabel {$CH_{m,n}$} at 133 3
\endlabellist
\centering
\includegraphics{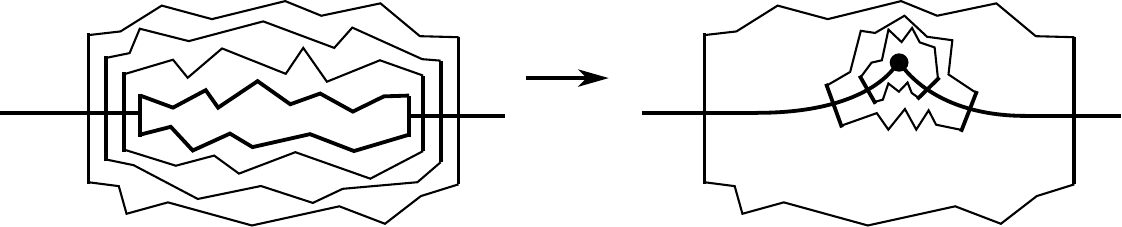}
\caption{Creating a singularity with a neighborhood system of Casson handles.}
\label{shrink}
\end{figure}
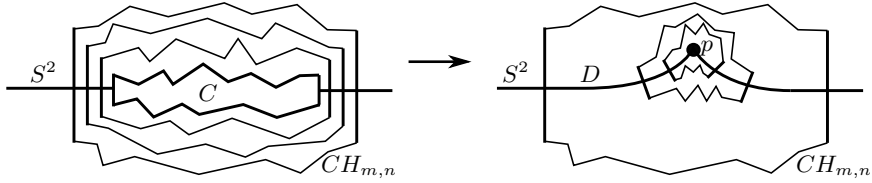

A similar approach is used in \cite{MinGen} to analyze the invariants of singularities of almost-smooth surfaces more generally.

\begin{Remark}\label{original}
These are essentially the original exotic planes implicit in \cite[Remark~4.2]{kink}. The construction in that remark was to one-point compactify (ramified versions of) $(\inter CH_{m,n}, \inter D)$, then smooth the singularity at infinity by an omitted argument similar to our proof that $g^\infty=0$ for Theorem~\ref{0g} (Section~\ref{DrawPlanes}). Yet another description of these planes is to smooth the core of $CH_{m,n}$ away from a boundary point (Theorem~\ref{almostSmooth}), then pass to the interior.
\end{Remark}

\subsection{Exotic $\R^4$ methods}\label{R4} 
In Section~\ref{Branch}, we will extensively use exotic $\R^4$ theory. We now illustrate the main ideas, starting from the proof of Theorem~\ref{CHclass}. (See Gompf--Stipsicz \cite[Section~9.4]{GS} for a broader look.)
First note that the manifold $W$ from that proof embeds in $\C P^2$, by standardly embedding $CH$ in the 2-handle of the obvious handle decomposition of the latter. We can assume the closure of $W$ has the form $B^4\cup C$ where $C$ is cellular as in the proof of Theorem~\ref{almostSmooth}. Then $R=\C P^2-\cl W$ is diffeomorphic to the complement of an almost-smooth sphere topologically isotopic to $\C P^1$, so it is homeomorphic to $\R^4$. However, it cannot be diffeomorphic to $\R^4$, since its end is diffeomorphic to that of the negative definite manifold $X-\cl W$: This has no smooth 3-spheres surrounding $\cl W$ in $X$, or else we could cut along such a 3-sphere and glue in $B^4$, contradicting Donaldson's Theorem. In fact, $R$ is a {\em large} exotic $\R^4$, meaning it has a compact, codimension-0 submanifold that cannot smoothly embed in $S^4$: If a sufficiently large compact subset embedded in $S^4$ (or in any closed, negative definite manifold), a similar gluing argument would fuse the latter into $X-\cl W$ to again contradict Donaldson. Now recall that such cellular subsets $C\subset CH$ occur in an uncountable nested family, so we obtain an uncountable family of $\R^4$-homeomorphs $R_s$ constructed in the same manner, nested in the reverse order. Thus, they are parametrized, preserving order, by $\Sigma$, the Cantor set with the upper endpoint of each middle third removed. Alternatively, we can construct a nested family parametrized by an interval, by considering open balls of sufficiently large radius in the topological $\R^4$ structure of $R$. Either way, the members $R_s$ of the family are pairwise nondiffeomorphic by the periodic end argument illustrated in Figure~\ref{periodic}. In fact, we obtain a 2-parameter family by end summing with opposite orientations, with $R_s\natural \overline{R_t}$ embedding in $R_{s'}\natural \overline{R_{t'}}$ if and only if $s\le s'$ and $t\le t'$. (There is no embedding with $s>s'$ since $\overline{R_{t'}}$ embeds in $\overline{\C P^2}$ and the Periodic End Theorem still applies with negative definite homology in the end. For the case $t>t'$, reverse orientation.) Similarly, the nonembedding statement for the manifolds $R_s$ persists if we end sum them with more general 4-manifolds that embed in closed, negative definite, simply connected 4-manifolds.

\begin{Remark}
We have also exhibited uncountably many almost-smooth spheres topologically isotopic to $\C P^1\subset\C P^2$ and distinguished (up to almost-smooth isotopy) by the diffeomorphism types of their exotic $\R^4$ complements. These can be constructed for any finite $\kappa$ with $\kappa_+>0$ by controlling the first stage as for Theorem~\ref{CHclass} and evaluating $\kappa$ as for Theorem~\ref{ginfty}.
\end{Remark}

In Section~\ref{Branch}, we discuss {\em small} exotic $\R^4$-homeomorphs, i.e., those that are not large. All known examples embed in $S^4$ (rather than just their compact subsets embedding). These are obtained by a method of Freedman \cite{DF} from failure of the smooth h-Cobordism Theorem for 4-manifolds (Donaldson \cite{Dpoly}). An end-periodic version of that theorem again yields uncountable families (DeMichelis--Freedman \cite{DF}). Since that version allows negative definite homology in the end, we obtain 2-parameter families \cite[Theorem~1.1]{menag}, \cite[Proposition~5.6]{BG} and their generalization as before \cite[Lemma~7.3]{MinGen}. (The conclusions are slightly weaker than before, since one must work relative to a certain compact subset: In the uncountable families, each diffeomorphism type appears at most countably often, so we obtain the cardinality of the continuum in ZFC set theory. One can construct such families so that some diffeomorphism type appears more than once \cite[Remark~6.8]{groupR4}.)

\subsection{Other dimensions}
We now show that exotic planes have no analogues in other dimensions.

\begin{prop}\label{highdim}
Suppose a homeomorphism $h\co\R^n\to\R^n$ is a local diffeomorphism near $Z=\R^k\times\{0\}$. Then $h(Z)$ is smoothly ambiently isotopic to $Z$ unless $n=4$ and $k=2$. The same holds if $Z$ is the annulus $(\R^k-\inter D^k)\times\{0\}$ unless $n=4$ and $k=2,3,4$, or $k=3$ and $n=5,6,7$.
\end{prop}

\noindent Every exotic plane or annulus in $\R^4$ has such a homeomorphism $h$ by uniqueness of topological normal bundles \cite[9.3D]{FQ}. One might hope for the same to hold for all smooth embeddings $\R^k\emb \R^n$ that are topologically standard, but the author is unaware of a sufficiently general uniqueness theorem for normal bundles. Since every orientation-preserving diffeomorphism $\R^k\to\R^k$ is isotopic to the identity, the first conclusion of the proposition can be immediately strengthened from an isotopy of the submanifold $h(Z)$ to an isotopy sending the restricted map $h|Z$ to $\id_Z$. However, this fails for annuli in some high dimensions: An exotic self-diffeomorphism of $S^{k-1}$ ($k\ge7$) extends radially over $\R^k$ and then as a product with $\id_{\R^{n-k}}$ over $\R^n$, giving a self-homeomorphism $h$ of $\R^n$ that is a local diffeomorphism off of $\{0\}\times\R^{n-k}$. This satisfies the hypotheses and conclusion of the second sentence of the proposition with $h(Z)=Z$. Suppose there were a smooth ambient isotopy sending $h|Z$ to $\id_Z$. For $k=n$, Proposition~\ref{rebuild}(b) would produce a forbidden diffeomorphism between $S^k$ and the exotic sphere $\Sigma$ obtained by gluing two balls via $h|S^{k-1}$. Similarly, for $k=n-1$, we would obtain a diffeomorphism $\Sigma\times\R\approx S^{n-1}\times\R$. This manifold would then contain disjoint copies of $\Sigma$ and $S^{n-1}$ cobounding a forbidden h-cobordism.

\begin{proof}
The cases $n\le 4$ only use smoothness of $h$ through that of its image submanifold $h(Z)$: When $n<4$, we can pairwise smooth $h\co(\R^n,Z)\to (\R^n,h(Z))$ by standard 3-manifold topology, then smoothly isotope $h$ to the identity. When $n=4$, smooth 1-manifolds cannot be knotted. Every embedding $\R^3\emb\R^4$ exhibits $\R^4$ as an end sum of two $\R^4$-homeomorphs. (A tubular neighborhood of the $\R^3$ can be identified with that of the gluing arc, e.g. \cite{CG}.) Since the monoid of these has no inverses \cite{infR4} (by the Eilenberg Swindle/Mazur Trick, cf.~last paragraph of Section~\ref{Sums}) both summands are standard, as is the original embedding.

When $Z=\R^k\times\{0\}$, $n>4$, we can assume after normal radial dilation that $h$ is a local diffeomorphism on $N=\R^k\times D^{n-k}$. Thus, the pulled back smoothing on the domain is standard on $N$. But $\R^n-\inter N\approx \partial N\times [0,\infty)$. Since smoothings rel boundary are classified by a homotopy lifting problem when $n>4$ (Kirby--Siebenmann \cite{KS}, see also Freedman--Quinn \cite[Section~8.3]{FQ} for a quick overview), the smoothing of $\R^n$ is isotopic rel $N$ to the standard smoothing. Equivalently, $h$ is topologically isotopic rel $N$ to a diffeomorphism, which is then smoothly isotopic to the identity, completing the proof for $\R^k\times\{0\}$.

The proposition holds without use of $h$ when $n\ge 2k+2$, since homotopy implies isotopy by transversality, completing the $k=3$ case. For the remaining case of annuli with $k\ne3$, it suffices to show that the topological open $k$-handle $\hat X$ arising from Proposition~\ref{rebuild}(a) is diffeomorphic to a standard open $k$-handle, preserving the attaching sphere setwise. This follows unless $k=3$ or $k\ge7$ by vanishing of the smoothing uniqueness obstruction $\pi_k({\rm TOP}/{\rm O})$ \cite{KS}, \cite[Section~8.3]{FQ}. For $k=n\ge7$, $\hat X$ must be diffeomorphic to a ball as required, by puncturing it and applying the h-Cobordism Theorem. Thus, the uniqueness obstruction is encoded in how its boundary is identified with $S^{k-1}$. Stability of high-dimensional smoothing theory \cite[Essay~I, Section~5, Remark~2]{KS} now gives the required diffeomorphism for $n>k\ge7$.
\end{proof}

Regarding the missing cases of the proposition, the bulk of this paper deals with $(n,k)=(4,2)$. The case of annuli with $n=k=4$ is equivalent to the notorious 4-dimensional smooth Schoenflies Conjecture \cite[Proposition~2.2]{groupR4}. The case $k=3$ is equivalent (as above) to nonexistence of an exotic open 3-handle with interior diffeomorphic to $\R^n$. When $n=4$, there are uncountably many smoothings of $S^3\times\R$ (for example, connected sums of $\R^4$-homeomorphs). Drilling out a neighborhood of a properly embedded line gives uncountably many exotic 3-handles, but it is not known if an exotic 3-handle can ever have interior diffeomorphic to $\R^4$. If there is an exotic 3-handle in dimension 5, 6 or 7, it and the corresponding exotic annulus are unique since $\pi_3({\rm TOP}/{\rm O})\cong\Z_2$.

\subsection{Unsmoothable submanifolds}
In 4-manifolds, there are many ways to construct topological (locally flat) embeddings of compact surfaces that cannot be smoothed by topological isotopy. An example from the 1980s is the topological sphere representing the class $\alpha$ in the proof of Theorem~\ref{CHclass}. The manifold $\inter(B^4\cup CH_{m,0}$) in the first paragraph of that proof has homology generated by a class with minimal genus $m$ but represented by a topological sphere (so by an unsmoothable surface of any genus less than $m$). It is an interesting open question whether such compact unsmoothable surfaces exist in $\R^4$. Corollary~\ref{annuli} exhibits topological planes in $\R^4$ that are unsmoothable relative to their smooth ends. However, without restriction on the end, noncompact surfaces  in smooth 4-manifolds are always smoothable:

\begin{prop}\label{smoothing}
Every topological (locally flat, proper) embedding of a noncompact surface into a smooth 4-manifold $X$ is topologically (ambiently) isotopic to a smooth embedding.
\end{prop}

\begin{proof}
Let $F\subset X$ be the image of the embedding. As in the proof of Corollary~\ref{cpctSurf}, we can smooth $F$ near the 1-skeleton of a cell decomposition and realize the rest of $F$ as topological cores of disjointly embedded Casson handles $CH_i$ in $X$. Choose a point $p_i\in CH_i$ in each attaching circle. These are the endpoints of a properly embedded family of rays $\gamma_i$ in $F$ whose interiors are disjoint from the Casson handles. Let $B_i$ be a half-ball in $CH_i$ centered at $p_i\in\partial CH_i$. By pushing along each $\gamma_i$ as in Figure~\ref{push}, we can disjointly reimbed each $CH_i-\{p_i\}$ in $X$, fixing $CH_i-B_i$ and properly embedding $B_i-\{p_i\}$, sending $F\cap CH_i-\{p_i\}$ into $F$. But $F\cap CH_i$ is a topological core by construction, so Theorem~\ref{almostSmooth} gives a compactly supported topological isotopy in $CH_i$ that smooths it except at $p_i$. Applying this to each reimbedded $CH_i-\{p_i\}$ smooths $F$. The embedding map of $F$ can then be immediately smoothed by a pairwise isotopy of $(X,F)$.
\end{proof}

\begin{figure}
\labellist
\small\hair 2pt
\pinlabel {$B_i$} at 62 48
\pinlabel {$p_i$} at 74 58
\pinlabel {$\gamma_i$} at 90 65
\pinlabel {$F\cap CH_i$} at 33 34
\pinlabel {$F$} at 86 18
\endlabellist
\centering
\includegraphics{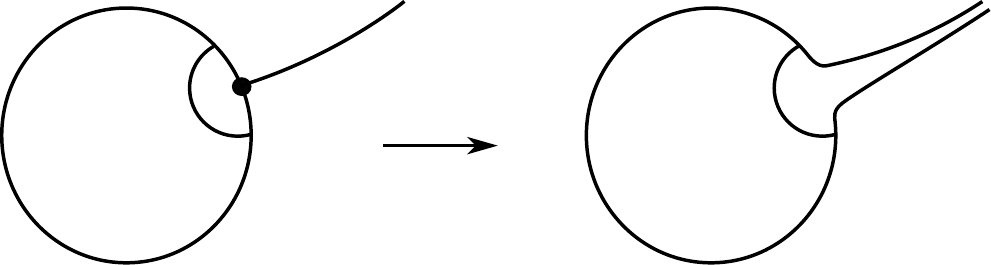}
\caption{Pushing a singularity out to infinity, as seen in the surface $F\subset X$.}
\label{push}
\end{figure}
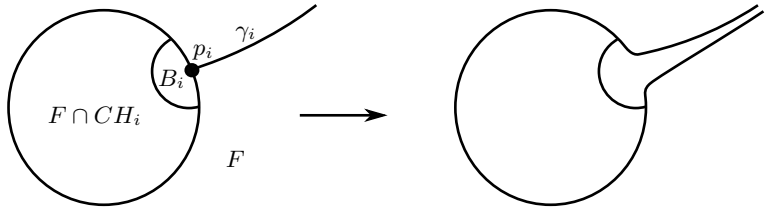

There are many ways to obtain unsmoothable embeddings of closed 3-manifolds into 4-manifolds. More strongly, embeddings that are not almost-smoothable (and cannot even be smoothed away from certain larger subsets) can be obtained in various ways using the topology of the 3-manifold.  (See the last paragraph of \cite[Section~6]{MinGen}). However, locally flat embeddings of $S^3$ and $\R^3$ in $\R^4$ are always topologically standard. (This follows up to homeomorphism from Brown \cite{Brown} and Cantrell \cite{Cn}, respectively, the latter after removing a point from the sphere pair in the corollary of \cite{Cn}. An ambient isotopy to the identity can then be obtained by straightening near $0\in\R^4$ \cite[2.2.2]{Q} and dilating.) For such embeddings, we can still obtain unsmoothability within a neighborhood of the submanifold, but must use a different approach:

\begin{prop}\label{R3} There is a topological embedding $\R^3\emb\R^4$ separating the components of some compact set $K$ in its complement, such that no smooth embedding of $\R^3$ separates $K$. Thus, the embedding $\R^3\emb\R^4-K$ is not smoothable by a topological isotopy. There is a topological embedding  $S^3\emb\R^4$ with a neighborhood in which it is not homologous to an almost-smooth embedding of $S^3$.
\end{prop}

\begin{proof}
There is a topological embedding $S^3\times\R\emb S^4$ whose image $U$ contains no smooth 3-sphere separating its ends, as discussed more carefully in Section~\ref{Branch1}. Both of the required embeddings arise from this by deleting a point from $S^4$. For the embedding of $\R^3$, delete a point $p$ from the image $S$ of $S^3\times\{0\}$ to obtain a topologically embedded $Q$ in $\R^4$ homeomorphic to $\R^3$, separating the two components of $K= S^4-U$. Suppose there were a smooth $Q'\approx\R^3$ separating these. Then $Q'$ would split $\R^4$ as an end sum. But this operation has no inverses \cite{infR4} (by the Eilenberg Swindle/Mazur Trick, cf.~Section~\ref{Sums}). Thus, $\R^4-Q'$ would be diffeomorphic to two copies of $\R^4$. But then each copy would have smooth 3-spheres near infinity, contradicting the fact that $U$ has no such 3-spheres.

To realize the data of the last sentence of the proposition, instead remove a point of $S^4-U$ to obtain $S\subset U\subset\R^4$. Suppose there were an almost-smooth 3-sphere $S'$ in $U$ as in that sentence. We could assume its singularity agreed with $p$ from the previous paragraph (after a smooth, compactly supported isotopy of $U$ sending one point to the other). Then removing $p$ from $(S^4,S')$ would give $Q'$ forbidden by the first paragraph.
\end{proof}


\section{Exotic branched coverings}\label{Branch}

Having realized all nonzero values of $\kappa^\infty$ (and thereby $g^\infty$)  by exotic planes for Theorem~\ref{ginfty}, we probe deeper by studying exotic planes for which the corresponding double branched covers can be recognized as exotic smoothings of $\R^4$, allowing us to harness the powerful theory of such smoothings. The first such example, due to Freedman, was later incorporated by the author into a peculiar $\Z_2\oplus\Z_2$-action \cite{menag}. We now analyze this action in detail. (We reverse the orientations of \cite[Sections~3 and 4]{menag} to obtain the more ``natural" orientations of subsequent papers that are stable under connected sum with $\overline{\C P^2}$ but not $\C P^2$. Signs of the double points of the Casson handles were unspecified in that paper, with one exception discussed below. We will ultimately see in Remark~\ref{reflect}(c) that both conventions can produce the same exotic $\R^4$, but with different $\Z_2\oplus\Z_2$-actions.) The action generates two different uncountable collections of exotic planes, corresponding to the two types of exotic $\R^4$-homeomorphs (small and large). One type is generated by unknots  (proof of Theorem~\ref{0g} in Section~\ref{DrawPlanes}) so has $g^\infty=0$, and can be drawn explicitly without local maxima (Section~\ref{DrawPlanes}). The other has $g^\infty>0$ (Proposition~\ref{gTaylor}), realizing each of infinitely many values of $g^\infty$ by uncountably many exotic planes (Corollary~\ref{ginf}), and requires infinitely many local maxima in any diagrams (Corollary~\ref{inf2h}). These seem intractable to draw explicitly. (A more detailed comparison of the two types of exotic planes appears at the beginning of Section~\ref{Questions}.) As applications, we show that many embedded surfaces, including all those arising from compact surfaces embedded in the 4-ball rel nonempty boundary, have infinitely many exotic cousins (Section~\ref{Surf}), and that the singularities of almost-smooth surfaces arising from Freedman theory as in Corollary~\ref{cpctSurf} typically require infinitely many local minima (Theorem~\ref{sing}, proved in Section~\ref{Taylor}). We also exhibit exotic planes with uncountable groups of symmetries that are not pairwise isotopic to the identity (Section~\ref{Group}).

\subsection{An exotic $\Z_2\oplus \Z_2$ action}\label{Branch1}
In \cite{menag}, the 4-sphere is exhibited as the union of two open subsets $R$ and $R^*$ (the former denoted $R'$ in \cite{menag}). Their intersection $U=R\cap R^*$  is homeomorphic to $S^3\times\R$, but has no smoothly embedded $S^3$ separating its ends. In particular, each of $R$ and $R^*$ is a small exotic $\R^4$. Let $N=R\# S^2\times S^2$, so $R^*\cup N$ is identified with $S^2\times S^2$ (Figure~\ref{U}). There is a standard action by $G=\Z_2\oplus \Z_2$ on $S^2\times S^2$ given in holomorphic affine coordinates for $S^2=\C P^1$ by the three involutions $r_y(u,v)=(\bar{u},\bar{v})$, $r_z(u,v)=(v,u)$ and $r_x(u,v)=(\bar{v},\bar{u})$. By construction, $U$, $N$ and $R^*$ are invariant under the action, although the connected sum decomposition of $N$ is not. Figure~\ref{N}(a) shows $N$ with the involutions given by $\pi$-rotation about the corresponding coordinate axes (so $r_x$ is $\pi$-rotation in the plane of the paper). To understand the figure, first  interpret the fine circles as 2-handles. These equivariantly cancel the 1-handles, resulting in a diagram of $S^2\times S^2$ with its $G$-action. Now imagine each fine 2-handle decomposed as a 1-handle and two 2-handles that are interchanged by $r_z$. (This could be drawn explicitly by adding a pair of balls to each fine circle where it intersects the $z$-axis.) For a fixed Casson handle $CH$, we can $G$-equivariantly embed a copy of $CH$ into each of the four resulting 2-handles. (Use the standard embedding of Section~\ref{Casson}, so each $CH$ is topologically (nonambiently) isotopic to the open 2-handle containing it.) Equivalently, we are $r_z$-equivariantly embedding a Casson handle $2CH$ (made by gluing together two copies of $CH$) into each of the two fine 2-handles in the figure. The manifold $N$ is obtained by replacing the fine 2-handles with these $r_z$-invariant Casson handles and removing the remaining boundary, so it is topologically (nonambiently) equivariantly isotopic to $S^2\times S^2$ minus the 4-handle. Its complement in $S^2\times S^2$, together with the topological open collar $U$ of $N$, is $R^*$. The $G$-actions on $N$ and $R^*$ are topologically standard. That is, up to homeomorphism, the action on $N$ is given by the standard action on $S^2\times S^2$ minus a fixed point, and the action on $R^*$ is obtained by rotating about the three coordinate axes of $\R^3$ in $\R^3\times\R$. To see $R$, surger the $S^2\times S^2$-summand out of $N$ (nonequivariantly) by replacing the 0-framing on one of the large curves by a dot. Note that since $N$ and $R$ are each made from a compact manifold by attaching Casson handles, they can be drawn explicitly (sometimes modulo the issue of the required complexity of $CH$). The manifolds and planes constructed from these can also be drawn (Section~\ref{Diagrams}). In contrast, $R^*$ is obtained by removing topological core disks or closed neighborhoods of these, obtained by infinitely iterated Freedman reimbedding (Section~\ref{AlmostSmooth}), so it seems an intractable problem to explicitly draw it or any derived exotic objects.

\begin{figure}
\labellist
\small\hair 2pt
\pinlabel {$R$} at 80 20
\pinlabel {$N$} at 130 8
\pinlabel {$R^*$} at 40 112
\pinlabel {$U$} at 47 65
\pinlabel {$S^3$} at 109 62
\pinlabel {$S^4\ \ \#\ \  S^2\times S^2$} at 105 103
\endlabellist
\centering
\includegraphics{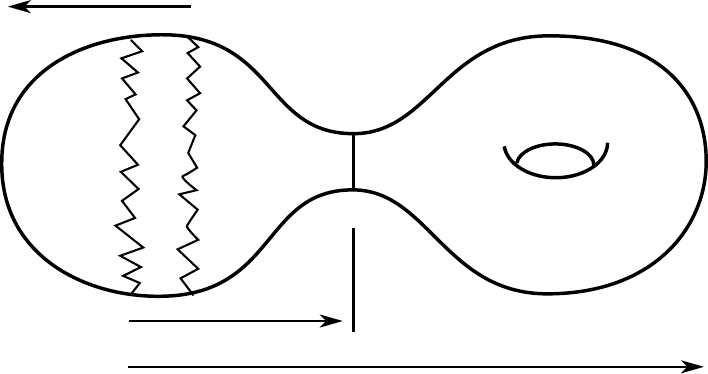}
\caption{A $\Z_2\oplus\Z_2$-invariant decomposition of $S^2\times S^2$ with $U$ an exotic $S^3\times \R$ at the end of two exotic $\R^4$-homeomorphs $R$ and $R^*$ in the (noninvariant) $S^4$-summand.}
\label{U}
\end{figure}
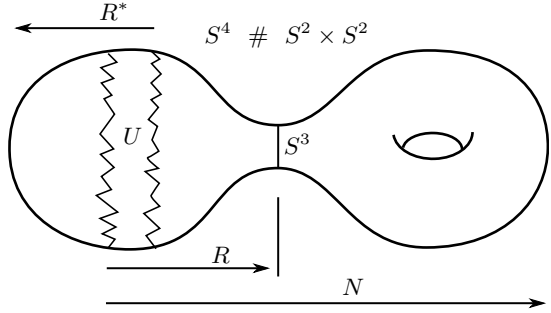

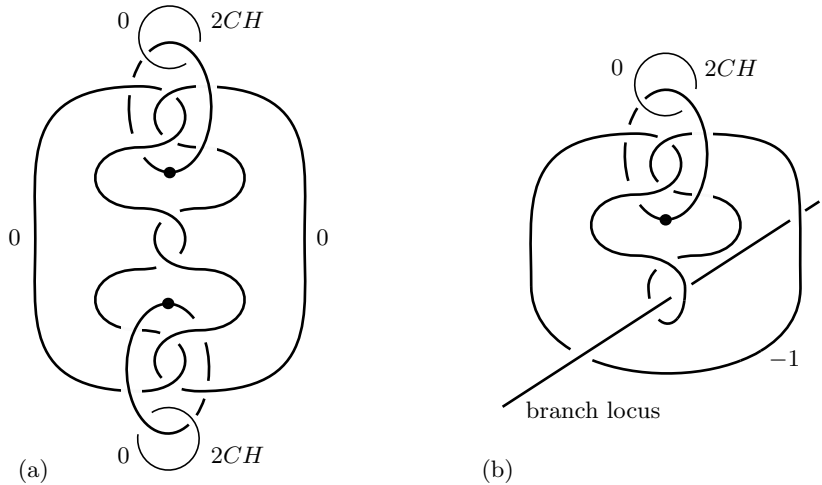
\begin{figure}
\labellist
\small\hair 2pt
\pinlabel {(a)} at 0 0
\pinlabel {(b)} at 175 0
\pinlabel {0} at -7 88
\pinlabel {0} at 109 88
\pinlabel {0} at 34 170
\pinlabel {0} at 34 6
\pinlabel {$2CH$} at 77 170
\pinlabel {$2CH$} at 77 6
\pinlabel {branch locus} at 211 23
\pinlabel {$-1$} at 283 42
\pinlabel {0} at 220 152
\pinlabel {$2CH$} at 263 152
\endlabellist
\centering
\includegraphics{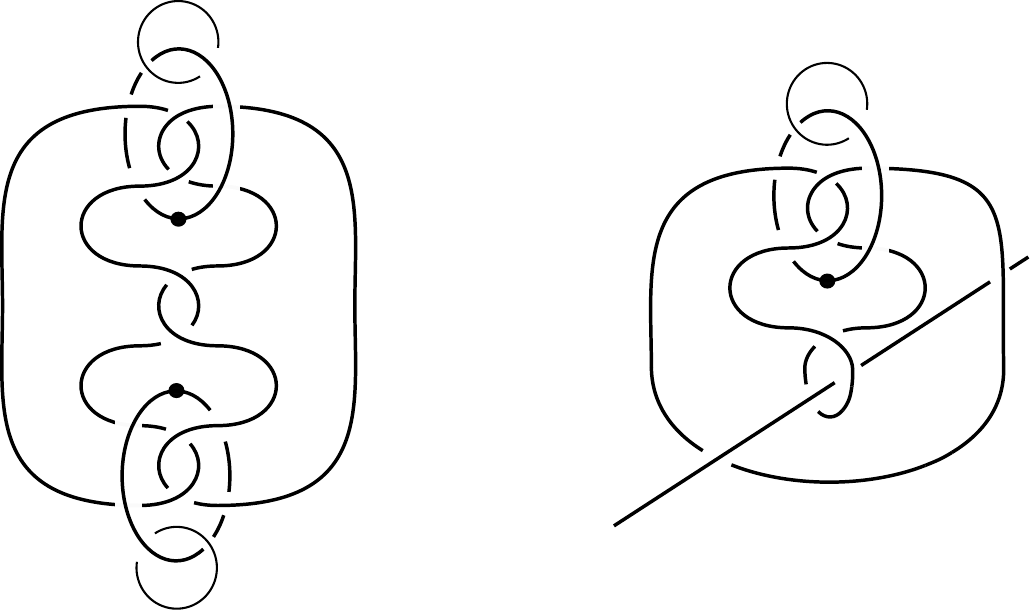}
\caption{The 4-manifold $N$ with its $G$-action (a) and its quotient $N_x$ (b). Putting a dot on one large 0-framed curve in (a) gives $R$, a small exotic $\R^4$ with a $G$-action of its end.}
\label{N}
\end{figure}

Exoticness of $R$ and $R^*$ follows from the original construction of $N$. For suitably chosen $CH$, $N$ embeds in the middle level of an h-cobordism between two nondiffeomorphic closed 4-manifolds $X_0$ and $X_1$ that are distinguished by gauge-theoretic invariants. (This construction goes back to Casson \cite{C} with a more complicated version of $N$, cf.\ also \cite[Section~9.3]{GS}.) The two large 0-framed curves represent the unique ascending and descending 2-sphere of the h-cobordism, with one extra pair of intersection points. Adding a dot to one curve surgers the corresponding sphere out of $N$, exhibiting copies of $R$ in $X_0$ and $X_1$, respectively. Each copy of $R$ contains a compact set $K$ (obtained from Figure~\ref{N} by removing the fine curves and dotting one 0-framed curve) such that the manifolds $X_i-\inter K$ are diffeomorphic. Then $R- K$ has no smooth $S^3$ separating its ends, for otherwise $X_0$ and $X_1$ would only differ by homotopy 4-sphere summands, so their invariants would agree. Thus, $U\subset R-K$ has no $S^3$ separating its ends, so $R$ and $R^*$ are exotic. Neither of the diffeomorphisms $r_x$ or $r_z$ can extend over $R$, since we could otherwise construct a diffeomorphism between $X_0$ and $X_1$. However, the diagram shows that $r_y$ does extend. By Bi\v zaca and the author \cite{BG}, we can take $CH$ to be the Casson handle $CH_+$ with a single positive double point in each kinky handle, or any refinement of it. (We will need to pass to an arbitrarily large compact subset of $N$ to obtain the embedding in an h-cobordism, but the rest of the discussion still applies.) By DeMichelis--Freedman \cite{DF} (as applied in \cite[Section~1]{menag} to the h-cobordism realizing the simple version of $N$ in Figure~\ref{N}), varying the thickness of $U$ as in the proof of Theorem~\ref{almostSmooth}, replacing $CH_+$ by suitable, highly ramified refinements of it, results in uncountably many diffeomorphism types of $R$ by an end-periodic h-cobordism argument. (One can similarly vary $R^*$, but distinguishing the resulting diffeomorphism types may require new tricks.)

The various quotients of the $G$-action are determined in \cite{menag}. We denote the quotients of $N$ by $r_x$ and by $G$ as $N_x$ and $N_G$, and similarly for the other quotients. The quotients other than for $r_x$ are standard: $R^*_y\approx R^*_z\approx R^*_G\approx N_y\approx N_G\approx\R^4$ and $N_z\approx{\C P^2}-\{p\}$. The induced involutions on the $\Z_2$-quotients, and hence their induced maps to $R^*_G$ and $N_G$ are also standard (complex conjugation in the case of $N_z$). The involution $r_y$ on $N=R\# S^2\times S^2$ restricts to a topologically standard involution on $R$, yielding a topologically standard branched covering $R\to\R^4$ of the standard $\R^4$ by a small exotic $\R^4$, as first observed by Freedman. (Surgering $N$ to $R$ does not change the quotient, but surgers the branch locus from a punctured torus to an exotic plane; see Section~\ref{DrawPlanes}.) The quotient $N_x$ is shown in Figure~\ref{N}(b) (where the $-1$-framing is the blackboard framing), and its complement in the quotient $\overline{\C P^2}$ of $S^2\times S^2$, extended by the topological collar $U_x$, is $R^*_x$. The figure exhibits $N_x$ as a Casson handle attached to $B^4$ along a $-1$-framed unknot. (The first stage is explicitly drawn and isotopic to the mirror image of Figure~\ref{double}(a) by symmetry of the Whitehead link, so it has a negative double point. The higher stages are given by $2CH$.) For $CH$ sufficiently ramified, $N_x$ embeds in $\overline{X}=\overline{\C P^2}\# k\C P^2$ representing the class $\alpha$ as in the proof of Theorem~\ref{CHclass} (with reversed orientation). Then $N_x$ is an exotic open Hopf bundle with no smoothly embedded sphere generating its homology. As in Section~\ref{R4}, $R_x^*$ then has the same end as a nondiagonalizable, positive definite 4-manifold (made from $\overline{X}$ by deleting a compact subset of $N_x$) so it is a large exotic $\R^4$ with a compact subset that cannot embed in any closed, positive definite 4-manifold, and $U_x=N_x\cap R^*_x$ has no 3-sphere separating its ends. Furthermore, $R^*_x$ lies in an uncountable family of pairwise nondiffeomorphic $\R^4$-homeomorphs obtained by varying the thickness of its topological collar $U_x$. The map $R^*_x\to R^*_G$ is a topologically standard branched covering map from a large exotic $\R^4$ to the standard $\R^4$. (In addition, the map $R^*\to R^*_x$ is a topologically standard branched covering from a small exotic $\R^4$ to a large $\R^4$, showing that the large/small dichotomy does not have a simple relationship with such branched covers. It is still an open question whether the standard $\R^4$ has such a map to an exotic one.) 

\subsection{Families of exotic planes}\label{planeFamilies}
We now have a supply of exotic but topologically standard double branched coverings of $\R^4$, whose branch loci in $\R^4$ must then be exotic planes. Recall from the end of Section~\ref{Sums} that double branched covering induces a monoid homomorphism from isotopy classes of topologically standard planes to $\R^4$-homeomorphs, which descends to their partially ordered quotient monoids (where the equivalence and partial order for planes are defined preceding Theorem \ref{order}). This immediately allows us to apply exotic $\R^4$ theory to exotic planes. We first prove Theorem~\ref{order}, constructing  uncountable families of exotic planes that are well-behaved under the partial order, then consider the end sum operation. Subsequent subsections expand these ideas to more general knotted surfaces and to exotic planes with group actions.

\begin{proof}[Proof of Theorem~\ref{order}]
When we vary $CH$ in Section~\ref{Branch1}, the resulting small $\R^4$-homeomorphs $R$ range over uncountably many diffeomorphism types. These realize uncountably many diffeomorphism types of ends since only countably many manifolds can have a given end. The branch loci of the corresponding involutions $r_y$ on $R$ must then realize uncountably many isotopy classes of exotic planes $P$ in $\R^4$ (which we call {\em simple}), determining uncountably many classes of exotic annuli. In Section~\ref{DrawPlanes}, we will see that these planes are all generated by unknots, which will complete the proof of Theorem~\ref{0g}(a). Embedding each $CH$ in a 2-handle $r_y$-equivariantly embeds $R$ in $\R^4$ without enlarging the branch locus, showing that $P\le\R^2$. Since $(\R^4,\R^2)$ embeds in every $(X,F)$ with $F\approx\R^2$, restricting to a diffeomorphism $\R^2\to F$, $P$ is equivalent to the standard plane as required to prove Theorem~\ref{order}(a).

While the large exotic $R_x^*$ lies in a pairwise nondiffeomorphic family parametrized by an interval, it requires more care to construct an uncountable family with quotients $R^*_G$ identified as $\R^4$. We expand the argument identifying $R^*_G$ in \cite{menag}. Let $C\subset CH$ be a compactum as in the proof of Theorem~\ref{almostSmooth}, so that $C$ is cellular in any 4-manifold  $X$ whose interior contains $CH$. We can now locate $R_x^*$ in Figure~\ref{N}(b) as the complement in $\overline{\C P^2}$ of a compactum obtained from a thinner version $H$ of the given handlebody (ignoring the fine circle) by attaching a copy of $C$ inside each copy of $CH$. (The proof in \cite{menag} used an almost-smooth core disk in place of $C$ but was otherwise the same.) The remaining involution is given in Figure~\ref{N}(b) by $\pi$-rotation about the $z$-axis and preserves $R^*_x$. Its branched covering map sends $\overline{\C P^2}$ to $S^4$ and each pictured handle of $H$ to a 4-ball attached to a previous 4-ball along a 3-ball in its boundary. (For example, the attaching region of the 2-handle is a solid torus covering  a 3-ball branched along a trivial pair of arcs.) The two copies of $CH$ are identified to a single copy, attached along half of its attaching region to the 4-ball $B$ comprising the image of $H$, so $R_G^*=S^4-(B\cup C)\approx S^4-C\approx S^4-\{p\}\approx\R^4$. Varying the parameter over $\Sigma$ as in Section~\ref{R4} now gives a $\Z_2$-invariant nested family of pairwise nondiffeomorphic, large exotic $\R^4$-homeomorphs $R^*_x$ whose quotients under the topologically standard involution are $\R^4$.

Theorem~\ref{order}(b) and (c) now follow from the exotic $\R^4$ theory discussed in Section~\ref{R4}. For $t\in \Sigma$, let $R_t$ be the corresponding large exotic $\R^4$, and let $P_t\subset\R^4$ be the corresponding branch locus. For $t<t'$, the inclusion $R_t\subset R_{t'}$ was constructed to have compact closure (for the periodic end argument). However, the fixed set of the involution on $R_x^*$ avoids the Casson handles $CH$ (although it intersects the 1-handle separating them in $2CH$) so after an equivariant isotopy we can assume the inclusion sends the fixed set of $R_t$ onto that of $R_t'$. This shows $P_t\le P_{t'}$. If also $s\le s'$ in $\Sigma$, we then have $P_s\natural \overline{P_t}\le P_{s'}\natural \overline{P_{t'}}$, where the bar denotes reversed ambient orientation. If $s>s'$ or $t>t'$, $R_s\natural \overline{R_t}$ cannot embed in $R_{s'}\natural \overline{R_{t'}}$, so the map from $\Sigma\times\Sigma$ to equivalence classes of exotic planes is an order-preserving injection, proving (b). If we replace the exotic planes $\overline{P_t}$ in this family by those coming from (a), the second summand no longer affects the resulting equivalence classes, which are then bijectively parametrized by $s\in\Sigma$, preserving order. For fixed $s$, the corresponding $\R^4$-homeomorphs still realize uncountably many diffeomorphism types \cite[Section~1]{menag} (since $R_s$ embeds in $\overline{\C P^2}$ so the end-periodic h-cobordism argument still applies, Section~\ref{R4}). Thus, we obtain uncountably many exotic planes in each equivalence class, proving (c).
\end{proof}

Sharpening our techniques, we now exhibit exotic planes with infinite order under end sum. While this is not surprising, it provides uncountable families that we will subsequently show have arbitrarily large $g^\infty$ (Corollary~\ref{ginf}). The same proof shows that some exotic planes have a certain rigidity, and shows in Section~\ref{Surf} that many proper 2-knots have infinitely many exotic cousins.

\begin{thm}\label{sum}
There is an exotic plane $P$ such that for the $n$-fold sums $P_n=\natural_n P$, the set $\{P_n|n=0,1,2,\cdots,\infty\}$ has the order type of its index set. In particular, the sums $P_n$ are all distinct. For every $n\in\Z^+$, there is a compact subset $K$ of $\R^4$ such that no pairwise self-diffeomorphism of $(\R^4,P_n)$ sends $K$ into $\R^4-K$. (This clearly fails for $n=0,\infty$.) There is an uncountable family of such planes $P$ indexed by $\Sigma$ such that, for each fixed $n\in\Z^+$, the resulting planes $P_n$ are all distinct, with order type given by $\Sigma$.
\end{thm}

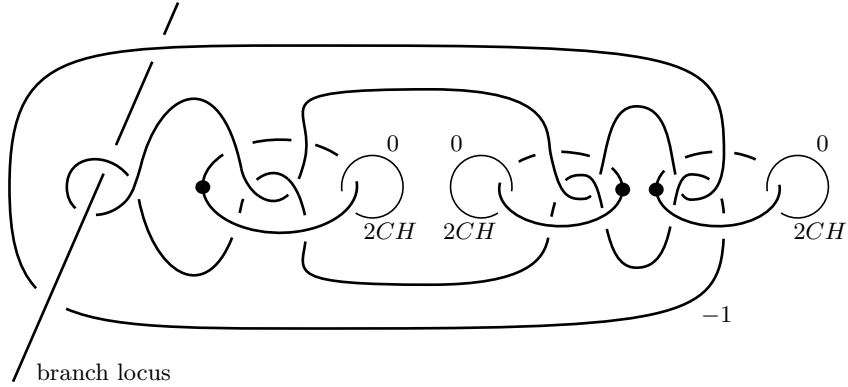
\begin{figure}
\labellist
\small\hair 2pt
\pinlabel {branch locus} at 36 4
\pinlabel {0} at 145 90
\pinlabel {0} at 169 90
\pinlabel {0} at 307 90
\pinlabel {$2CH$} at 144 57
\pinlabel {$2CH$} at 174 57
\pinlabel {$2CH$} at 306 57
\pinlabel {$-1$} at 267 25
\endlabellist
\centering
\includegraphics{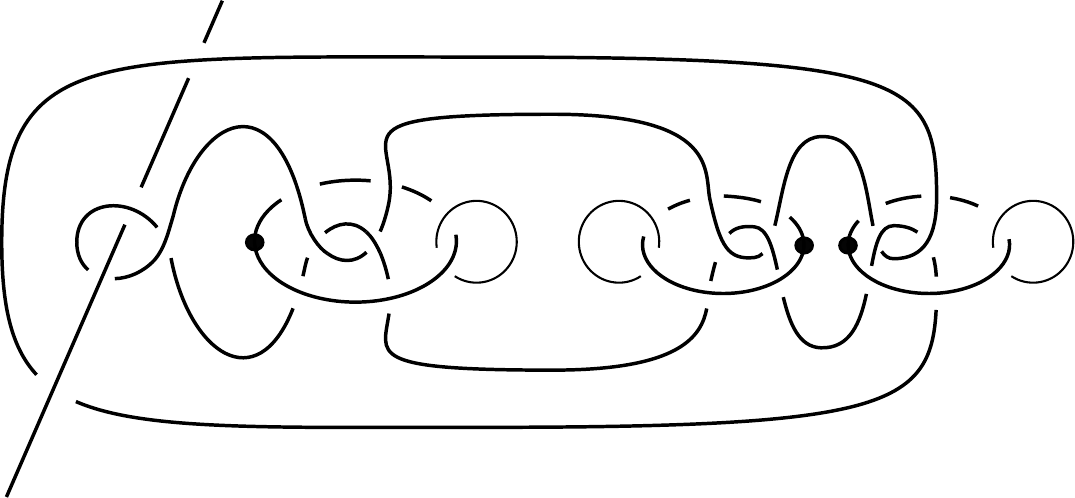}
\caption{A more complicated version of $N_x$.}
\label{Nx}
\end{figure}

\begin{proof}
This follows from the analogous results for $\R^4$-homeomorphs. First, we construct a variation of $R^*_x$. Let $X'=\overline{\C P^2}\# 16\C P^2$. By Freedman's classification \cite{F}, $X'$ splits topologically as $\overline{\C P^2}\# Y$, where $Y$ is a topological manifold with an even, positive definite intersection form of rank 16. (Note that the intersection form of the latter sum is isomorphic to that of $X'$ since both forms are odd and indefinite with the same $b_\pm$.) The latter splitting exhibits a topological $\C P^1$ whose tubular neighborhood can be smoothly exhibited in $X'$ as a Casson handle attached to a 4-ball (as in the proof of Corollary~\ref{cpctSurf}). After refinement, we may assume this neighborhood is diffeomorphic to $N_x$ from Section~\ref{Branch1}. The only complication is that the first stage of the Casson handle, as exhibited in Figure~\ref{N}(b), may require additional double points. We may assume, after further refinement, that these occur in pairs of opposite sign as in Figure~\ref{Nx} (whose first stage is isotopic to the mirror image of Figure~\ref{double}(c)). This similarly extends the diagram of the double branched cover in Figure~\ref{N}(a), which is in turn realized by making finger moves in the middle level of the h-cobordism. The discussion in Section~\ref{Branch1} then applies as before (for $CH$ sufficiently ramified) with the end of $R^*_x$ agreeing with that of an {\em even}, positive definite, simply connected 4-manifold. Since $R_x^*$ is embedded with compact closure in a larger exotic $\R^4$ in $\overline{\C P^2}$, it lies in a compact submanifold $Q$ of $\overline{\C P^2}$ whose double $Z=Q\cup_{\id_\partial}\overline Q$ is a closed, spin 4-manifold containing $R_x^*$. It follows that the manifolds $R_n=\natural_n R_x^*$ for $n=0,1,2,\cdots,\infty$ are ordered like their index set: If $R_m$ embeds in $R_n$ with $m>n$, then $R_n$ contains disjoint copies of $R_n$ and $R_1$. Iterating, we can find arbitrarily many disjoint copies of $R_1$ in $R_n\subset \# _nZ$. Since the (simply connected) end of $R_1=R^*_x$ agrees with that of a positive definite spin manifold, cutting and pasting gives a closed, spin 4-manifold with $b_-=nb_-(Z)$ fixed and $b_+$ arbitrarily large, contradicting a theorem of Furuta \cite{Fu}. By the previous proof, $R_x^*$ varies over a $\Z_2$-invariant, $\Sigma$-indexed family whose quotients are  diffeomorphic to $\R^4$. (Replace $B\cup C$ by a boundary sum $C_k=B\natural kC$, where  Figure~\ref{Nx} has $k$ copies of $2CH$. This is still cellular.) For each $n\in\Z^+$, the corresponding family of manifolds $R_n$ is nested with the order type of $\Sigma$. (Otherwise we could put a periodic end on a nondiagonalizable, definite manifold constructed as an $n$-fold end sum).

The theorem now follows easily. Set $P$ equal to the branch locus in $R^*_G\approx\R^4$ generating $R^*_x$. For $m\le n$ we have $\R^2\le P_{n-m}$, so $P_m\approx P_m\natural\R^2\le P_m\natural P_{n-m}=P_n$. For $m>n$ this inequality cannot hold since double branched covering preserves order. The $\Sigma$-ordered families follow similarly. (As in the previous proof, the branch locus of $R^*_x$ avoids the copies of $CH$.) To construct $K$ of the theorem, note the previous paragraph shows that for fixed $n\in\Z^+$, there is a finite upper bound on the number of disjoint copies of $R_1$ that can be simultaneously embedded in $R_n$. Let $K'\subset R_1$ be a compact submanifold containing a slightly smaller exotic $\R^4$ whose end still agrees with the end of an even definite manifold. Then the same argument bounds the number of disjoint copies of $K'$ in $R_n$. Let $\tilde{K}$ be a compact submanifold of $R_n$ containing the maximal number of copies of $K'$. Then $\tilde{K}$ cannot embed in $R_n-\tilde{K}$. Let $K\subset\R^4$ be the image of $\tilde{K}$ under the branched covering. Then no diffeomorphism of $(\R^4,P_n)$ can send $K$ into $\R^4-K$.
\end{proof}

\subsection{Knotted surfaces}\label{Surf}
The proof of Theorem~\ref{sum} also shows that many other surfaces in 4-manifolds are topologically isotopic to infinitely many distinct embeddings. In fact, we will see (Corollary~\ref{slice}) that this holds for the interior of any compact surface embedded rel nonempty boundary in $B^4$. Recall from Section~\ref{Sums} that an end sum $(X,F_1)\natural (\R^4,F_2)$ can naturally be written as $(X,F_1\natural F_2)$, uniquely up to smooth isotopy if each $F_i$ has a unique end and finite genus. If $F_i$ has several ends, we may need to choose one to specify the isotopy (or diffeomorphism) type, and if it has infinitely many ends or infinite genus, we may need to specify the defining ray in $F_i$.

\begin{thm}\label{2knot}
Let $F\subset X$ be a noncompact surface embedded in a 4-manifold, with a double branched cover $\tilde{X}$ embedding in a compact 4-manifold $W$ (possibly with boundary).

(a) If $W$ is spin, then for any family $\{P_n|n=0,1,2,\dots,\infty\}$ as constructed in Theorem~\ref{sum}, the pairs $(X,F_n)=(X,F)\natural(\R^4,P_n)$ (using fixed auxiliary data for $F$ if needed) are all topologically isotopic to $(X,F)$, but no two pairs are diffeomorphic.

(b) If $W$ is instead closed, simply connected and definite, then $(X,F)$ is topologically isotopic to uncountably many nondiffeomorphic pairs of the form $(X,F_t)=(X,F)\natural(\R^4,P)$, where $P$ (up to orientation) varies over simple exotic planes as constructed in the previous sections for Theorem~\ref{order}(a).
\end{thm}

\begin{proof}
For (a), assume after doubling if necessary that $W$ is closed. The proof of Theorem~\ref{sum} (with an extra $\tilde{X}$ summand) shows that for $m>n$, there is no embedding $(X,F_m)\emb(X,F_n)$ sending $F_m$ onto $F_n$, so the pairs are pairwise nondiffeomorphic. (If the original embedding $\tilde{X}\subset W$ has a sufficiently complicated set-theoretic boundary, we may have to improve the embedding by deleting a tubular neighborhood of a ray from $\tilde{X}$. This creates a new embedding of $\tilde X$ with some smooth 3-manifold boundary so that we may construct the end sum $\tilde X\natural R_n$ inside the spin manifold $W\# nZ$.) For (b), reverse orientation on $X$ if necessary so that $W$ is negative definite and orient the planes $P$ as usual. Then the double branched covers (for any choices of auxiliary data) have the form $\tilde{X}\natural R$ with $\tilde{X}\subset W$ and $R$ a small exotic $\R^4$ as in Section~\ref{Branch1}. These represent uncountably many diffeomorphism types by \cite[Lemma~7.3]{MinGen}. (That lemma followed from the end-periodic h-cobordism method with a negative definite end, cf.~Section~\ref{R4}. It used a slightly different version of $R$, but the difference does not affect its proof.)
\end{proof}

\begin{Remarks}\label{2knotRem}
The above proof actually shows that (b) remains true under the weaker hypothesis that each compact, codimension-0 submanifold of $\tilde{X}$ lies in some $W$ as given. However, the simple connectivity hypothesis is essential for the proof since a nontrivial $\pi_1(W)$ would show up in the end of the associated end-periodic manifold, obstructing the required theorem from gauge theory. (To see a typical difficulty, note that the interior of the $E_8$-plumbing violates the Periodic End Theorem if we drop simple connectivity of the end.) Note that (a) only requires enough auxiliary data to determine $\tilde X\natural R_n$. For example, no such data is needed if $\tilde X$ is simply connected at infinity (or its unique end is Mittag-Leffler \cite{CG}).
\end{Remarks}

\begin{cor}\label{slice}
Suppose $F$ is a compact (orientable) surface embedded (smoothly) in the 4-ball with $F\cap\partial B^4=\partial F\ne\emptyset$. Then there are infinitely many surfaces in $\inter B^4\approx\R^4$ topologically but not smoothly isotopic to $\inter F$.
\end{cor}

\begin{proof}
It suffices to show that the double cover of $B^4$ branched along $F$ is spin, for then (a) of the theorem applies. The spin structure on $B^4$ lifts to the double cover $\tilde{E}$ of $E=B^4-F$. This spin structure does not extend over the lifted branch locus $\tilde F$. However, $H_1(E)$ is $\Z$ (since $F$ is orientable) with $\pi_1(\tilde{E})$ mapping onto $2\Z$, so there is an element of $H^1(\tilde{E};\Z_2)$ pairing nontrivially with the meridian of $\tilde F$. This modifies the spin structure on $\tilde{E}$ so that it does extend.
\end{proof}

There are also nontrivial slice disks in $B^4$ whose interiors are topologically isotopic to uncountably many embeddings $\R^2\emb\R^4$ by (b), for example with $W$ the double of a contractible branched double cover. These results suggest the utility of separating the study of smooth proper 2-knots in general from that of exotic planes. Recall that the topologically standard proper 2-knots comprise a submonoid of all proper 2-knots (as well as of various other monoids of embedded surfaces). Thus, the submonoid has an ``action" whose orbit space is obtained by calling two proper 2-knots equivalent if they become isotopic after end sum with suitable exotic planes. The orbit space has a well-defined forgetful map into the monoid of topological proper 2-knots up to topological isotopy. This is surjective by Proposition~\ref{smoothing}.

\begin{ques}\label{modplanes}
How far is this forgetful map from being injective?
\end{ques}

\noindent See also Questions~\ref{univ}. The question can similarly be formulated for all noncompact pairs $(X,F)$.

\subsection{Genus at infinity and the Taylor invariant}\label{Taylor}
The proof of Theorem~\ref{sum} allows us to more deeply understand exotic annuli and singularities of surfaces, including the structure of their radial functions, using $g^\infty$ and the {\em Taylor invariant} \cite{Ta}. For simplicity, we use the following variant of the latter:

\begin{de}\label{gammaStar}
For a spin 4-manifold $V$, define $\gamma^*(V)\in\Z^{\ge0}\cup\{\infty\}$ to be the smallest $b$ such that every compact, codimension-0 submanifold $Q$ of $V$ embeds in a closed, spin 4-manifold with $b_+ = b_- \le b$.
\end{de}

\noindent We will not need the actual Taylor invariant $\gamma(V)$, whose definition is more technical but applies to all 4-manifolds. (It focuses on those compact subsets $Q$ that lie in suitable 4-balls topologically embedded in $V$.) But it is immediate from the definitions that $\gamma(V)\le\gamma^*(V)$ with equality whenever $V$ is homeomorphic to $\R^4$, so we will use the two interchangeably in the latter case. We immediately obtain some useful properties:

\begin{prop}\label{taylor}
a) The invariant $\gamma^*$ is nondecreasing under inclusion and subadditive under end sum.
\item[b)]  When $\gamma^*(V)$ is finite, there is a compact $Q\subset V$ such that every open $U$ with $Q\subset U\subset V$ has $\gamma^*(U)=\gamma^*(V)$. 
\item[c)] For the double branched covers $R_n$ of the exotic planes $P_n$ of Theorem~\ref{sum}, $\gamma(R_\infty)$ is infinite, while the other values $\gamma(R_n)$ are finite and nonzero for $n\ne0$ but become arbitrarily large as $n$ increases.
\end{prop}
 
\noindent Note that (c) again distinguishes infinitely many diffeomorphism types in $\{R_n\}$ and hence in $\{P_n\}$. 

\begin{proof}
a) Nondecreasing behavior is clear. For subadditivity, note that an end sum is exhausted by boundary sums of compact subsets of the summands.

b) Since $\gamma^*(V)$ is finite, there is a $Q$ that admits no embedding as in the above definition with $b<\gamma^*(V)$. Any such $Q$ works by (a).

c) We return to the proof of Theorem~\ref{sum}, which exhibits $R_n$ as $\natural_n R_x^*$, with $R_x^*$ embedded in a closed, spin manifold $Z$ with $b_+=b_-$. The embedding shows that $\gamma(R_x^*)\le b_+(Z)$ is finite, as is $\gamma(R_n)$ (by subadditivity). To see that $\gamma(R_n)$ takes arbitrarily large values, choose $Q\subset R_x^*$ large enough to contain an exotic $\R^4$ with the same end as an even, definite 4-manifold. Then for any finite $b$, Furuta's Theorem gives an upper bound on the number of disjoint copies of $Q$ that can be found in a spin manifold as in Definition~\ref{gammaStar}, forcing $\gamma(R_n)$ to increase without bound. Now $\gamma(R_n)>0$ for all finite $n>0$ by subadditivity, and $\gamma(R_\infty)$ is infinite since $\gamma^*$ is nondecreasing under inclusion.
\end{proof}

We can now analyze $g^\infty(P_n)$ via the following:

\begin{prop}\label{gTaylor}
Let $F\subset\R^4$ be a surface whose end is annular, and let $V$ be the double branched cover. Then $\gamma(V)\le\gamma^*(V)\le g(F)+g^\infty(F)$.
\end{prop}

\begin{proof}
Since $V$ is spin (as in the proof of Corollary~\ref{slice}), $\gamma^*(V)$ is defined and satisfies the first inequality. It now suffices to assume $g^\infty(F)$ is finite. Given a compact submanifold $Q$ of $V$, we can modify $(\R^4,F)$ outside of the image of $Q$ to get a closed surface $\hat{F}$ in $S^4$ with genus $g(\hat{F})=g(F)+g^\infty(F)$. The double branched cover $Y$ of $S^4$ along $\hat{F}$ is the required closed, spin manifold containing $Q$ and with $b_+ = b_- = g(\hat{F})$: Since $\hat{F}$ is orientable, it has a Seifert hypersurface. Pushing its interior into the 5-ball and double covering shows that $Y$ bounds so has signature 0. The equalities for $b_\pm$ then follow immediately from Hsiang--Szczarba \cite[Theorem~3.2]{HS}.
\end{proof}

\begin{cor}\label{ginf}
The exotic planes $P_n=\natural_n P$ of Theorem~\ref{sum} (with $P$ fixed) realize infinitely many values of $g^\infty$. Varying $P$ realizes each of infinitely many values of $g^\infty$ by uncountably many exotic planes.
\end{cor}

\noindent This again shows that no two of the exotic planes $P_n$ are isotopic (for each fixed $P$), since otherwise the monoid structure would show that there were only finitely many isotopy classes.

\begin{proof} We  show below that $g^\infty(P)$ is finite. It follows that $g^\infty(P_n)$ is finite for all $n<\infty$. (More generally, $g^\infty$ is subadditive on end sums.) Since $\gamma(R_n)$ takes arbitrarily large values, the proposition then shows that $g^\infty(P_n)$ takes infinitely many values. Fixing $n\in\Z^+$ and varying $P$ as in Theorem~\ref{sum}, we obtain uncountably many diffeomorphism types of the form $R_n$. For sufficiently large parameter values in $\Sigma$, these all have the same value $\gamma$ of $\gamma(R_n)$ (Proposition~\ref{taylor}(b)). The corresponding uncountable family of exotic planes $P_n$ all have $g^\infty(P_n)\ge \gamma$, so at least one value of $g^\infty\ge\gamma$ is realized by uncountably many such planes. Since increasing $n$ makes $\gamma$ arbitrarily large, the last sentence of the corollary follows.

To verify finiteness of $g^\infty(P)$, recall that the fixed set $\R P^2$ of our involution on $\overline{\C P^2}$ (rotation about the $y$-axis in Figure~\ref{Nx}) is disjoint from the Casson handles $CH$ (although it intersects the 1-handles between them inside each $2CH$), and that the image $R_G^*$ of $R^*_x$ in the quotient $S^4$ is a standard $\R^4$ complementary to a cellular set $C_k=B\natural kC$ and intersecting $\R P^2$ in $P$. Then $P$ is the interior of a compact disk $D^*$ in $\R P^2$, with $\partial D^*$ lying in the manifold part of $\partial C_k$. Since $\inter C_k$ is simply connected, we can cap off $D^*$ with a compact surface $F\subset C_k$ of some genus $g$. By cellularity, $R_G^*=S^4-C_k$ is diffeomorphic to $S^4-\{p\}$ with $p\in\inter C_k- F$, fixing a preassigned compact subset of the domain. Pulling back $D^*\cup F$ to $R_G^*$ gives a genus-$g$ surface in $\R^4$ agreeing with $P$ on a preassigned compact subset, showing $g^\infty(P)\le g$.
\end{proof}

The Taylor invariant gives yet another way to analyze the complexity of annuli in $\R^4$, by considering local maxima and superlevel sets $r^{-1}[a,\infty)$ of a radius function.

\begin{thm}\label{infLocMax}
Let $F\subset\R^4$ be a surface whose end is annular, with $\pi_1(\R^4-F)$ finitely generated, and with double branched cover $V$. If $\gamma^*(V)>2g(F)$ then the distance function to a generic point of $\R^4$ must have infinitely many local maxima on $F$.
\end{thm}

\noindent It is unclear to the author whether the $\pi_1$-hypothesis is necessary. However, it is automatically satisfied in our case of interest, when the annulus at infinity is topologically standard (so compactifying gives a locally flat surface in $S^4$).

\begin{cor}\label{inf2h}
Any level diagram (using the radius function) of an exotic plane from Theorem~\ref{order}(b,c) or \ref{sum} requires infinitely many local maxima.
\end{cor}

\begin{proof}
The double branched cover $V$ of any of these has a compact submanifold $Q$ (in some $R_x^*$-summand) that does not embed in a closed 4-manifold with $b_-=0$, so $\gamma^*(V)>0$.
\end{proof}

\noindent Scholium~\ref{infEnd} below gives exotic planes satisfying the even stronger condition that the number of components of the regular superlevel sets must become arbitrarily large with increasing radius. In contrast, we show in Section~\ref{Diagrams} that the simple exotic planes from Theorem~\ref{order}(a) have diagrams with no local maxima, so their superlevel sets must be connected.

\begin{proof}[Proof of Theorem~\ref{infLocMax}]
The level diagram of $F$ determined by the distance function can be used to construct a handlebody whose interior is $V$. One approach is to first construct the complement of $F$. This is obtained from a 0-handle by adding a $(k+1)$-handle for each $k$-handle of $F$ (e.g.\ \cite[Section~6.2]{GS}). If $F$ has only finitely many local maxima, then its complement will have only finitely many 3-handles. Since $V$ is made by double covering and adding a 2-handle and $2g(F)$ 3-handles, it will also have only finitely many 3-handles. Suppose $Q$ is a compact subhandlebody in $V$ containing all of the 3-handles. Since $Q$ embeds in its double $DQ$, which is closed and spin with signature 0, it suffices to show $$\frac12 b_2(DQ)=b_2(Q)\le b_2(V)=2g(F).$$ 

The first two relations are essentially Taylor's proof \cite[Theorem~4.3]{Ta} that $\gamma(V)\le b_2(V)$ for any 4-manifold with finitely many 3-handles. We simplify the details by avoiding the delicate nonspin case. The inequality follows immediately since $V$ is built from $Q$ without 3-handles. The obvious retraction $DQ\to Q$ splits the long exact sequence of the pair to give
$$H_k(DQ)\cong H_k(Q)\oplus H_k(DQ,Q)$$
for each $k$. For $k=2$, the last term is isomorphic to $H_2(Q,\partial Q)\cong H^2(Q)$, giving the first equality.

The last equality was essentially proved by Hsiang and Szczarba \cite{HS} in the case of a topologically standard end (equivalently for closed manifolds). The general case follows the same method: We wish to show $b_1(V)=0=b_3(V)$, for then an easy Euler characteristic computation completes the proof. The second of these equalities follows from the first -- immediately from duality in the closed case of \cite{HS}, and with a bit more work in our setting: If $b_3(V)\ne0$, then there is some compact $Q'\subset V$ with connected boundary, carrying a class $\alpha\in H_3(Q')$ that maps nontrivially into $V$. This group injects into $H_3(DQ')$ by the above splitting. Thus, some $\beta\in H_1(DQ')$ intersects $\alpha$ nontrivially. Since $\partial Q'$ is connected, we can assume $\beta\in H_1(Q')$. Since $\beta\cdot\alpha\ne0$, it follows that $\beta$ has infinite order in $H_1(V)$, so $b_1(V)\ne0$. To show that $b_1(V)$ vanishes, let $Y$ be obtained from the exterior of $F$ in $\R^4$ by adding a 2-cell along a curve wrapping twice around the meridian. Then $\pi_1(Y)$ is finitely generated by hypothesis, and $H_1(Y)\cong\Z_2$. The double cover of $Y$ has $\pi_1(\tilde{Y})\cong\pi_1(V)$ an index-2 subgroup of $\pi_1(Y)$. Apply \cite[Lemma~4.1]{HS}: Since $\pi_1(Y)$ is finitely generated with finite cyclic $H_1$, and $\pi_1(V)$ has prime power index in $\pi_1(Y)$ with abelian quotient, we conclude that $H_1(V)$ is finite.
\end{proof}

We now exhibit exotic planes for which the number of components of the superlevel sets must become arbitrarily large:

\begin{sch}\label{infEnd}
There are uncountably many exotic planes $P\subset\R^4$ with $g^\infty(P)=\infty$, such that for each $m\in\Z^+$, there is a compact $K\subset\R^4$ for which every $P$-transverse integral homology ball $B\subset\R^4$ containing $K$ has complement intersecting $P$ in at least $m$ components. Any annulus (or surface with annular end) in $\R^4$ inherits these same properties after end sum with $P$.
\end{sch}

\begin{proof}
Let $P$ be any exotic plane whose double branched cover $V$ has infinite Taylor invariant, so $g^\infty(P)=\infty$ by Proposition~\ref{gTaylor}. For example, we can obtain uncountably many of these starting with some $P_\infty$ from Theorem~\ref{sum} (so $V=R_\infty$) and applying Theorem~\ref{2knot}(b) with $(X,F)=(\R^4,P_\infty)$, augmented by the first sentence of Remarks~\ref{2knotRem}. Given $m$, we can choose $K$ so that any $B$ containing $K$ has double branched cover $\tilde{B}\subset V$ that cannot embed in any closed, spin 4-manifold  with $b_+=b_-<m$. Since $P$ is a plane, the number of components of $P-B$ equals $b_1(P\cap B)+1$. To understand $\tilde B$, create a connected surface $(F,\partial F)\subset (B,\partial B)$ from $P\cap B$ by connecting its components near $\partial B$ using the minimal number of 1-handles. Then $b_1(F)=b_1(P\cap B)$.  Applying \cite{HS} as in the proof of Theorem~\ref{infLocMax}, we see that the double cover $Q$ of the integral ball $B$ branched along $F$ has $b_2(Q)=b_1(F)$ (the right side replacing $2g(F)$ in the previous argument) and contains a copy of $\tilde{B}$, so $m\le b_\pm(DQ)=b_2(Q)=b_1(F)=b_1(P\cap B)$. The first sentence of the theorem follows immediately. A similar argument (with $m$ shifted) applies to $A\natural P$ for any annulus $A\subset \R^4$, once we fill $A$ by a compact surface $F_0$ that we choose $K$ to contain. We also have $g^\infty(A\natural P)=\infty$ since a genus-$g$ surface capping $A\natural P$ near infinity gives a cap for $P$ of genus $g+g(F_0)$.
\end{proof}

\begin{Remark}
Other results from exotic $\R^4$ theory descend similarly to exotic planes. For example, if $R$ is an exotic $\R^4$ containing some $R_x^*$ as in Section~\ref{Branch1}, it has a compact subset that cannot be enclosed by a rational homology sphere. (It has the same end as a nondiagonalizable definite manifold. Capping the latter by the rational ball cut out in $R$ would contradict Donaldson's Theorem.) Thus, any exotic plane from Theorem~\ref{order}(b,c) or \ref{sum} has an associated compact subset of $\R^4$ that is not enclosed by any 3-manifold with corresponding double branched cover a rational homology sphere. This shows yet again that such planes cannot be standard near infinity, for otherwise there would be large 3-spheres covered by $S^3$.
\end{Remark}

We can now complete the discussion, begun preceding Theorem~\ref{sing}, of almost-smoothing topologically embedded surfaces. As we have seen (Corollary~\ref{cpctSurf}), a compact, locally flat $F\subset X$ is always topologically ambiently isotopic  to a surface that is smooth except at a unique point $p$, and \cite{MinGen} allows control of $g$ and $\kappa$ of the singularity. We wish to understand the possible local level diagrams centered at $p$ (up to almost-smooth isotopy as defined before Theorem~\ref{sing}).

\begin{proof}[Proof of Theorem~\ref{sing}]
First we dispense with the exceptional case of obtaining a singularity with $g=0$, where $F$ is initially smooth by hypothesis. Theorem~\ref{0g}(a) exhibits exotic planes in $\R^4$ determining exotic annuli with $g^\infty=0$. Explicit diagrams of these planes without local maxima are given by Theorem~\ref{drawPlanes}. One-point compactifying any of these gives a topologically unknotted, almost-smooth sphere in $S^4$ whose singularity has $g=0$ and lacks local minima but is not smoothable by almost-smooth isotopy. Connected summing $F$ with such a sphere gives the required almost-smooth surface.

For the remaining cases, a smooth $F$ can be topologically isotoped to have a unique singularity, with any nonzero (finite or infinite) $\kappa$ and $g=\max\{\kappa_\pm\}$, using the one-point compactifications of the exotic planes of Theorem~\ref{ginfty}. If $F$ is not smooth we almost-smooth it as in Corollary~\ref{cpctSurf}, with the proof of Theorem~\ref{ginfty} (Section~\ref{CH}) realizing any $\kappa$ and $g$ as before with $\kappa_\pm$ sufficiently large. Either way, the singularity comes from a core of some Casson handle $CH_F$. By construction, the numbers of first-stage double points of $CH_F$ with each sign cannot exceed the desired $\kappa_\pm$, but the higher stages of $CH_F$ can be chosen with arbitrarily large ramification. It now suffices to show that in such a sufficiently ramified Casson handle $CH_F$, every almost-smooth core obtained by Freedman's construction (Theorem~\ref{almostSmooth}) requires infinitely many local minima, and to arrange our desired intersection condition of (b) of the theorem. For the latter, after almost-smoothing $F$, delete its singular point $p$ from a 4-ball neighborhood of $p$, then invert to obtain an exotic annulus in $\R^4$. End-sum this with $P$ from Scholium~\ref{infEnd} and fill $p$ back in. The resulting surface is still topologically isotopic to $F$, but has $g=\infty$, and the number of components of its intersection with small homology balls increases without bound as required. (The complement of a homology ball in $S^4$ is again a homology ball.)

To show that every Freedman core disk $D$ of $CH_F$ requires infinitely many local minima, we (indirectly) compare with an exotic plane $P$ requiring infinitely many local maxima. We arrange the latter property as for Corollary~\ref{inf2h}, by choosing $P$ with double branched cover $R_x^*$ as in Section \ref{Branch1}. We wish to arrange $CH_F$ to be the same Casson handle $CH$ that appears (twice) in Figure~\ref{N}(b). We are not allowed to refine the first stage of $CH_F$, but the topological Hopf bundle $N_x$ embeds as required in the almost-definite $\overline{X}$ from the proof of Theorem~\ref{CHclass} as long as the first stage of $CH$ (in the second stage of $N_x$) has a negative double point and the higher stages are sufficiently ramified \cite{D5T}. As in Section~\ref{R4}, this shows the corresponding $R_x^*$ has $\gamma>0$ as required, so we can assume  $CH=CH_F$ after possibly reversing orientation and refining both. (If $CH$ has no positive first-stage double point, we lose the proof that the double branched cover $R^*$ of $R_x^*$ is exotic, but this is not presently needed.) The core disk $D\subset CH_F= CH$ is constructed as in the proof of Theorem~\ref{almostSmooth} by shrinking a suitable cellular subset $C$. As in the proof of Theorem~\ref{order} (Section~\ref{planeFamilies}), the quotient of $R^*_x$ is given by $R^*_G=S^4-(B\cup C)\approx  S^4-\{p\}\approx \R^4$. As in the proof of Corollary~\ref{ginf}, the branch locus $\R P^2$ intersects $R^*_G$ in $P$ and intersects $\partial (B\cup C)$ transversely in a knot $K$ (bounding $P$) in its 3-manifold region. (This can be constructed explicitly by taking the quotient of Figure~\ref{N}(b). The branch locus in $N_x$ intersects $2CH$ only in a 2-dimensional 1-handle, with the same core as the 4-dimensional 1-handle attached to $H$ that separates the two copies of $CH$. In the quotient $N_G$, this contributes a pair of arcs to $K$, running along the free half of the attaching region of $CH$.) The knot $K$ lies in the attaching region of the Casson handle whose closure is $B\cup C$, so after we extend the core $D$ across $B$, $K$ lies in an $S^1\times D^2$ tubular neighborhood of $\partial  D$ in $\partial  (B\cup C)$. Thus, the diffeomorphism $S^4- (B\cup C)\to S^4-\{p\}$, which created $D$, also sends the exotic plane $P\subset R^*_G$ to a surface that near $p$ is given by a satellite on $D-\{p\}$ with pattern $[0,\infty)\times K\subset[0,\infty)\times S^1\times D^2$. If $D$ had only finitely many local minima in its radial function, then some subdisk $D'\subset D$ containing $p$ would have none. Then $D'-\{p\}$ in $S^4-\{p\}=\R^4$ would have no local maxima, and hence $P$ would have only finitely many, contradicting our choice of the latter.
\end{proof}

\begin{sch}\label{ginftySch}
The exotic planes of Theorem~\ref{ginfty} with nonzero $g^\infty$ and $\kappa^\infty$ require infinitely many local maxima (assuming sufficient ramification in the construction).
\end{sch}

\begin{proof} This follows immediately from the previous proof, since the ends of these planes are constructed by puncturing Freedman core disks with arbitrarily large ramification above the first stage.
\end{proof}

\subsection{Group actions}\label{Group}
We now investigate symmetries of exotic planes. The main theme continues to be that we can extract information from exotic $\R^4$ theory. In this case, we consider the author's paper \cite{groupR4}, which constructed exotic $\R^4$-homeomorphs admitting uncountable group actions that inject into the mapping class group, and similar inextendable group actions at infinity. By expanding the theory to manifold pairs, we will obtain similar actions on exotic planes. For a pair $(X,F)$, let $\D(X,F)=\pi_0(\Diff_+(X,F))$ denote the group of pairwise isotopy classes of pairwise self-diffeomorphisms preserving both orientations. We write $\D(X)$ if $F$ is empty and $\D(F)$ if $X=\R^4$. Similarly, we consider pairwise group actions at infinity through ``germs'' of pairwise diffeomorphisms:

\begin{de}
A {\em diffeomorphism at infinity} of $(X,F)$ is a pairwise, proper embedding into $(X,F)$ of the closed complement of a compact subset of $F$, up to enlarging the latter. Two such diffeomorphisms are {\em isotopic} if they are pairwise, properly isotopic after a sufficiently large compact subset is removed from their domains.
\end{de}

\noindent Diffeomorphisms at infinity form a group under the obvious notion of composition. An {\em action at infinity} of a group $G$ is a homomorphism from $G$ into this group. Let $\D^\infty(X,F)$ denote the group of isotopy classes of diffeomorphisms at infinity. There is an obvious forgetful homomorphism $\D(X,F)\to\D^\infty(X,F)$. It can be shown as in \cite{groupR4} that its kernel and cokernel must be countable. While $\D(\R^4)$ is trivial, $\D^\infty(\R^4)$ is unknown. (It is the group of invertible elements in the monoid of homotopy 4-spheres under connected sum \cite{groupR4}, so countable and abelian, but its triviality is equivalent to the 4-dimensional smooth Schoenflies Conjecture.) Thus, we let $\D^\infty(F)$ denote the kernel of the homomorphism $\D^\infty(\R^4,F)\to\D^\infty(\R^4)$. We obtain a homomorphism $\D(F)\to\D^\infty(F)$.

\begin{thm}\label{group} (a) There is an exotic plane $P_\infty$ on which the uncountable group $\Q^\omega$ acts, as well as all countable subgroups of $\R$ and $S^1$, and the free group $G_\infty$ on countably infinitely many generators (where we use the discrete topology on each of these groups). These actions can be chosen to inject into $\D(P_\infty)$ and $\D^\infty(P_\infty)$.

\item{(b)} There is an exotic plane $P'_\infty$ such that $G_\infty$ and all countable subgroups of $\R$ and $S^1$ each act at infinity, injecting into the cokernel $\D^\infty(P'_\infty)/\im\D(P'_\infty)$.

Both $P_\infty$ and $P'_\infty$ can be chosen to have $g^\infty=0$ or $\infty$ (and be simple in the former case) and can be chosen from among uncountably many isotopy classes in each case.
\end{thm}

\noindent Note that $\R$ and $S^1$ have many countable subgroups. For example, $(\Q/\Z)\oplus_\infty\Q$ embeds in $S^1=\R/\Z$, rationally generated by $\Q/\Z$ and the square roots of all primes. The group $\Q^\omega$ cannot appear in (b) since the given cokernel is countable. In the case $g^\infty=0$, the actions in the theorem are described explicitly by level diagrams in $\R^4$ in Section~\ref{drawGroup}.

\begin{proof}
We first construct the exotic plane $P_\infty$ and its actions. Let $P$ be any of the exotic planes arising in the proof of Theorem~\ref{order}(a) or (c) in Section~\ref{planeFamilies}. By construction, $P=P_-\natural P_+$, where $P_-$ is double branch covered by a small exotic $(R,r_y)$ from Section~\ref{Branch1}, and $P_+$ is the standard plane in (a) of that theorem, but double covered in (c) by a large $R_x^*$, which we assume is constructed as for Theorem~\ref{sum}. Let $P_\infty=\natural_\infty P$. Then $g^\infty(P_\infty)$ is 0 or $\infty$ in the two respective cases, by subadditivity of $g^\infty$ and Proposition~\ref{gTaylor}, respectively. (In the former case, the infinite end sum should still be considered simple; see Section~\ref{Questions}.) In either case, $P_\infty$ can be chosen from uncountably many isotopy classes obtained by varying $R$ and applying \cite[Lemma~7.3]{MinGen} as in the proof of Theorem~\ref{2knot}(b) (where $\tilde{X}$ consists of all summands but one copy of $R$). For fixed $P$, Proposition~\ref{infSum} allows  us to construct this infinite end sum using any choice of countably infinitely many disjoint rays in $\R^2\subset\R^4$, and the resulting plane is independent of the choice. To realize an action by a countable subgroup $G$ of $\R$, take the rays to be $[0,\infty)\times G\subset\R\times\R=\R^2$. Then translation by any element of $G$ determines a self-diffeomorphism of $(\R^4,P_\infty)$.  (Truncating the rays as in the proof of Proposition~\ref{infSum} does not change the resulting exotic plane. For a careful check that the action extends, see the proof of \cite[Theorem~4.4]{groupR4}.) For subgroups of $S^1$, do the same construction in polar coordinates. (Finite subgroups can be embedded in infinite subgroups, or used for actions on finite end sums.) For $\Q^\omega$ (or any countable direct product of countable subgroups of $\R$), realize each factor using a separate copy of $(\R^4,\R^2)$ so that $(-\infty,-1]\times\R^3$ is held fixed, then end sum these together using another copy of $(\R^4,\R^2)$ (fixed by the action). For $G_\infty$, let $F$ be a plane minus an infinite discrete set, consider a single ray in $F\times\{0\}\subset F\times\R^2$, and lift to the universal cover $\R^2\subset \R^4$.

The proof of (a) is completed by comparing with the proof of \cite[Theorem~4.4]{groupR4}, the corresponding theorem for $\R^4$-homeomorphs. By construction, the double branched cover of $P$ can be identified with the corresponding exotic $\R^4$ in that proof. (The latter is denoted $R_S$ or $R_S\natural R_L$ therein, depending on whether we are in the case $g^\infty=0$ or $\infty$. We should use $R$ from Section~\ref{Branch1} in place of $R_S$ that appears in that proof and Remark \ref{reflect}(d) below, which causes no difficulties since we presently have no need of Stein structures.) The double branched cover $R_\infty$ of $P_\infty$ is then an infinite end sum of these, and by construction, the actions of the previous paragraph lift to actions on $R_\infty$ that were shown in that proof to inject into $\D(R_\infty)$ and $\D^\infty(R_\infty)$. If any group element $\gamma$ is isotopic to the identity in $(\R^4,P_\infty)$ then its lift $\tilde\gamma$ to $R_\infty$ is isotopic to either the identity or the covering involution $r$. In the first case, $\gamma$ must be the identity by \cite{groupR4}. Otherwise, the same proof applies to $\tilde\gamma\circ r$ since $r$ preserves the summands and commutes with the relevant involution $r_x$ at the end of each $R$-summand. More strongly, any $\gamma\ne\id$ cannot even be isotopic to the identity after removing a compact subset, completing (a).

For (b), recall that the summand $P_-$ of $P$ is double branch covered by $(R,r_y)$ (Figure~\ref{N}(a)). We have seen that the other involution $r_x$ on the end of $R$ cannot be diffeomorphically extended over all of $R$. Thus, it descends to an involution of the end of $P_-$ that cannot extend to a self-diffeomorphism of $(\R^4,P_-)$. (For an explicit description of this in $\R^4$, see Scholium~\ref{endSym}.) The involution is standard on the end of $\R^4$ (ignoring $P_-$) since it extends standardly over $N_y$. It preserves the orientation of $\R^4$ but reverses it on $P_-$. Construct $P'_\infty$ as in the previous paragraphs, except with the orientation on one copy of $P_-$ reversed. Then $P'_\infty$ is the same as $P_\infty$ outside a compact set. In particular, the previous group actions still inject into $\D^\infty(P'_\infty)=\D^\infty(P_\infty)$. The double branched covers of $P'_\infty$ and $P_\infty$ can be diffeomorphically identified, but their group actions at infinity are then conjugate by an $r_x$-twist on one summand. According to \cite[Theorem~4.6]{groupR4}, nontrivial elements of the new group actions can no longer extend over $R_\infty$. The same then applies to $P'_\infty$. The last sentence of the theorem follows for $P'_\infty$ just as for $P_\infty$ since their double branched covers agree.
\end{proof}


\section{Level diagrams}\label{Diagrams} 

We now explicitly describe some exotically embedded surfaces by drawing their successive levels with respect to the radius function on $\R^4$, which we can take to be a Morse function on the surfaces. The key step will be to understand the first stage core $c_1$ of a Casson handle, whose interior we have already seen to be diffeomorphic to $\R^4$ (Proposition~\ref{engulf}). Once we can draw such immersed planes $\inter c_1$ in $\R^4$ explicitly, their ends will typically be explicit exotic annuli as in Theorem~\ref{0g}(b). We can then use the satellite construction to obtain explicit diagrams of the simple exotic planes from Section~\ref{Branch} that prove Theorem~\ref{0g}(a), the simplest being Figure~\ref{plane} in the introduction. We focus on using the simplest Casson handle $CH_+$, whose kinky handles each have a single positive double point, then indicate how to generate uncountably many isotopy classes by using more general Casson handles. We exhibit the inextendible involutions of the ends of exotic planes used in the proof of Theorem~\ref{group}(b), and indicate how to draw group actions as in that theorem.

\subsection{Annuli from Casson handle cores}\label{DrawCH}
First recall that the proof of Proposition~\ref{engulf} provides inclusions $T_{n-1}\subset B_n\subset\inter T_n$, for each $n\ge2$, where $B_n$ is the 4-ball mapped onto the ball of radius $n$ by the diffeomorphism $\inter CH_+\approx\R^4$, and $T_n\subset\inter CH_+$ is a suitably thinned version of the initial $n$-stage subtower of $CH_+$. The same proof shows that each $T_n$ in $CH_+$ is diffeomorphic to its top stage kinky handle, which can be identified with $S^1\times D^3$. This is pictured (as in Figure~\ref{kink}) by the large solid torus in Figure~\ref{Tn}(a), extended by product with the interval $I=[0,1]$. Before thinning, $T_{n-1}$ can then be identified with a neighborhood of the attaching circle of the kinky handle, so it is represented by the thick Whitehead curve extended over the subinterval $[\frac12,1]\subset I$. It follows by induction that the attaching circle $\partial c_1$ of $T_n$ appears at $t=1$ as the $n$-fold Whitehead double of the core circle of the solid torus. (By symmetry of the Whitehead link, this matches the description in Section~\ref{Casson}.) After  the thinning operation, $T_{n-1}$ appears as the same solid torus given by the thick curve, but only extended over $[\frac35,\frac45]$. The core of this solid torus extends over $[\frac45,1]$ as an annulus $A_n$. We see inductively that $\inter c_1\subset\bigcup T_n=\inter CH_+$ is obtained from the core immersed disk of the thinned $T_1$ by its union with each of the annuli $D^{n-1}A_n\subset T_n-\inter T_{n-1}$, $n\ge2$. (Thus, it intersects the boundary of each of the thinned copies of $T_n$ in its attaching circle, as required.)

\begin{figure}
\labellist
\small\hair 2pt
\pinlabel {(a)} at 0 5
\pinlabel {(b)} at 150 5
\pinlabel {$-2$} at 197 10
\endlabellist
\centering
\includegraphics{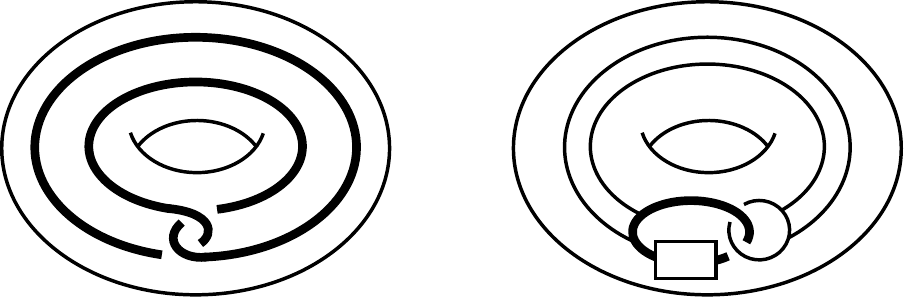}
\caption{Exhibiting $T_{n-1}\subset B_n\subset\inter T_n$ for $n\ge2$. The tower $T_{n-1}$ is given by the thick curve, whose $(n-1)^{st}$ double is the attaching circle of $T_{n-1}$ and (in (a) before thinning) also of the large solid torus $T_n$ and $CH_+$. The 4-ball $B_n$ surrounds $T_{n-1}$ and the 0-handle in (b).}
\label{Tn}
\end{figure}

To see $\inter c_1$ as a level diagram in $\R^4$, we must intersperse this description with the balls $B_n$ and interpret these as round balls in $\R^4$. First we isotope $T_{n-1}$ vertically in $T_n$, fixing the concave-left strand of the clasp but moving the concave-right strand down to the interval $[\frac15,\frac25]$. This vertically stretches the annulus $A_n$ connecting $\partial T_{n-1}$ to $\partial T_n$, but the projection to Figure~\ref{Tn}(a) is unchanged. Next we move $T_{n-1}$ along $T_n$ preserving the coordinate $t$ in $I$, again fixing the concave-left strand, so that $T_{n-1}$ projects into a small ball as in (b). (This is an isotopy since the two strands of the clasp lie in disjoint intervals of $I$.) This move preserves the blackboard framing, but lowers the writhe by 2, so we must put two left twists into the embedding of $T_{n-1}$ in (b) to recover the 0-framing in (a).  The isotopy drags along the part of $A_n$ with $t<\frac12$, but leaves the part with $t>\frac12$ fixed, so that $A_n$ is stretched across a horizontal disk at $t=\frac12$. This is shown in (b), with the disk interpreted as a canceling 0-1 handle pair. Finally, we raise all of $T_{n-1}$ back to $[\frac35,\frac45]$. This pushes the 1-handle of $A_n$ to some level above $\frac45$, but the 0-handle stays at $t=\frac12$ since it is blocked above by $T_{n-1}$. In this final configuration, viewed as levels with increasing $t$, we first see a 0-handle of $A_n$ appear at $t=\frac12$. Its boundary persists until $T_{n-1}$ appears as a meridian solid torus for $\frac35\le t\le\frac45$. At $t=\frac45$, that meridian becomes one boundary component of $A_n$, and then persists until the 1-handle connects it to the boundary of the 0-handle. After this, the remaining boundary is a Whitehead curve that persists until at $t=1$ it becomes the other boundary component of $A_n$, the attaching circle of the top kinky handle of $T_n$. We take $B_n$ to be a small 4-ball in $T_n$ containing $T_{n-1}$ and the 0-handle of $A_n$, but not the 1-handle. 

Combining the descriptions of $\inter c_1$ and $A_n$ from the last two paragraphs, we can exhibit $\inter c_1\subset\R^4$ recursively as Figure~\ref{c1}. At $r=1$ we see a Hopf link, which we interpret as the boundary of a pair of intersecting 0-handles in the unit 4-ball, realizing the double point of $c_1$. (The twist box on the single strand of $\alpha_0$ is a reminder that the $-2$-framing on this component becomes the canonical 0-framing of Definition~\ref{0framing} at subsequent stages.) As the radius function on $\R^4$ increases to 2, a 1-handle connects these disks while threading through another pair of 0-handle boundaries. The resulting twisted Whitehead curve has now been exhibited bounding a disk with one double point, the core of $T_1$, where $T_1$ is embedded with two left twists, along with a pair of 0-handles for $DA_2$ as in Figure~\ref{Tn}(b) ($n=2$). The next iteration produces the 1-handles completing $DA_2$ as in that figure, along with the four 0-handles of $D^2A_3$ in $T_3$. Iterating Figure~\ref{c1} for all $n\in\Z^{\ge0}$ exhibits $\inter c_1$ in $\R^4$.

\begin{figure}
\labellist
\small\hair 2pt
\pinlabel {$\alpha_{n+1}$} at 255 49
\pinlabel {$\alpha_n$} at 105 49
\pinlabel {$\alpha_n$} at 10 49
\pinlabel {$\alpha_0=$} at 37 117
\pinlabel {$\alpha_n=$} at 165 117
\pinlabel {$\alpha_{n-1}$} at 208 132
\pinlabel {$-2$} at 84 120
\pinlabel {$-2$} at 218 103
\pinlabel {$-2$} at 152 8
\pinlabel {$2^{n}$ strands} at 65 22
\pinlabel {$2^{n-1}$ strands} at 280 117
\pinlabel {$2^{n}$ strands} at 215 22
\pinlabel {$2^{n+1}$ strands} at 312 22
\pinlabel {$\approx$} at 230 40
\pinlabel {$r=n+1$} at 20 -5
\pinlabel {$r=n+2$} at 230 -5
\endlabellist
\centering
\includegraphics{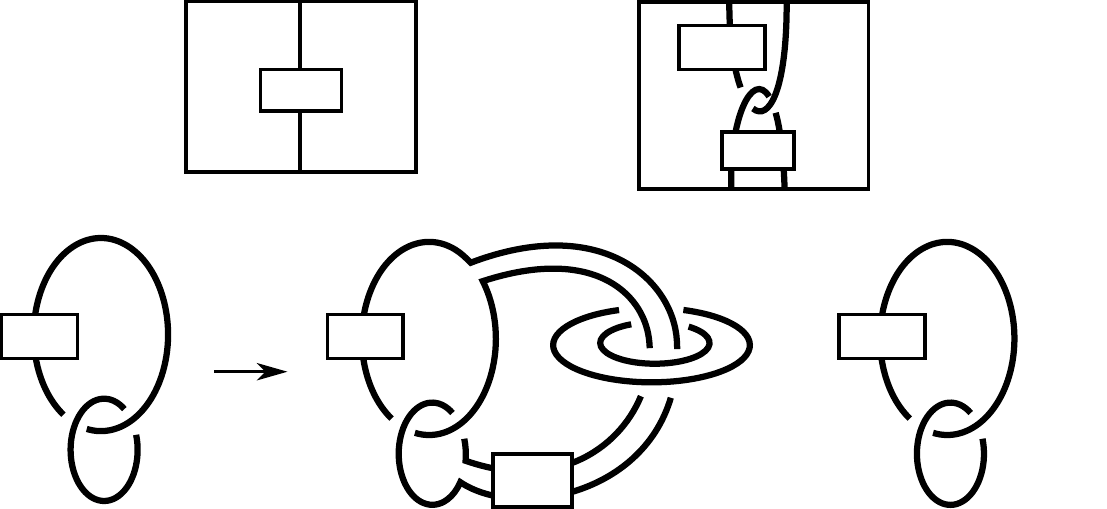}
\caption{The first stage core $\inter c_1$ of $\inter CH_+\approx\R^4$ as an infinite level diagram. For $n\ge1$, the tower $T_n$ is embedded (with $-2$ twists) in $B_{n+1}$ as the ball of radius $n+\frac12$ with a 1-handle attached, extended to level $r=n+1$ as the solid torus in the lower left diagram containing the upper circle and $\alpha_n$ box but disjoint from the lower circle.}
\label{c1}
\end{figure}
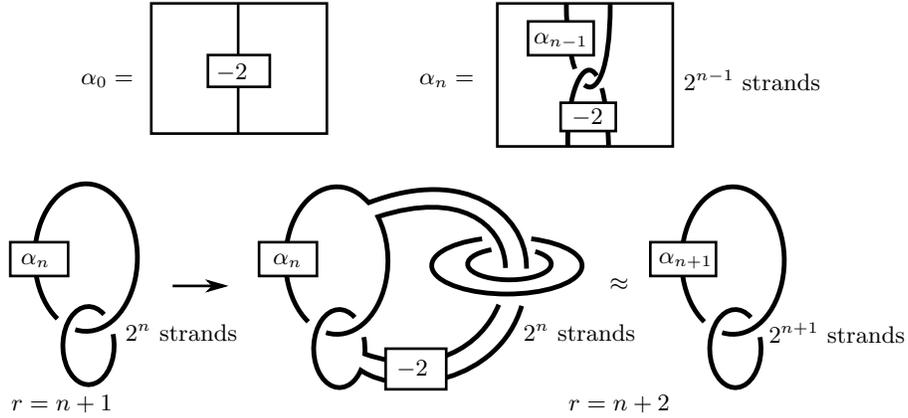
For an arbitrary Casson handle, the first stage core $\inter c_1$ can be drawn similarly. The main complication is that we must use boundary sums of solid tori at each stage as in Figure~\ref{kink}. For example, the corresponding diagram for $CH_{m,n}$ from Section~\ref{CH} is made by drawing $m+n$ copies of Figure~\ref{c1} in separate regions, $n$ of these with reversed ambient orientation, and connected summing them at each copy of $\alpha_0$. (Note that $\alpha_0$ lies in a left half-space whose intersection with the diagram is independent of $r\ge1$.) The beginning of a more typical $\inter c_1$ is shown in Figure~\ref{ram}. The core of the thinned first stage $T_1$ is shown by the top three disks connected with two bands. (This exhibits a disk with two positive double points whose regular neighborhood is $T_1$.) Each dashed arc represents a pair of parallel ribbons connecting the boundary of a band to a pair of 0-handle boundaries. Then $T_2$ is a ball containing these 0-handles and $T_1$, with three 1-handles added along the dashed arcs. Each 1-handle of the surface has $-2\sigma$ twists, where $\sigma$ is the sum of the signs of the double points at the next stage. For example, the upper left band is untwisted since the set of disks it threads through corresponds to a pair of double points of opposite sign (as seen in the second level of edges of the corresponding graph).

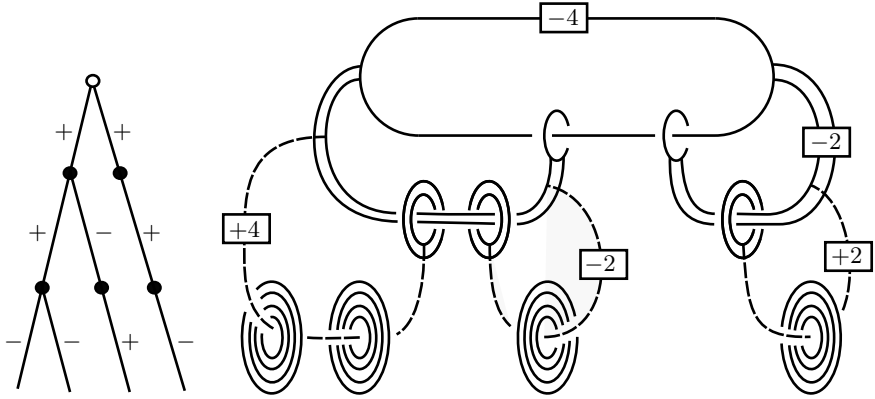
\begin{figure}
\labellist
\small\hair 2pt
\pinlabel {$-4$} at 206 142
\pinlabel {$-2$} at 305 95
\pinlabel {$+4$} at 86 62
\pinlabel {$-2$} at 220 49
\pinlabel {$+2$} at 313 52
\pinlabel {$+$} at 18 99
\pinlabel {$+$} at 40 99
\pinlabel {$+$} at 8 61
\pinlabel {$-$} at 33 61
\pinlabel {$+$} at 51 61
\pinlabel {$-$} at -1 20
\pinlabel {$-$} at 21 20
\pinlabel {$+$} at 43 20
\pinlabel {$-$} at 64 20
\endlabellist
\centering
\includegraphics{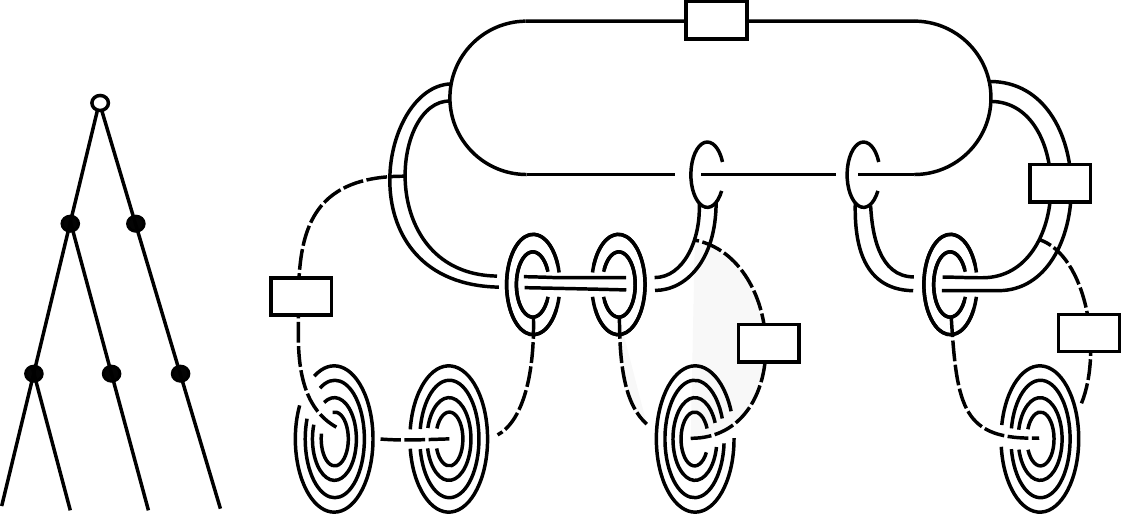}
\caption{Drawing the first stage core $\inter c_1$ of the Casson handle with the given signed graph.}
\label{ram}
\end{figure}

To draw an exotic annulus as in Corollary~\ref{annuli}, it now suffices to delete an open disk from $\inter c_1$ containing at least one sheet of each double point. (Note that $\partial c_1\subset\partial CH$ also bounds any almost-smooth topological core, so the corresponding annuli are smoothly isotopic.) We could delete an arbitrarily large disk, so that the diagram starts at some large radius $r$. However, it is easiest to see the pattern if we delete a small disk, namely the disk with boundary at $r=1$ made from $\alpha_0$ in Figure~\ref{c1} or the disk bounded by the top circle in Figure~\ref{ram}. Corollary~\ref{annuli} and its proof now imply:

\begin{prop}\label{drawAnn}
Figure~\ref{c1} represents an exotic annulus, with boundary given by the circle containing $\alpha_0$ at $r=1$, and with $g=\kappa_+=1$ and $\kappa_-=0$. We realize all nonzero values of $\kappa$, with $g=\max\{\kappa_+,\kappa_-\}$, by connected sums of copies of this diagram (suitably oriented), and uncountably many realizing each value as in Figure~\ref{ram}. \qed
\end{prop}

While we have explicitly described annuli realizing all  nonzero $g$ and $\kappa$ as sums of $g$ or $\kappa_+ +\kappa_-$ suitably oriented copies of Figure~\ref{c1}, the ramification required for uncountable families is harder to describe. In principle, it can be described explicitly by applying Bi\v zaca's algorithm \cite{B} to Freedman's uncountable nesting of Casson handles. However, for the first Casson handle distinguished from a given $CH_{m,n}$ by this method, the number of disks at the $k^{th}$ stage increases highly superexponentially with $k$, and subsequent Casson handles grow successively faster. On the other hand, it seems likely that Casson handles are classified up to diffeomorphism by their signed trees, in which case the corresponding annuli are also.

\subsection{Exotic planes}\label{DrawPlanes}
We now wish to draw a simple exotic plane $P$ with double branched cover $R$ a small exotic $\R^4$ as in Section~\ref{Branch1}. Recall that $R$ is obtained from Figure~\ref{N}(a) of $N$, the exotic punctured $S^2\times S^2$, by a surgery that changes one large curve from 0-framed to dotted. (We remove boundary as needed to get open manifolds.) The involution $r_y$ is rotation about a horizontal line in the paper. Since we do not presently need the other involutions of $N$, we can replace $2CH$ by any sufficiently complicated Casson handle $CH'$. In fact, $R$ is exotic if we use $CH_+$ or any refinement of it \cite{BG}. These refinements range $R$ over uncountably many diffeomorphism types (Section~4.1), realizing uncountably many ends (since only countably many manifolds can have a given end), yielding uncountably many exotic planes $P$ and annuli.

The branch locus in $N$ is the fixed set of $r_y$, given in Figure~\ref{N}(a) as an unknotted disk in the 0-handle, bounded by the $y$-axis, together with a 2-dimensional 1-handle $D^1\times D^1$ inside each 2-handle $D^2\times D^2$. In the quotient, these 2-handles become 4-balls attached to the 0-handle along 3-balls, so they do not contribute to the topology. However, the 1-handles inside them appear as in Figure~\ref{Ny}(a) (which is essentially \cite[Figure~8]{menag}, reflected as discussed at the beginning of Section~\ref{Branch} above). The dashed arcs in the figure are ribbon moves exhibiting a level diagram of the obvious punctured-torus Seifert surface pushed into the interior of the 4-manifold. (We have essentially inverted the Akbulut--Kirby algorithm \cite{AK} for drawing branched covers of pushed-in Seifert surfaces.) Surgering $N$ to $R$ replaces an $S^2\times D^2$ by $D^3\times S^1$, where $r_y$ reflects both factors of each. This does not change the quotient manifold, but surgers the branch locus along an unknotted disk. The resulting exotic plane $P$ is obtained in the figure by deleting one of the dashed arcs, breaking the previously visible $r_z$-symmetry when $CH'=2CH$. Removing the Casson handle exhibits $P$ as the interior of a ribbon disk $D\subset I\times S^1\times D^2$. This is also shown after isotopy in (b) and (c). (The isotopy to (b) is $r_z$-equivariant if the second ribbon is shown symmetrically in (b), with the Casson handle $2CH$ attaching to a meridian at the upper fixed point of the dotted circle.) Replacing the Casson handle by a 2-handle in any of these diagrams cancels the 1-handle, exhibiting $D$ as an unknotted disk in the 4-ball. Thus, in the topological category, $P$ is an unknotted plane in $\R^4$ as expected (and $N$ is the branched cover of a topologically standard punctured torus). In the smooth category, the pictured $I\times S^1\times D^2$ is essentially a boundary collar of $CH'$, so we get diffeomorphisms $R_y\approx\inter CH'\approx\R^4$, showing directly that $P$ is smoothly a plane in $\R^4$ (necessarily exotic since it is branch covered by $R$). We can identify $[0,1)\times S^1\times\{0\}\subset I\times S^1\times D^2$ as the annulus $A=[0,\infty)\times S^1\subset \inter c_1$ of Proposition~\ref{drawAnn}. (Think of $CH'$ as attached at level $0\in I$.) Thus, in $\inter CH'\approx\R^4$, $P$ is made from $A$ by a satellite construction whose pattern begins with $D$ near $\partial A$ and  extends as $[0,\infty)\times\partial D$ along the rest of $A$. This can be drawn by quadrupling all strands in the diagram of $A$ and inserting the generalized clasp of $D$ shown in (c), using the canonical 0-framing of $A$. We conclude:

\begin{figure}
\labellist
\small\hair 2pt
\pinlabel {Branch locus} at 69 73
\pinlabel {0} at 48 199
\pinlabel {$CH'$} at 91 199
\pinlabel {$D$} at 255 17
\pinlabel {$D$} at 285 100
\pinlabel {(a)} at 5 200
\pinlabel {(b)} at 208 200
\pinlabel {(c)} at 150 5
\endlabellist
\centering
\includegraphics{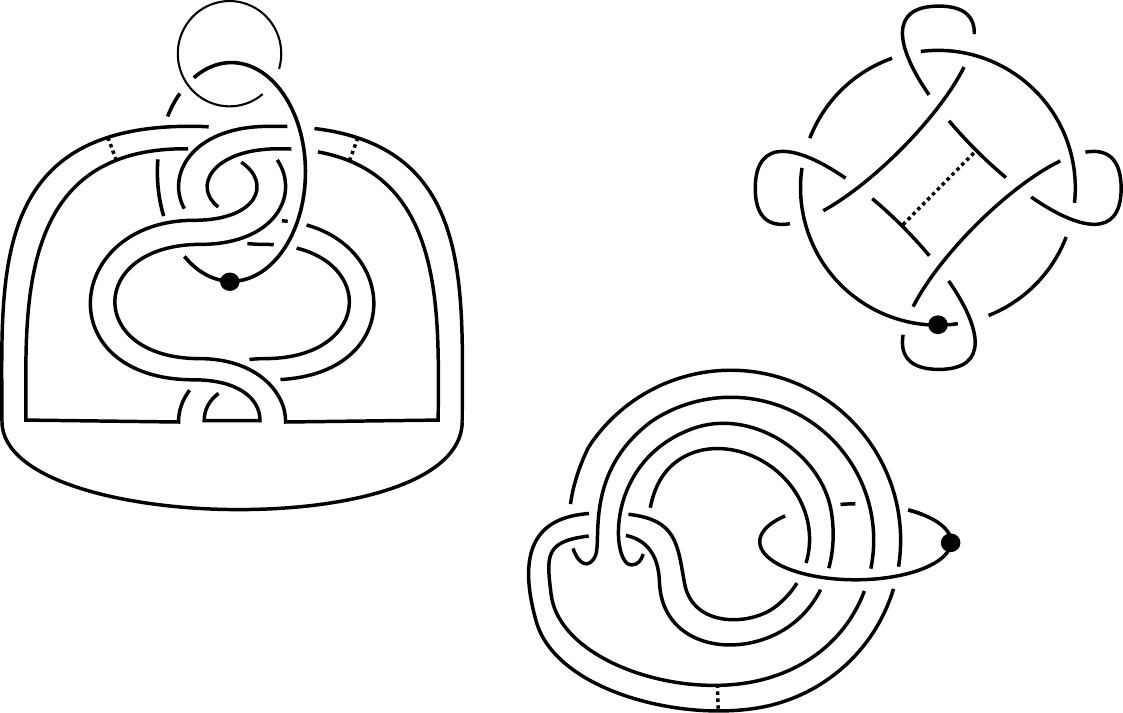}
\caption{(a) Exotic punctured torus (with two ribbon moves) and plane $P$ (with one ribbon move) respectively branch covered by $N$ and $R$, and (b,c) the corresponding pattern $P=\inter D$.}
\label{Ny}
\end{figure}
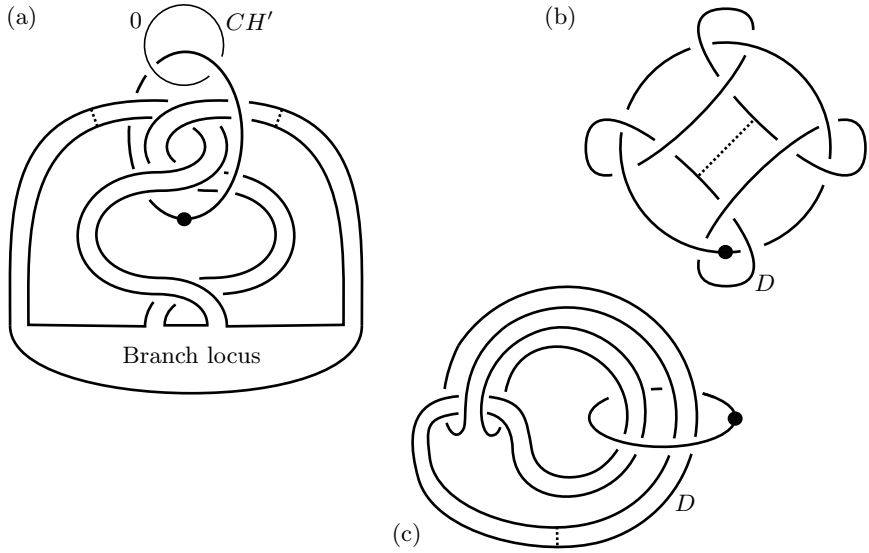

\begin{thm}\label{drawPlanes}
Figure~\ref{plane} shows a simple exotic plane double covered by an exotic $R$ made as in Section~\ref{Branch1} with $CH_+$ replacing each $2CH$. The corresponding diagram with an arbitrary Casson handle $CH'$ in place of $CH_+$ is made from the diagram for $CH'$ as in Figure~\ref{ram} by inserting $\alpha_0$ from Figure~\ref{plane} (with the number of twists in its box chosen suitably) and quadrupling the other strands. In particular, this includes an uncountable family of distinct simple exotic planes whose ends determine distinct annuli.  \qed
\end{thm}

\noindent To use Figure~\ref{ram} as drawn, for example, we would change the twist box in $\alpha_0$ from $-2$ to $-4$. Unlike for annuli, there is no known way to extract an explicit pairwise nonisotopic family of these planes, since we only obtain an uncountable family in which each isotopy class occurs at most countably often. (See Section~\ref{R4}.) However, we do know there is an uncountable distinct family made from Casson handles with only positive double points, cf.\ \cite[Proof of Lemma~3.2(b), end of Section~6]{groupR4}. As with annuli, we are free to conjecture that distinct signed trees determine different exotic planes.

\begin{proof}[Proof of Theorem~\ref{0g}]
The main remaining step is to understand the behavior of the ends of our surfaces. We first show that every Casson handle interior has $g^\infty(\inter c_1)=0$. For this, we one-point compactify $\R^4$ to $S^4$ and look at an arbitrarily small neighborhood of $\infty$. After further reducing this neighborhood, we may assume its closure $Z$ is the complement of some (thinned) $\inter T_n$ as in Section~\ref{DrawCH}. Since the latter is a neighborhood of a wedge of circles in $S^4$, $Z$ is a boundary sum of copies of $S^2\times D^2$, made from a dotted unlink diagram of $T_n$ by changing dots to zeroes. In the case of $CH_+$, $\inter c_1$ intersects $\partial Z=\partial T_n\approx S^2\times S^1$ in $D^{n-1}D^*\mu$, where $\mu$ is a meridian of the 0-framed unknot describing $Z$, and the first double $D^*$ is taken using the $-2$-framing of $\mu$ (Figure~\ref{Tn}(b) with $n$ replaced by $n+1$). Since the framing of $\mu$ can be changed by any even number by sliding over the 2-handle of $Z$, we may instead describe $\inter c_1\cap\partial Z$ as $D^n\mu$. After we reduce $Z$ further by removing the 2-handle, this curve is exhibited as an unknot in $S^3$, so it obviously bounds an embedded disk in $Z$. This proves $g^\infty(\inter c_1)=0$ for $CH_+$. The proof for an arbitrary Casson handle is similar, except that $Z$ will be given by a 0-framed unlink and other even framings of the meridians may arise (cf.\ Figure~\ref{ram}).

Theorem~\ref{0g} now follows immediately. Corollary~\ref{annuli} already exibited uncountably many annuli realizing each nonzero value of $\kappa$ or $g$. These were the ends of surfaces $\inter c_1$ (for refinements of the Casson handles $CH_{m,n}$) so had $g^\infty=0$ as required for (b) of the theorem. Each exotic plane $P$ constructed for Theorem~\ref{drawPlanes} is a satellite on such an annulus $A$ with pattern $D\cup ([0,\infty)\times \partial D)$. After we cap $A$ by a disk in $Z$ as in the previous paragraph, we can also cap $P$ since $\partial D$ is unknotted in $S^3$. The resulting 2-sphere can be pulled entirely into $Z$ along $A$, where it is seen to be unknotted since $D$ is an unknotted disk in the 4-ball. Thus, $P$ is generated by unknots, so $g^\infty(P)=0$. This proves (a), and the annuli determined by these planes give the missing case of (b) with $\kappa_+=\kappa_-=g=0$.
\end{proof}

\begin{Remarks}\label{pi1}
(a) Similar reasoning shows that a surface $F$ in $\R^4$ is generated by 2-knots whenever it is a satellite on an annulus with pattern given by a slice disk $D\subset I\times S^1\times D^2$. If $D$ is also unknotted in $B^4$ then $F$ is generated by unknots (since the resulting sphere can be pulled into a neighborhood of a circle near infinity, and this circle is necessarily unknotted, with the correct framing in $\Z_2$). It seems harder in general to recognize when $F$ is an exotic plane.

(b) Arguably, the most surprising point of this section is that the planes and annuli we have drawn with level diagrams are topologically standard, so we check this directly in Figures~\ref{plane} and~\ref{c1}. (The other cases are similar.) Each pictured surface $F$ (a plane or $\inter c_1$) is exhibited without local maxima, so any nullhomologous loop in its complement can be retracted to the spine of a nullhomotopy and then pushed outward to a product of commutators in some $\partial T_n-F$. It now suffices to show that the image of $\pi_1(\partial T_n-F)\to\pi_1((\R^4-\inter T_n)-F)$ is abelian, for then $\pi_1(\R^4-F)\cong\Z$, and the complement of an open tubular neighborhood of $F$ has a system of neighborhoods of infinity (complementary to the subsets $T_n$) with fundamental group $\Z$. The neighborhood system guarantees that the pictured annulus is topologically standard (equivalently, its compactification is locally flat) by Venema \cite{V}, and the knot group $\Z$ then implies that the plane in Figure~\ref{plane} is standard (since its compactification is a topologically unknotted sphere by Freedman, e.g., \cite[11.7A]{FQ}). 

To prove the $\pi_1$-condition, recall that either diagram shows $T_n$ in the lower left picture as a solid torus $T$ in the boundary of $B_{n+1}$, extended by a collar into $\inter B_{n+1}$, where $T$ contains the upper component of the thick Hopf link and avoids the lower one. Then $\partial T_n\cap F$ is the curve in $T$ generated by the recursion. The meridian $\mu$ of $T$ bounds a disk along $\partial T_n-F$, parallel to the 0-handles whose boundaries are also meridians. The longitude $\lambda$ of $T$ is isotopic in $\R^4-\inter T_n-F$ to a meridian of the solid torus containing these 0-handle boundaries. If we push $\mu$ and $\lambda$ up to level $n+2$, they become isotopic disjointly from $F$. Thus, both are nullhomotopic in $(\R^4-\inter T_n)-F$. (The nullhomotopy of $\lambda$,  when taken to lie in $\R^4-T_n$ except for $\lambda\subset \partial T_n$, is a disk immersed with one double point; this is the core $c_{n+1}$ of the top stage of $T_{n+1}$.) It follows that the desired image of $\pi_1(\partial T_n-F)$ is also that of $\pi_1(T-F)/\langle\pi_1(\partial T)\rangle$. To compute this group, we can remove the $-2$-twist box from $\alpha_n$, erase the thick lower curve from the left picture and work in $S^3$. But we can recheck by induction that the resulting circle is unknotted. (The twist box in each remaining $\alpha_k$ disappears when we unwrap the Whitehead curve determined by $\alpha_{k+1}$.) Thus, this latter group and its image are abelian (in fact, $\Z$) as required. A similar discussion applies to surfaces generated by more general Casson handles as in Figure~\ref{ram}. The main difference is that each $\lambda$ bounds an immersed disk with multiple double points (the core of the attached kinky handle). 
\end{Remarks}

\subsection{The inextendible involution of the end}
Recall that the inextendible actions of Theorem~\ref{group}(b) were constructed using an involution of the end of an exotic plane $(\R^4,P)$ that could not be diffeomorphically extended over the pair.  This involution descends from either of the involutions $r_x$ or $r_z$ on $N$ (Figure~\ref{N}(a)), so is pictured as $r_z$ in Figure~\ref{Ny}(a) of $N_y\approx\R^4$. As we have seen, the pictured invariant punctured torus is the branch locus of the branched covering $N\to N_y$, and deleting one ribbon move yields $P$ with the inextendible involution of its end. To draw a level diagram of this, we must perform the satellite operation of Theorem~\ref{drawPlanes} $r_z$-equivariantly. The pattern is seen in Figure~\ref{Ny}(b), where we attach $2CH$ to a meridian of the upper fixed point of the dotted circle, with $CH$ any refinement of $CH_+$ (to guarantee that $R$ is exotic). The level diagram of $2CH$ can be drawn by connected summing two diagrams for $CH$ so that they are interchanged by a $\pi$-rotation that reverses the string orientation on the sum. The equivariant satellite operation then quadruples all strands of $2CH$ and at one fixed point inserts the generalized clasp exhibited at the center of Figure~\ref{Ny}(b). (This is awkward to draw in 2 dimensions, but can be visualized 3-dimensionally.) We conclude:

\begin{sch}\label{endSym}
For any refinement $CH$ of $CH_+$, Figure~\ref{Ny}(b) is the pattern for an $r_z$-equivariant satellite operation on the annulus obtained from $2CH$, exhibiting a simple exotic plane $P$ whose end is symmetric under a $\pi$-rotation of $\R^4$, but for which the involution cannot extend to any self-diffeomorphism of $(\R^4,P)$. Symmetrically adding a second ribbon move to the figure gives an invariant, topologically trivial punctured torus with the same exotic end as $P$. This is the branch locus of the branched covering $N\to N_y\approx\R^4$. \qed
\end{sch}

\begin{Remarks}\label{reflect}
(a) In the literature \cite{DF}, \cite{menag}, \cite{BG}, $R$ is often described by deleting a pair of ribbon disks of the $(-3,3,-3,3)$-pretzel link from $B^4$ and attaching a Casson handle to a meridian of each component. These disks are shown in Figure~\ref{pretzel} (ignoring the widely dashed circle). The two finely dashed arcs, one of which contains the point at infinity, are the required ribbon moves. (One of these arcs immediately cancels, but we retain both to display a symmetry.) Our current methods easily recover this pretzel link description of $R$, while exhibiting the action by $G=\Z_2\oplus\Z_2$ on its end displayed in Figure~\ref{N}(a): There is an isotopy of Figure~\ref{Ny}(b) interchanging the ribbon disks, after we isotope the disk bounded by the dotted circle to add a cancelable ribbon move. This isotopy restricts equivariantly (under rotation about the $z$-axis) to the link (ignoring the ribbon moves), so it preserves the inextendible involution of the end of $P$. We can now interchange the roles of the two circles, so that $P$ is given by the obvious disk bounded by the dotted circle and the ambient $\R^4$ is the complement of the other disk with a Casson handle added to its meridian. The double branched cover $R$ of $P$ is now easily seen to be Figure~\ref{pretzel}, with branch locus given by the obvious disk in $B^4$ bounded by the widely dashed circle. The involution $r_y$ is the pictured covering involution. If the Casson handles $2CH$ are attached to meridians at points projecting to the top of the dashed circle in Figure~\ref{pretzel}, the full $G$-action of the end is generated by $r_y$ and rotation about the vertical axis, which turns out to be $r_x$ in Figure~\ref{N}(a). (This approach does not distinguish $r_x$ from $r_z$. However, we can reach this same conclusion directly from Figure~\ref{N}(a) by considering the small dotted circles to represent disks in $B^4$ and equivariantly canceling the large 1-2 handle pair of $R$. The required equivariant isotopies to reach Figure~\ref{pretzel}, via  \cite[Figure~12]{menag}, are difficult.)

\begin{figure}
\labellist
\small\hair 2pt
\pinlabel  {$r_y$} at 87 23
\endlabellist
\centering
\includegraphics{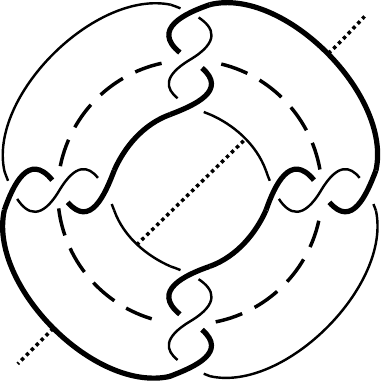}
\caption{Ribbon disks for the $(-3,3,-3,3)$-pretzel link. Adding a Casson handle to each meridian of the complement gives $R$ with its involution $r_y$ and $\Z_2\oplus\Z_2$-action of its end (with $r_x$ given by rotation about the vertical axis).}
\label{pretzel}
\end{figure}

(b) The pretzel link of Figure~\ref{pretzel} actually exhibits a $\Z_2^3$-action extending the $G$-action above. This only seems to contribute one additional $G$-action to our current discussion (see (c) below), but the full $\Z_2^3$-action may be useful for cork theory. This action can naturally be seen in the figure by one-point compactifying $\R^3$ and identifying the page as the equatorial $S^2$ of $S^3\subset\R^4$. Then the three standard coordinate axes and the three coordinate circles containing center points of the twisted bands become the intersections with $S^3$ of the six coordinate 2-planes of $\R^4$, so are natural axes of rotation, generating the orientation-preserving subgroup of the $\Z_2^4$ of coordinate reflections of $\R^4$. The antipodal map can be extended over $R$ but fixes only $0\in\R^4$ with quotient homeomorphic to a cone on $\R P^3$. Rotation in the plane of the paper extends over the ribbon complement but cannot be further extended over $R$ as an involution with branch locus $\R^2$, since it sends each component $K_i$ of $L$ to itself without fixed points. Such a rotation of $K_i$ would correspond to a rotation in $\partial (2CH)$ fixing the attaching circle. The resulting fixed set in $2CH$ would then be a forbidden smooth disk spanning the attaching circle. There are only two subgroups of order greater than 2 avoiding these two elements. However, if we replace the Casson handles by 2-handles with a given framing $n$, the full $\Z_2^3$-action action extends over the resulting 3-manifold, and for $n\le-4$ (and maybe larger) these diffeomorphisms do not all extend over the resulting cork (cf.~\cite[Theorem~6.4(d)]{groupR4}).

(c) Recall from the beginning of Section~\ref{Branch} that our diagrams are oriented as in \cite{BG}, so oppositely to \cite{menag}. We can now see that both conventions produce the same $\R^4$-homeomorphs. This is because Figures~\ref{Ny}(b) and \ref{pretzel} admit reflectional symmetries (diagonally). Thus, there is an orientation-preserving diffeomorphism between Figure~\ref{N}(a) and its mirror image (after surgering the 0-framed curve to a dotted circle, but before attaching the Casson handles). This preserves (up to isotopy) the meridians where the Casson handles attach. Replacing both copies of $2CH$ by $CH_+$ in both diagrams now gives diffeomorphic smoothings $R$ of $\R^4$, so the version from \cite{menag} is also exotic without ramification via \cite{BG}. We similarly obtain a diffeomorphism for any fixed choice of $CH$, and the involutions $r_y$ (but not $r_x$) correspond. The main difference between the two conventions is that the quotients of $r_x$ (and similarly $r_z$) have opposite orientations. Thus, $N_x$ in \cite{menag} is an exotic $\C P^2-\{p\}$ with a positive double point while $N_z$ is standard, and $R_x^*$ has the end of a negative definite, nondiagonalizable 4-manifold but lives in $\C P^2$ and cannot embed in a negative definite, closed 4-manifold. For a suitably ramified $CH$, we obtain $R$ as an exotic $\R^4$ admitting both $G$-actions, sharing the involution $r_y$, and with the corresponding large exotic $\R^4$-homeomorphs $R_x^*$ from the two actions mirror images of each other. The origin of this symmetry is that $S^2\times S^2$ has an orientation-reversing diffeomorphism reflecting (say) the first factor. In holomorphic affine coordinates, this conjugates the group action of Section~\ref{Branch1} by complex conjugation in the first factor, interchanging $r_x$ and $r_z$. In diagrams, this reverses the clasp of the Hopf link and the orientation of one circle (so the linking number is still $+1$). Equivalently, we isotopically flip over one circle fixing the other. The symmetry is broken when we pass to the subset $N$ by adding clasps and Casson handles, resulting in the exotic quotient of $N$ lying in oppositely oriented Hopf bundles.

(d) There is also a simpler exotic $\R^4$ appearing in \cite{BG} (see \cite{GS} for a simpler construction) but it does not appear to admit an involution or generate an exotic plane. (It does admit a Stein structure, while $R$ does not apear to.) This arises since the middle level of the h-cobordism has a 2-handle canceling the sum of the 1-handles in Figure~\ref{N}(a). The resulting diagram has a single dotted circle, surrounding the central clasp, with a single meridian Casson handle that can be any refinement of $CH_+$. The analogue of Figure~\ref{pretzel} is the $(-3,3,-3)$-pretzel knot $K$, also known as $\bar 9_{46}$, obtained from the figure by removing a right-twisted band (e.g. \cite[Lemma~7.1]{groupR4}). This time there is no symmetry cyclically permuting the bands, so the relevant symmetry group replacing $\Z_2^3$ in (b) above is $\Z_2^2$. Now the analogue of $r_y$ sends $K$ to itself without fixed points (rather than interchanging the components of $L$), so the reasoning in (b) shows it cannot extend over the Casson handle. Thus, we do not obtain an involution of the exotic $\R^4$, only involutions $r_x$ and $r_z$ of its ends that no longer commute. (The analogue of Figure~\ref{N}(a) has no common fixed point on the central dotted circle for locating a Casson handle compatible with both.) Either of these can be chosen as Casson's involution that cannot extend over the exotic $\R^4$. The latter is realized as rotation in the plane of the paper if the knot is drawn with parallel bands \cite[Lemma~7.1]{groupR4}. Since the analogue of Figure~\ref{pretzel} has no reflectional symmetry, there is no discussion analogous to (c) above. In particular, it is not known if adding $CH_+$ to the mirror image of the pretzel-knot ribbon-disk complement gives an exotic $\R^4$.
\end{Remarks}

\subsection{Group actions}\label{drawGroup}
To draw the more general group actions of Section~\ref{Group} on simple exotic planes, we only need to understand end sums. The finite (cyclic) case consists of equivariantly connected summing copies of a given diagram to a 0-handle with its obvious action, using a band at each copy of $\alpha_0$. For infinite sums, we also need to add constants to the radial coordinates of the summands so that the intersection with each 4-ball has a finite handle structure. For example, a $\Z$-invariant exotic plane is obtained from a collection $\{P_m|\thinspace m\in\Z\}$ of copies of $P$, where $P_m$ is obtained from $P$ by translating by $m$ units in a fixed direction (and $P$ is drawn in a 3-ball of diameter 1) and adding $|m|$ to its radial coordinate $r$. The generator sends $P_m$ to $P_{m+1}$ by translation and adding $\pm1$ to $r$ (and suitably adjusting $r$ on the bands connecting to the 0-handle). Realizing a $\Q$-action is similar but more subtle, since small elements of $\Q$ require large shifts of $r$. This is allowable since we assume the discrete topology on $\Q$. (Each group element is continuous since the summands lie in disjoint closed regions in $\R^4$ whose union has no extraneous limit points. The algebraic structure fits together by the equivalent description in Section~\ref{Group}.) For $\Q^\omega$, start infinitely many such $\Q$-clusters on separate intervals of the 0-handle, staggered in $r$ to retain local finiteness of the handle structure. For free groups, tessellate the plane with fundamental domains and start a copy of $P$ in each domain as the 0-handle boundary expands. (It may be simplest to consider each free group as a subgroup of the free group on two generators.) For the inextendible actions of the end in Theorem~\ref{group}(b), construct these actions as for (a), but then switch the ribbon move of Figure~\ref{Ny}(b) to the other diagonal in one copy of $P$.


\section{Open questions}\label{Questions}

We have now constructed and studied two main types of exotic planes. The simple examples from Theorem~\ref{0g}(a) are constructed so that their double branched covers are made from a simple 4-manifold by adding Casson handles (Figure~\ref{N}(a)). Thus, we could draw the embeddings explicitly (Theorem~\ref{drawPlanes}). The resulting radial functions have no local maxima and connected superlevel sets. These planes are generated by unknots (Definition~\ref{gen}), so have $g^\infty=\kappa^\infty_\pm=0$. Each is double branch covered by a small exotic $\R^4$ that embeds in $\R^4$, and the planes are all equivalent, in the sense of Theorem~\ref{order}, to the standard plane. These properties are retained by infinite end sums of such examples, which we still consider simple. The other type of exotic plane is more complicated. We constructed such planes in two different ways: by deleting a singular point from an almost-smooth 2-sphere in $S^4$ realized by an infinite intersection of Casson handles (Theorem~\ref{ginfty}), or by exhibiting the double branched cover as the complement of a similarly complicated intersection (Theorems~\ref{order}(b,c) and~\ref{sum}). Either way, the construction seems too complicated to draw explicitly. Each level diagram requires infinitely many local maxima (Scholium~\ref{ginftySch} and Corollary~\ref{inf2h}, respectively), and sometimes the number of components of the superlevel sets must become arbitrarily large (Scholium~\ref{infEnd}). Planes of this second type have nonzero $g^\infty$. In fact, the exotic planes of Theorem~\ref{ginfty} realize all nonzero values of $\kappa^\infty$ and $g^\infty=\max\{\kappa^\infty_\pm\}$, and those of Theorem~\ref{sum} realize each of infinitely many nonzero values of $g^\infty$ by uncountably many exotic planes (Corollary~\ref{ginf}). In both Theorems~\ref{order}(b,c) and~\ref{sum}, each plane is double branch covered by a large exotic $\R^4$ with nonzero Taylor invariant.

\begin{ques}\label{2types}
Are the above properties of simple exotic planes all equivalent?
\end{ques}

One candidate for a counterexample is the branch locus $P^*_y$ of the map $R^*\to R_y^*\approx\R^4$ generated by the involution $r_y$ (Section~\ref{Branch1}). Like our simple examples, this is generated by unknots, so $g^\infty(P^*_y)=0$. (Figure~\ref{Ny}(a) shows the embedding $N_y\subset S^4$, where we thicken the Casson handle to a 2-handle and add a 4-handle, with $R^*_y$ the open complement of a smaller version of $N_y$. The full branch locus is the pictured punctured torus, capped by an unknotted disk in the 4-handle. Surgering this gives a disk in the end of $R^*_y$ capping $P^*_y$ to an unknotted sphere in the 4-handle.)  This contrasts with any plane double covered by $R^*_x$, which has $g^\infty>0$ (and for which the corresponding full branch locus is $\R P^2$). However, $R^*$ is made by removing from $S^2\times S^2$ a subset (with cellular quotient) built with nested intersections of Casson handles, so in that respect it resembles $R_x^*$ more than $R$. Do level diagrams of $P^*_y$ require infinitely many local maxima? This would follow if $R^*$ requires infinitely many 3-handles (proof of Theorem~\ref{infLocMax}). However, $R^*$ has vanishing Taylor invariant, so our argument used  on $R_x^*$ breaks down. The author has not analyzed the analogous plane $P_z^*$ in $R^*_z\approx\R^4\subset\C P^2$ generated by $r_z$. What is its behavior at infinity? The branch locus in $\C P^2$ is homologically essential (a quadric curve) so doesn't immediately show $g^\infty=0$. We can ask more generally,

\begin{quess}\label{gen2knots} Does every exotic plane with small double cover have $g^\infty=0$? Is every exotic plane with $g^\infty=0$ (so generated by 2-knots) generated by unknots? Is there a relation between these conditions and having a diagram with no local maxima? Is there an exotic plane requiring local maxima but only finitely many?
\end{quess}

\noindent Another possible source of counterexamples is the exotic planes of Theorem~\ref{ginfty} with $g^\infty>0$, which have unknown double branched covers.

\begin{quess}\label{CHbranch}
Must these covers be large? Is there an exotic plane whose double cover is the standard $\R^4$? Do distinct exotic planes ever have diffeomorphic double covers? For example, are $P_y^*$ and $P_z^*$ distinct? What about $P_\infty$ and $P'_\infty$ from Theorem~\ref{group}? (More generally, consider end sums with different choices of surface orientation.) What can be said about higher degree branched covers?
\end{quess}

Sections~\ref{Branch} and \ref{Diagrams} raise other questions:

\begin{quess}\label{moreG}
What else can be extracted from the $\Z_2\oplus\Z_2$-action of Section~\ref{Branch1}? What about the simpler exotic $\R^4$ of Remark~\ref{reflect}(d)? Are there other interesting exotic group actions that we can study in this manner?
\end{quess}

\noindent In addition to the above discussion of $P^*_y$ and $P^*_z$, we can consider other quotients from Section~\ref{Branch1}. Since the action is topologically standard, the branch loci of the maps $N\to N_z\approx \C P^2-\{p\}$ and $N_x\to N_G\approx\R^4$ are, respectively, an exotic punctured quadric curve and M\" obius band. It should be possible to draw these with explicit level diagrams in $\R^4$ or its blowup. It should also be possible to describe the whole $\Z_2\oplus\Z_2$-action on $N$ via explicit surfaces in $R_G\approx\R^4$ and perhaps use this to shed some light on the action on the end of $R^*$. Does this action, or the $\Z_2^3$-action of Remark~\ref{reflect}(b), provide new insight on corks? A large $\R^4$ can be constructed in $S^2\times S^2$ rather than $\overline{\C P^2}$ \cite[Section~9.4]{GS} -- or in many other manifolds, cf.~\cite[Theorem~6.6]{steintop}. In particular, a large $\R^4$ can be embedded in Figure~\ref{U}, containing the $S^4$-summand. Can this (or more general examples) be assumed equivariant? What new phenomena result?

\begin{ques} Is there a family of pairwise nonisotopic exotic planes that is naturally parametrized by an interval (perhaps preserving the order defined for Theorem~\ref{order})?
\end{ques}

\noindent The double branched cover $R_x^*$ of an exotic plane constructed for Theorem~\ref{order}(b,c) lies in a pairwise nondiffeomorphic family parametrized by an interval (Section~\ref{R4}). This can be assumed to extend our $\Z_2$-invariant family parametrized by $\Sigma$, provided that the construction uses towers with enough embedded surface stages (as in Freedman--Quinn \cite{FQ}) in place of Casson handles (cf.~\cite[Theorem~3.2]{DF}). However, it isn't clear whether the quotients $R^*_G$ are the standard $\R^4$ for parameter values outside the Cantor set.

\begin{ques} Does every compact surface in $B^4$ generate an uncountable family of topologically isotopic surfaces in $\R^4$ as in Corollary~\ref{slice}? How generally do noncompact surfaces in $\R^4$ lie in such families (countable or uncountable)?
\end{ques}

\noindent Theorem~\ref{2knot}(b) gives uncountable families for certain slice disks, but fails for higher genus surfaces. An infinite end sum $F_\infty$ of standard punctured tori is a candidate for a surface that cannot be changed by summing with exotic planes, cf.~Questions~\ref{stabilize} and below. However, an exotic $F_\infty$ can be made by attaching copies of $2CH$ to $(-\infty,0]\times\R^3$ along an infinite union of Hopf links so that a fixed $\pi$-rotation flips each Casson handle. This $\#_\infty S^2\times S^2$-homeomorph can be chosen to be Stein and hence exotic by the adjunction inequality. Furthermore it, and hence the exotic $F_\infty$ branch locus, comes in uncountably many diffeomorphism types distinguished by the genus function \cite[Theorem~3.5]{MinGen}.

\begin{quess}\label{univ}
How different is topological proper 2-knot theory from the smooth theory modulo exotic planes? (See Question~\ref{modplanes}.) We can ask this for embeddings of $\R^2$ or for higher-genus punctured surfaces, or consider infinite-genus surfaces such as $F_\infty$. Is there a ``universal" plane analogous to (and maybe branch covered by) the Freedman--Taylor universal $\R^4$ \cite{FT}?
\end{quess}

\noindent The defining property of the universal $\R^4$ (that is, the interior of the universal half-space of \cite{FT}) is that end summing with it (at each end) turns homeomorphic 4-manifolds diffeomorphic whenever the corresponding Kirby--Siebenmann uniqueness obstruction vanishes. One might hope for an exotic plane that similarly implies the map of Question~\ref{modplanes} is a bijection, but this is probably too optimistic.

\begin{probs}\label{ginfrange} a) Prove that every positive integer is $g^\infty(P)$ for uncountably many distinct exotic planes $P$ (cf.~Corollary~\ref{ginf} and Scholium~\ref{infEnd}). Prove the same for $\kappa^\infty$ in $(\Z^{\ge0}\cup\{\infty\})\times(\Z^{\ge0}\cup\{\infty\})$.

b) Find exotic planes (or annuli) with $g^\infty\ne\max\{\kappa^\infty_\pm\}$. Find exotic annuli with $g\ne\max\{\kappa_\pm\}$.
\end{probs}

\noindent One approach to (a) would be to apply Theorem~\ref{2knot}(b) (and maybe Remark~\ref{2knotRem}) to the exotic planes of Theorem~\ref{ginfty}. However, it is not clear whether the definiteness hypothesis can be applied, cf.~Questions~\ref{CHbranch}. This may depend on signs of double points, so may work better if one component of $\kappa^\infty$ vanishes. If the surface  in (b) is allowed to be topologically knotted, examples can be constructed from knots in $S^3$. (The annulus made from the figure-eight knot has $g=g^\infty=1$ but $\kappa_\pm=\kappa^\infty_\pm=0$.)

\begin{prob}\label{hardfig} Draw a level diagram describing an exotic plane (or annulus) with $g^\infty>0$.
\end{prob}

Section~\ref{Group} presented exotic planes with large discrete group actions whose nontrivial elements were not pairwise isotopic to the identity, as well as planes with similar inextendible actions of their ends. It is natural to ask what other sorts of actions can occur. Since $\R$-actions (flows) can always be constructed, we restrict to actions with torsion. For comparison, torus knots in $S^3$ admit circle actions, and hence finite cyclic actions whose elements are pairwise isotopic to the identity. These knots can then be coned to PL (or holomorphic) almost-smooth (but not locally flat) embeddings $\R^2\emb\R^4$ with circle actions.

\begin{quess} Are there exotic planes with finite cyclic actions whose elements are pairwise isotopic to the identity? Circle actions? Are all torsion group actions on exotic planes discrete? What about actions on exotic (topologically standard) annuli?
\end{quess}

We conclude with several more general questions. To begin, note that connected summing a pair $(X,F)$ with $S^4$ containing a standardly embedded $T^2$, or $\R P^2$ with normal Euler number $-2$ or $+2$, sums any double branched cover with $S^2\times S^2$, $\C P^2$ or $\overline{\C P^2}$, respectively. Thus, the double branched cover of any exotic plane becomes diffeomorphic to that of the standard plane after sum with infinitely many standard tori along a discrete set. According to \cite[Proposition~5.4]{BG}, any $R$ as in Section~\ref{Branch1} with only positive double points in $CH$ becomes standard after such a sum with copies of $\C P^2$ but not with a sum of copies of $\overline{\C P^2}$. Thus, the corresponding planes remain exotic after infinite connected sums with standard positive projective planes.

\begin{quess}\label{stabilize}
Do these planes become standard after infinite connected sums with negative projective planes? With tori? Do sums with tori make all planes standard?
\end{quess}

\noindent In a compact setting with suitably controlled knot groups, the corresponding question for finite sums with standard tori has been answered affirmatively (Baykur--Sunukjian \cite{BS}). Studying sums with standard copies of $\R P^2$ of a fixed sign may also be interesting in the compact setting.

\begin{quess}
Applying Proposition~\ref{rebuild}(a) to any exotic plane determining an exotic annulus gives an exotic open 2-handle with standard interior and (unlike possibly all Casson handles) a smooth core. Is this useful? What does Proposition~\ref{rebuild}(b) give us?
\end{quess}

\begin{prob} Find a more direct way to distinguish exotic planes. Are there combinatorial invariants?
\end{prob}

\noindent Such invariants could not be determined by underlying topology such as the knot group.

\begin{quess}
Are there exotic planes in $\C^2$ (or in the open unit ball or other Stein structure on $\R^4$) that are holomorphic? Symplectic? Lagrangian? What about almost-smooth (topologically standard) planes (e.g.~Corollary~\ref{annuli}) that are symplectic (or Lagrangian) at nonsingular points?
\end{quess}

\noindent There are Lagrangian disks in $B^4$ \cite{Ch}, and holomorphic embeddings of a complex open disk into $\C^2$ \cite{BKW}, that are topologically knotted, but there is no smoothly knotted {\em algebraic} embedding $\C\emb\C^2$. (See \cite{R} for a topological proof of the latter.) Every exotic plane $P$ is holomorphic in some complex (K\"ahler but not Stein) structure on $\R^4$. (Perturb $P\subset \C^2$ so that it contains some holomorphic open disk $D$, then note that $(\C^2-(P-D),D)$ is diffeomorphic to the pair $(\R^4,P)$.) Our diagrams from Section~\ref{Diagrams} seem hard to make symplectic since they are constructed with many antiparallel sheets. Our other exotic planes cannot be holomorphic, by the maximum modulus principle, since their diagrams require local maxima. Similarly, any holomorphic exotic plane must have double branched cover with vanishing Taylor invariant by Theorem~\ref{infLocMax}.

\begin{quess}\label{uniqueSum}
Are there embeddings of one-ended surfaces in $\R^4$ whose end sum in $\R^4$ depends on a choice of rays? What about sums of the form $(X,F_1)\natural (\R^4,F_2)=(X,F_1\natural F_2)$? Does the answer to the latter depend on whether we distinguish these up to pairwise diffeomorphism or isotopy? Can an end sum as in Proposition~\ref{infSum} fail to commute?
\end{quess}

\noindent By Proposition~\ref{infSum}, any examples would involve infinite genus. Recall that pairwise diffeomorphism implies isotopy for surfaces in $\R^4$. Compare with ray dependence of sums of 4-manifolds, e.g.,~\cite{CG}, \cite{CGH}.

\end{document}